\documentclass[11pt,a4paper,reqno]{amsart}
\usepackage[T1]{fontenc}
\usepackage{amsmath}
\usepackage{amsthm}
\usepackage{amsfonts}
\usepackage{amssymb}
\usepackage{graphicx}
\usepackage{amsbsy}
\usepackage{mathrsfs}
\usepackage{bbm}
\newtheorem{theorem}{Theorem}[section]
\newtheorem{lemma}[theorem]{Lemma}

\newtheorem{proposition}[theorem]{Proposition}

\newtheorem{definition}[theorem]{Definition}
\theoremstyle{remark}
\newtheorem{remark}[theorem]{\it \bf{Remark}\/}

\numberwithin{equation}{section}
\catcode`@=11
\def\section{\@startsection{section}{1}%
  \z@{1.5\linespacing\@plus\linespacing}{.5\linespacing}%
  {\normalfont\bfseries\large\centering}}
\catcode`@=12
\newcommand{\be}{\begin{equation}}
\newcommand{\ee}{\end{equation}}
\newcommand{\bea}{\begin{eqnarray}}
\newcommand{\eea}{\end{eqnarray}}
\newcommand{\bee}{\begin{eqnarray*}}
\newcommand{\eee}{\end{eqnarray*}}

\def\pa{\partial}

\def\NN{\mathbb{N}}
\def\RR{\mathbb{R}}

\def\fref#1{{\rm (\ref{#1})}}

\def\G{{\Gamma}}

\catcode`@=11
\def\supess{\mathop{\operator@font Sup\,ess}}
\catcode`@=12

\def\NN{\mathbb{N}}

\def\RR{\mathbb{R}}

\def\H{{\mathcal H}}

\def\e{\varepsilon}

\def\fref#1{{\rm (\ref{#1})}}

\def\R2+{\RR ^2_+}

\def\lsl{\frac{\lambda_s}{\lambda}}

\def\pa{\partial}

\def\lim{\mathop{\rm lim}}

\def\e{\varepsilon}

\def\log{{\rm log}}

\def\vdl{V_{\lambda}^{(2)}}

\def\lsl{\frac{\lambda_s}{\lambda}}

\def\vul{V_{\lambda}^{(1)}}

\def\eb{\e}
\def\ebo{\e}

\def\qbt{P_{B_1}}

\def\qbtb{P_{B_1}}

\def\S{\Sigma}

\def\pa{\partial}

\def\pa{\partial}
\title[]{Stable blow up dynamics for the critical co-rotational Wave Maps and equivariant Yang-Mills problems}
\author[P. Rapha\"el]{Pierre Rapha\"el}
\address{Institut de Math\'ematiques de Toulouse, Universit\'e Toulouse III, France}
\email{pierre.raphael@math.univ-toulouse.fr}
\author[I. Rodnianski]{Igor Rodnianski}
\address{Mathematics Department, Princeton University, USA}
\email{irod@math.princeton.edu}

\begin{document}
\maketitle

\begin{abstract}
We exhibit stable finite time blow up regimes for the energy critical co-rotational Wave Map with the ${\Bbb S}^2$ target in all homotopy classes and for the  critical equivariant $SO(4)$ Yang-Mills problem. We derive sharp asymptotics on the dynamics at blow up time and prove quantization of the energy focused at the singularity.
\end{abstract}


\section{Introduction}


In this paper, we study the dynamics of two critical problems: the $(2+1)$-dimensional Wave Map and the $(4+1)$-dimensional Yang-Mills equations. These problems admit non trivial static solutions (topological solitons) which have been extensively studied in the literature both from the mathematical and physical point of view, see e.g. \cite{BP},\cite{BPST},\cite{DK},\cite{He},\cite{Mo},\cite{U},\cite{W}. The static solutions for the (WM) are harmonic maps from $\Bbb R^2$ into $\Bbb S^2\subset {\Bbb R}^3$
satisfying the equation
$$
-\Delta \Phi=\Phi |\nabla\Phi|^2
$$
 They are explicit solutions of the $O(3)$ nonlinear $\sigma$-model of isotropic plane ferromagnets. 
 For the (YM) equations a particularly interesting class of static solutions is formed by (anti)self-dual instantons, 
 satisfying the equations 
 $$
 F=\pm *F
 $$
for the curvature $F$ of an $so(4)$-valued connection over ${\Bbb R}^4$. The $4$-dimensional euclidean Yang-Mills theory forms a basis of the Standard Model of particle physics and its special static solutions played an important role
as pseudoparticle models in Quantum Field Theory.

 The geometry of the moduli space of static solutions has been a subject of a thorough investigation, see e.g. 
 \cite{Wa},\cite{ADHM},\cite{DK},\cite{ES}.  In particular, the moduli spaces are incomplete due to  
 the scale invariance property of both problems. This gave rise to a plausible scenario of singularity formation in the corresponding time dependent equation which has been studied heuristically, numerically and very recently from a mathematical point of view, \cite{B2},\cite{IL},\cite{KST1},\cite{KST2},\cite{RS},\cite{MS} and references therein.

The focus of this paper is the investigation of special classes of solutions to the critical $(2+1)$-dimensional (WM) and the critical $(4+1)$-dimensional (YM) describing a {\bf stable} (in a fixed
co-rotational class) and {\bf universal} regime in which an open set of initial data leads to a finite time formation 
of singularities.\\

The Wave Map problem for a map $\Phi:{\Bbb R}^{2+1}\to {\Bbb S}^2\subset {\Bbb R}^3$ is described by a nonlinear 
hyperbolic evolution equation
$$
\pa_t^2 \Phi-\Delta\Phi=\Phi \left (|\nabla\Phi|^2-|\pa_t\Phi|^2\right)
$$
with initial data $\Phi_0:{\Bbb R}^2\to {\Bbb S}^2$ and $\pa_t\Phi|_{t=0}=\Phi_1:{\Bbb R}^2\to T_{\Phi_0} {\Bbb S}^2$.
We will study the problem under an additional assumption of co-rotational symmetry, which can be described as follows.
Parametrizing the target sphere with the Euler angles $\Phi=(\Theta,u)$ we assume that the solution has a special form
$$
\Theta(t,r,\theta)=k\theta,\qquad u(t,r,\theta)=u(t,r)
$$
with an integer constant $k\ge 1$ -- homotopy index of the map $\Phi(t,\cdot): {\Bbb R}^2\to {\Bbb S}^2$. Under such 
symmetry assumption the full wave map system reduces to the one dimensional semilinear wave equation:  
\be
\label{WMsphere}
\partial^2_{t}u-\partial^2_{r}u-\frac{\partial_ru}{r}+k^2\frac{\sin(2u)}{2r^2}=0, \ \ k\geq 1, \ \ (t,r)\in \RR\times\RR_+, \ \ k\in \NN^*.
\ee 
Similarly, the equivariant reduction, given by the ansatz,
 $$
 A_\alpha^{ij}=(\delta^i_\alpha x^j-\delta^j_\alpha x^i) \frac {1-u(t,r)}{r^2},
 $$
 of the $(4+1)$-dimensional Yang-Mills system 
\begin{align*}
&F_{\alpha\beta} =\partial_\alpha A_\beta-\partial_\beta A_\alpha + [A_\alpha,A_\beta],\\
&\pa_\beta F^{\alpha\beta}+[A_\beta,F^{\alpha\beta}]=0,\qquad \alpha,\beta=0,...,3
\end{align*}
for the $so(4)$-valued gauge potential $A_\alpha$ and curvature $F_{\alpha\beta}$,
leads in the semilinear wave equation: 
\be
\label{YMradial}
\partial_{t}^2u-\partial^2_{r}u-\frac{\partial_ru}{r}-\frac{2u(1-u^2)}{r^2}=0,  \ \ (t,r)\in \RR\times\RR_+.
\ee 
The problems \eqref{WMsphere} and \eqref{YMradial} can be unified by an equation of the form
\be
\label{equation}
\left\{\begin{array}{ll}\partial^2_{t}u-\partial^2_{r}u-\frac{\partial_ru}{r}+k^2\frac{f(u)}{r^2}=0,\\
	u_{|t=0}=u_0, \ \ (\partial_tu)_{|t=0}=v_0
	\end{array}\right .
	 \ \ \mbox{with} \ \ f=gg'
\ee 
and 
$$
g(u)=\left\{\begin{array}{ll} 
\sin(u), \ \ k\in \NN^* \ \ \mbox{for} \ \ (WM)\\
\frac 12 (1-u^2),\ \ k=2 \ \ \mbox{for} \ \ (YM).
\end{array}\right .
$$
\eqref{equation} admits a conserved energy quantity
$$
E(u,\partial_t u)=\int_{{\Bbb R}^2} \left ((\pa_t u)^2+ |\pa_r u|^2 + k^2\frac{g^2(u)}{r^2}\right)
$$
which is left invariant by the scaling symmetry $$u_{\lambda}(t,r)=u(\frac{t}{\lambda},\frac{r}{\lambda}), \ \ \lambda>0.$$The minimizers of the energy functional can be explicitly obtained as \be\label{eq:WQ}
Q(r)=2\tan^{-1} (r^k) \ \ \mbox{for}\ \ (WM), \ \ Q(r)=\frac {1-r^2}{1+r^2}\ \ \mbox{for} \ \ (YM),
\ee
and their rescalings which constitute the moduli space of stationary solutions in the given corotational homotopy class.\\

A sufficient condition for the global existence of solutions to \fref{equation} was established in the pioneering works by Christodoulou-Tahvildar-Zadeh \cite{CT}, Shatah-Tahvildar-Zadeh \cite{ST}, Struwe \cite{Struwe}. It can be described as folllows: for smooth initial data $(u_0,v_0)$ with $E(u_0,v_0)<E(Q)$, the corresponding solution to \fref{equation} is global in time and decays to zero, see also \cite{CKM}. More precisely, it was shown that if a singularity is formed at time $T<+\infty$, then energy must concentrate at $r=0$ and $t=T$. This concentration must happen strictly inside the backward light cone from $(T,0)$, that is if the scale of concentration is $\lambda(t)$, then 
\be
\label{noteselfsimilar}
\frac{\lambda(t)}{T-t}\to 0 \ \ \mbox{as} \ \ t\to T.
\ee
Note that the case $\lambda(t)=T-t$ would correspond to self-similar blow up which is therefore ruled out. Finally, a universal blow up profile may be extracted in rescaled variables, at least on a sequence of times:
\be
\label{profile}
u(t_n,\lambda(t_n)r)\to Q \ \ \mbox{in} \ \ H^1_{loc} \ \ \mbox{as}  \ \ n\to +\infty.
\ee
These results hold for more general targets for (WM) with $Q$ being a non trivial harmonic map. In particular, this implies the global existence and propagation of regularity for the corotational (WM) problem with targets admitting no non trivial harmonic map from $\Bbb R^2$. Very recently, in a series of works \cite{Taor},\cite{Tao27},\cite{St-Ta},\cite{St-Ta1},\cite{KSS}, this result has been remarkably extended to the full (WM) problem without the assumpion of corotational symmetry, hence completing the program developed in \cite{KS},\cite{KlMa},\cite{T},\cite{Tao},\cite{KM}.\\

These works leave open the question of existence and description of singularity formation in the presence of non trivial harmonic maps, or the instanton for the (YM). This long standing question has first been addressed through some numerical and heuristic works in \cite{B1},\cite{B2},\cite{IL},\cite{PZ},\cite{SO}. In particular, the 
blow up rates of the 
concentration scale 
\begin{align*}
&\lambda(t)\sim B\frac{T-t}{|\log(T-t)|^{\frac{1}{2}}} \ \ \ \ \mbox{for (YM)},\\
&\lambda(t)\sim A (T^*-t) e^{-\sqrt {|\ln (T^*-t)|}}\ \ \ \ \mbox{for (WM) with}\ \ k=1
\end{align*}
with specific constants $A,B$ have been predicted in a very interesting 
work \cite{B2} and, a very recent, \cite{SO} respectively.

Instability of $Q$ for the $k=1$ (WM) and (YM) was shown by C\^ote in \cite{Cote}. 
A rigorous evidence of singularity formation has been recently given via two different approaches. In \cite{RS}, Rodnianski and Sterbenz study the (WM) system for a large homotopy number $k\geq 4$ and prove the existence of {\it stable} finite time blow up dynamics. These solutions behave near blow up time according to the decomposition
\be
\label{strctureq}
u(t,r)=(Q+\e)(t,\frac{r}{\lambda(t)}) \ \ \mbox{with} \ \ \|\e,\partial_t\e\|_{\dot{H}^1\times L^2}\ll1 
\ee
 with a lower bound on the concentration:
\be
\label{lowerbound}
\lambda(t)\to 0\ \ \mbox{as} \ \ t\to T \ \ \mbox{with} \ \ \lambda(t)\geq \frac{T-t}{|\log(T-t)|^{\frac{1}{4}}}.
\ee
In \cite{KST1}, \cite{KST2}, Krieger, Schlag and Tataru consider respectively the (WM) system for $k=1$ and the (YM) equation and exhibit finite time blow up solutions which satisfy \fref{strctureq} with 
\begin{equation}\label{eq:ratesKST}
\begin{split}
&\lambda(t)=(T-t)^{\nu} \ \ \mbox{for (WM) with} \ \ \ k=1,\\
&\lambda(t)=(T-t) |\log (T-t)|^{-\nu} \ \ \mbox{for (YM)} 
\end{split}
\end{equation} 
for any chosen $\nu>\frac 32$.
This continuum of blow up solutions is believed to be non-generic.


\subsection{Statement of the result}


In this paper, we give a complete description of a stable singularity formation for the (WM) for all homotopy classes and the (YM) in the presence of corotational/equivariant symmetry near the harmonic map/instanton. The following theorem is the main result of this paper.

\begin{theorem}[Stable blow up dynamics of co-rotational Wave Maps and Yang-Mills]
\label{mainthm}
Let $k\geq 1$. Let $\H_a^{2}$ denote the affine Sobolev space \fref{defhatwo}.There exists a set ${\mathcal{O}}$ of initial data which is open in $\mathcal H_a^{2}$ and a universal constant $c_k>0$ such that the following holds true. For all $(u_0,v_0)\in{\mathcal{O}}$, the corresponding solution to \fref{equation} blows up in finite time $0<T=T(u_0,v_0)<+\infty$ according to the following universal scenario:\\
{\em (i) Sharp description of the blow up speed}: There exists $\lambda(t)\in \mathcal C^1([0,T),\RR^*_+)$ such that:
\be
\label{convustarbis}
 u(t,\lambda(t)y)\to Q\ \ \mbox{in} \ \ H^1_{r,loc} \ \ \mbox{as}  \ \ t\to T
\ee
with the following asymptotics:
\be
\label{universallawkgeq}
\lambda(t)=c_k(1+o(1))\frac{T-t}{|\log(T-t)|^{\frac{1}{2k-2}}} \ \ \mbox{as} \ \ t\to T\ \ \mbox{for} \ \ k\geq 2,\\
\ee
\be
\label{universallawkgeq2}
\lambda(t)=(T-t)e^{-\sqrt{|\log(T-t)|}+O(1)} \ \ \mbox{as} \ \ t\to T \ \ \mbox{for} \ \ k=1.
\ee
$$
\lambda(t)=c_2(1+o(1))\frac{T-t}{|\log(T-t)|^{\frac{1}{2}}} \ \ \mbox{as} \ \ t\to T\ \ \mbox{for} \ \ \mbox{(YM)}.\\
$$
Moreover, 
$$
b(t):=-\lambda_t(t)=\frac{\lambda(t)}{T-t}(1+o(1))\to 0 \ \ \mbox{as} \ \ t\to T 
$$
{\em (ii) Quantization of the focused energy}: Let $\mathcal H$ be the energy space \fref{defenergyspcar},  then there exist $(u^*,v^*)\in \mathcal H$ such that the following holds true. Pick a smooth cut off function $\chi$ with $\chi(y)=1$ for $y\leq 1$ and let $\chi_{\frac{1}{b(t)}}(y)=\chi(b(t)y)$, then:
\be
\label{convustarb}
\lim_{t\to T}\left\|u(t,r)-\left(\chi_{\frac{1}{b(t)}}Q\right)(\frac{r}{\lambda(t)})-u^*,\partial_t\left[u(t,r)- \left(\chi_{\frac{1}{b(t)}}Q\right)(\frac{r}{\lambda(t)})-v^*\right]\right\|_{\mathcal H}=0.
\ee
Moreover, there holds the quantization of the focused energy:
\be
\label{qunitoief}
E_0=E(u,\partial_tu)=E(Q,0)+E(u^*,v^*).
\ee
\end{theorem}

This theorem thus gives a complete description of a stable blow up regime for all homotopy numbers $k\geq 1$
and the (YM) problem, which can be formally compared with the $k=2$ case of (WM). Stable blow up solutions in $\mathcal O$ decompose into a singular part with a universal structure and a regular part which has a strong limit in the scale invariant space. Moreover, the amount of energy which is focused by the singular part is a universal quantum independent of the Cauchy data.\\

{\it Comments on the result}\\

{\it 1. $k=1$ case}: In the $k\geq 2$ and (YM) case, the blow up speed $\lambda(t)$ is to leading order universal ie independent of initial data. On the contrary, in the $k=1$ case, the presence of the $e^{O(1)}$ factor in the blow up speed seems to suggest that the law is not entirely universal and has an additional degree of freedom depending on the initial data. In general, the analysis of the $k=1$ and to some extent $k=2$ problems is more involved. In particular for $k=1$, the instability direction $r\pa_r Q$ driving the singularity formation misses the $L^2$ space logarithmically. This anomalous logarithmic growth is fundamental in determining the blow up rate. On the other hand, this anomaly also adversely  influences the size of the radiation term which implies that there is only a logarithmic difference between the leading order and the radiative corrections. This requires a very precise analysis and a careful track of all logarithmic gains and losses. In the case of larger k, these gains are polynomial and hence the effect of radiation is more easily decoupled from the leading order behavior. In this paper, we adopted a universal approach which simultaneously treats all cases.\\

{\it 2. $k=2$ case}: The analysis of the $k=2$ case for the (WM) problem is almost identical to that required
to treat the (YM) equations. In what follows we will subsume the (YM) problem into the $k=2$ regime of (WM),
making appropriate modifications, caused by a small difference in the structure of the nonlinearities in the two equations, in necessary 
places. \\

{\it 3. Regularity of initial data}: The open set ${\mathcal O}$ of initial data described in the theorem 
contains an open subset of $C^\infty$ data coinciding with $Q$ for all sufficiently large values of $r\ge R$. 
As a consequence,  the main result of the paper in particular describes singularity formation in solutions arising
from {\it smooth} initial data. This should be compared with the results in \cite{KST1},\cite{KST2} where solutions,
specifically constructed to exhibit the blow up behavior given by the rates in \eqref{eq:ratesKST}, lead to the 
initial data of limited regularity dependent on the value of the parameter $\nu$ and degenerating as $\nu\to \frac 32$.\\

{\it 4. Comparison with the $L^2$ critical (NLS)}: This theorem as stated can be compared to the description of the stable blow up regime for the $L^2$ critical (NLS) $$iu_t+\Delta u+u|u|^{\frac{4}{N}}=0, \ \ (t,x)\in [0,T)\times\RR^N, \ \ N\geq 1,$$ see Perelman \cite{Perelman} and the series of papers by Merle and Rapha\"el \cite{MR1}, \cite{MR2}, \cite{R1}, \cite{MR3}, \cite{MR4}, \cite{MR5}. There is a conceptual analogy between the mechanisms of a stable regime singularity formation for the critical (WM) and (YM) problems and the $L^2$ critical (NLS) problem. 
For the latter problem the sharp blow up speed and the quantization of the blow up mass is derived in \cite{MR3}, \cite{MR4}, \cite{MR5}. The concentration 
occurs on an almost self-similar scale 
$$\lambda(t)\sim \sqrt{\frac{2\pi(T-t)}{\log|\log(T-t)|}} \ \ \mbox{as} \ \ t\to T.$$ 
In both (WM), (YM) and the $L^2$ critical (NLS) problems self-similar singularity formation is corrected by subtle interactions
between the ground state and the radiation parts of the solution. The precise nature of these interactions, affecting the 
blow up laws, depends in a very sensitive fashion on the asymptotic behavior of the ground state: polynomially decaying 
to the final value for the (WM) and (YM) and exponentially decaying for the (NLS), see also \cite{LMR} for related considerations. This dependence becomes particularly apparent upon examining
the blow up rates for the (WM) problem in different homotopy classes parametrized by $k$. For $k=1$ the harmonic 
map approaches its constant value at infinity at the slowest rate, which leads to the strongest deviation of the corresponding blow up 
rate from the self-similar law.\\

{\it 5. Least energy blow up solutions}: The importance of the $k=1$ case for the (WM) problem is due to the fact that the $k=1$ ground state is the least energy harmonic map: $$E(Q)=4\pi k.$$ A closer investigation of the structure of $Q$ for $k\geq 2$ shows that this configuration corresponds to the accumulation of $k$ topological charges at the origin $r=0$. 
For the full, non-symmetric problem, we expect such configurations to split under a generic perturbation into a collection of $k=1$
harmonic maps and lead to a different dynamics driven by the evolution of each of the $k=1$ ground states and their interaction.

From this point of view the stability of the least energy $k=1$ configuration under generic non-symmetric perturbations is an important remaining problem.


\subsection{Functional spaces and notations}
\newcommand\ep{\epsilon}


For a pair of functions $(\ep(y),\sigma(y))$, we let 
\be
\label{defenergyspcar}
\|(\ep,\sigma)\|_{\mathcal H}^2=\int\left[\sigma^2+(\partial_y \ep)^2+\frac{\ep^2}{y^2}\right]
\ee
define the energy space. We also define the $\mathcal H^2$ Sobolev space with norm:
\be
\label{defhsoboloevkgeqtwo}
\|(\ep,\sigma)\|^2_{\H^2}=\|(\ep,\sigma)\|_{\mathcal H}^2+\int\left[(\partial_y^{2}\ep)^2+\frac{(\partial_y \ep)^2}{y^2}+(\partial_y \sigma)^2+\frac{\sigma^2}{y^2}\right] \ \ \mbox{for}  \ \ k\geq 2,
\ee
\be
\label{defhsoboloevkgeqone}
\nonumber \|(\ep,\sigma)\|^2_{\H^{2}} =  \|(\ep,\sigma)\|_{\mathcal H}^2+\int\left[(\partial_y^2 \ep)^2+(\partial_y\sigma)^2+\frac{\sigma^2}{y^2}\right]
 +  \int_{y\leq 1}\frac{1}{y^2}\left(\partial_y\ep-\frac{\ep}{y}\right)^2\ \mbox{for} \ \  k=1.
\ee
and a related norm, in the relevant case of $\sigma=\pa_t\epsilon$, 
\be\label{eq:th-norm}
\|\ep\|_{\tilde{\mathcal H}}^2=|H\ep|^2_{L^2}+\|(\pa_t\ep,0)\|_{H}^2
\ee
where $H$ is the linearized Hamiltonian defined in \fref{defhamiltoninalineaire}. Observe that \fref{defenergyspcar}, \fref{defhsoboloevkgeqtwo}, \fref{defhsoboloevkgeqone} and \eqref{eq:th-norm}
require vanishing of 
$(\ep,\sigma)$ at the origin.\\
We then define an affine space 
\be
\label{defhatwo}
\mathcal H^2_a=\mathcal H^2+Q.
\ee
\vskip 1pc
\noindent
We denote $$(f,g)=\int fg=\int_{0}^{+\infty}f(r)g(r)rdr$$ the $L^2(\RR^2)$ radial inner product. We define the differential operators: 
\be
\label{defdl}
\Lambda f=y\cdot\nabla f \ \ (\mbox{$\dot{H}^1$ scaling}), \ \ Df=f+y\cdot\nabla f\ \  (\mbox{$L^2$ scaling})\ee and observe the integration by parts formula:
\be
\label{adjoinctionfrimula}
(Df,g)=-(f,Dg), \ \ (\Lambda f,g)+(\Lambda g,f)=-2(f,g).
\ee
Given $f$ and $\lambda>0$, we shall denote: $$f_{\lambda}(t,r)=f(t,\frac{r}{\lambda})=f(t,y),$$ 
and the rescaled variable will always be denoted by $$y=\frac{r}{\lambda}.$$ For a time-dependent scaling parameter $\lambda(t)$ we define the rescaled time
$$
s=\int_0^t \frac {d\tau}{\lambda^2(\tau)}
$$
We let $\chi$ be a smooth positive radial cut off function $\chi(r)=1$ for $r\leq 1$ and $\chi(r)=0$ for $r\geq 2$. For a given parameter $B>0$, we let 
\be
\label{defchib}
\chi_B(r)=\chi(\frac{r}{B}).
\ee
Given $b>0$, we set
\be
\label{defbnot}
B_0=\frac 1{b \sqrt{3\int y\chi(y) dy}},\quad B_c=\frac{2}{b}, \quad B_1=\frac{|\log b|}{b}.
\ee


\subsection{Strategy of the proof}


We now briefly sketch the main ingredients of the proof of Theorem \ref{mainthm}.\\

{\bf Step 1} The family of approximate self similar profiles.\\

We start with the construction of suitable approximate self-similar solutions in the fashion related to 
the approach developed in \cite{MR2}, \cite{MR4}. Following the scaling invariance of \fref{equation}, we pass to the the self-similar variables and look for a one parameter family of self similar solutions  dependent on a small parameter $b>0$: $$u(t,r)=Q_b(y), \ \ y=\frac{r}{\lambda(t)}, \ \ \lambda(t)=b(T-t).$$ This transformation maps \fref{equation} into the self-similar equation:
\be
\label{eqselfsimilaire}
-\Delta v+b^2D\Lambda Q_b+k^2\frac{f(v)}{y^2}=0
\ee
where the differential operators $\Lambda,D$ are given by \fref{defdl}. A well known class of exact solutions are given by the explicit profiles:
$$Q_b(r)=Q\left(\frac{r}{1+\sqrt{1-b^2r^2}}\right), \ \ r\leq \frac{1}{b}.$$ These solutions were used by 
C\^ote to prove that $Q$ is unstable for both (WM) and (YM),
\cite{Cote}. A direct inspection however reveals that these have infinite energy due to a logarithmic divergence  on the backward light cone $$r=(T-t) \ \ \mbox{equivalently} \ \ y=\frac{1}{b}.$$ This situation is exactly the same for the $L^2$ critical (NLS), \cite{MR2}, and reveals the critical nature of the problem.
Note that in higher dimensions finite energy self-similar solutions can be shown to exist thus providing 
explicit blow up solutions to the Wave Map and Yang-Mills equations, \cite{Sh}, \cite{CST}.\\
In order to find {\it finite energy} suitable approximate solutions to \fref{eqselfsimilaire} in the vicinity of the ground state $Q$ we construct to a formal expansion $$Q_b=Q+\Sigma_{i=1}^pb^{2i}T_i.$$ Substituting the ansatz into the self-similar equation \fref{eqselfsimilaire}, we get at the order $b^{2i}$ an equation of the form: 
\be
\label{vhohollo}
HT_{i}=F_i
\ee 
where 
\be
\label{defhamiltoninalineaire}
H=-\Delta +k^2\frac{f'(Q)}{y^2}
\ee
is obtained by linearizing \eqref{eqselfsimilaire} on $Q$ (setting $b=0$) and $F_i$ is a nonlinear expression. The solvability of \fref{vhohollo} requires that $F_i$ is orthogonal  to the kernel of $H$, 
which is explicit by the variational characterization of $Q$: 
\be
\label{kernel}
Ker(H)=\mbox{span}(\Lambda Q)
\ee 
and hence the orthogonality condition: 
\be
\label{cnooenoe}
(F_i,\Lambda Q)=0.
\ee
While the condition \fref{cnooenoe} seems at first hand to be a very nonlinear condition, it can be easily checked to hold due to the specific algebra of the $H^1$ critical problem and its connection to the Pohozaev identity. In fact, if $Q_b^{(p)}=Q+\Sigma_{i=1}^pb^{2i}T_i$ is the expansion of the profile to the order $p$, then \fref{cnooenoe} holds as long as the Pohozaev computation is valid:
\bea
\label{oeoe}
 & & \left(-\Delta Q_b^{(p)}+b^2D\Lambda Q_b^{(p)}+k^2\frac{f(Q_b^{(p)})}{y^2},D\Lambda Q_b^{(p)}\right)\\
 & = & \lim_{R\to+\infty}\left[\frac{b^2}{2}|r\Lambda Q_b^{(p)}(R)|^2+\frac{k^2}{2}|g(Q_b^{(p)}(R)|^2\right]
= 0,\label{eq:Poh}
\eea
see step 2 of the proof of Proposition \ref{propqb}, section \ref{sectionqb}. By a direct computation, $F_1\sim D\Lambda Q\sim \frac{1}{y^k} \ \ \mbox{as} \ \ y\to +\infty$ and at each step, the inversion of \fref{vhohollo} dampens the decay of $T_{i+1}$ at infinity by an extra $y^2$ factor, and hence the validity of \eqref{eq:Poh} comes under question after $p$ steps, for as $y\to \infty$:
\be
\label{coeobneo}
T_p(y)\sim \frac{c_k}{y} \ \ \mbox{for} \ \ p=\frac{k-1}{2},  \ \ k\ \ \mbox{odd},
\ee
\be
\label{cnoneonee}
T_p(y)\sim c_k\ \ \mbox{for} \ \ p=\frac{k}{2},  \ \ k\ \ \mbox{even}.
\ee
In fact \fref{coeobneo}, \fref{cnoneonee} will result in a {\it universal nontrivial} flux type contribution to \fref{oeoe}. Moreover, $T_p$ is the 
first term which gives an infinite contribution to the energy of the approximate self-similar profile $Q^{(p)}_b(\frac r{\lambda(t)})$.  $T_p$ is the {\it radiation} term which becomes dominant in the region $y\ge \frac 1{b}$ -- exterior to the backward light cone from a singularity at the point 
$(T,0)$.  We therefore stop the asymptotic expansion at $p$\footnote{We will in fact also need the next term $T_{p+1}$ in the expansion. Its construction will be made possible thanks to a subtle cancellation, see step 4 of the proof of Proposition \ref{propqb}} and localize constructed profiles by connecting $Q_b$ to the constant $a=Q(+\infty)$, which is also an exact self-similar solution:
\be
\label{defbniheohe}
P_{B_1}=\chi_{B_1}Q_b+(1-\chi_{B_1})a, \ \ B_1= \frac{|\log b|}{b}>>\frac{1}{b}
\ee
where $\chi_{B_1}=1$ for $y\leq B_1$, $\chi_{B_1}=0$ for $y\geq 2B_1$. $P_{B_1}$ satisfies an approximate self-similar equation of the form: 
\be
\label{esciheihee}
-\Delta P_{B_1}+b^2D\Lambda P_{B_1}+k^2\frac{f(P_{B_1})}{y^2}=\Psi_{B_1}
\ee 
where $\Psi_{B_1}$ is very small inside the light cone $y\leq \frac{1}{b}$ but encodes a slow decay near $B_1$ induced by the cut off function and the radiative behavior of $T_p$ at infinity.\\

{\bf Step 2} The $H^2$ type bound.\\

Let now $u(t,r)$ be the solution to \fref{equation} for a suitably chosen initial data close enough to $Q$. Given the profile $P_{B_1}$, we introduce, with the help of the standard modulation theory, a decomposition of the wave: 
$$
u(t,r)=P_{B_1(t)}(\frac r{\lambda(t)}) + w(t,r)
$$
or alternatively
$$u(t,r)=(P_{B_1(t)}+\e)(s,y), \ \ y=\frac{r}{\lambda(t)}, \ \ \frac{ds}{dt}=\frac{1}{\lambda}$$ with $B_1$ given by \fref{defbniheohe} and where we have {\bf set} the relation 
\be
\label{defbbbb}
b(s)=-\lsl=-\lambda_t.
\ee 
The decomposition is complemented by the orthogonality condition\footnote{The actual orthogonality condition is defined with respect to a cut-off version of $\Lambda Q$.} 
$$\forall s>0, \ \  (\e(s),\Lambda Q)=0$$ as is natural from \fref{kernel}. Our first main claim is the derivation of a {\it pointwise in time bound } on $\e$ 
\be
\label{cbeoeooe}
\|\e\|_{\tilde{H}}\lesssim b^{k+1}
\ee 
in a certain weighted Sobolev space $\tilde{\mathcal H}$. The norm in the space $\tilde{\mathcal H}$ is given by the 
expression
\be\label{eq:tH}
\|\epsilon\|^2_{\tilde{\mathcal H}} = |H \epsilon|_{L^2}^2+\|(\partial_t\epsilon,0)\|_{\mathcal H}^2. 
\ee
and is based on the linear Hamiltonian $H$ associated with the ground state $Q$, see \eqref{eq:th-norm}.
We note in passing that, after adding the norm $\|(\ep,\pa_t\ep)\|_{\mathcal H}^2$, for $k\ge 2$ this norm is equivalent to the $\mathcal H^2$ norm introduced in \eqref{defhsoboloevkgeqone}. There are however subtle differences in the corresponding norms in the
case $k=1$, connected with the behavior for $y\ge 1$.
\vskip 1pc
\noindent
Bounds related to \eqref{cbeoeooe} but for a weaker norm than $\tilde{\mathcal H}$ and with $b^{k+1}$ replaced by
$b^4$ were derived in \cite{RS} for higher homotopy classes $k\ge 4$. They were a consequence of the proof of energy and Morawetz type estimates for the corresponding nonlinear problem satisfied by $w$. The linear part of the equation 
for $w$ is given by the expression
$$
\pa_t^2 w+ H_\lambda w
$$
with the Hamiltonian
\be
\label{defhinitiale}
H_\lambda=-\Delta + \frac {f'(Q_{\lambda})}{r^2}
\ee
Special variational nature of $Q$, discovered in \cite{BP}, provides an important factorization property 
for $H_\lambda$:
\be\label{eq:factor}
H_\lambda= A^*_\lambda A_\lambda, \qquad A_\lambda=-\pa_r + k \frac {g'(Q_{\lambda})}{r}. 
\ee
It arises as a consequence of the fact that\footnote{We restrict this discussion to the (WM) case. Similar
considerations also apply to the (YM) problem, \cite{Bo}} $Q$ represents the co-rotational global minimum of energy $V[\Phi]$
in a given 
topological class of maps $\Phi: {\Bbb R}^2\to {\Bbb S}^2$ of degree $k$.
$$
    V[\Phi]\ =\ \frac{1}{2}
    \, \int_{\mathbb{R}^2} \left (\nabla_x\Phi\cdot\nabla_x\Phi\right) \ dx,
$$
which can be factorized
using the notation $\epsilon_{ij}$ for the antisymmetric tensor on
two indices, as follows:
\begin{align}
  \begin{split}
    V[\Phi]\ &= \ \frac{1}{4}\,
    \int_{\mathbb{R}^2}\ \left[(\pa_i\Phi\pm \epsilon_{i}^{\ \, j}\Phi\times\pa_j\Phi)
    \cdot (\pa^i\Phi\pm \epsilon^{ij}\Phi\times\pa_j\Phi)\right]\  dx\ \\
    &\ \ \ \ \ \ \ \ \ \ \ \ \ \ \ \ \ \ \ \ \ \ \ \ \ \ \ \ \ \
    \ \ \ \ \ \  \ \ \ \pm \
    \frac{1}{2}\, \int_{\mathbb{R}^2} \epsilon^{ij} \Phi\cdot (\pa_i \Phi\times
    \pa_j\Phi)\ dx \ ,
  \end{split}\label{general_lines}\\
    &=\ \frac{1}{4}\, \int_{\mathbb{R}^2}\
     \left[(\pa_i\Phi\pm \epsilon_{i}^{\ \, j}\Phi\times\pa_j\Phi)
    \cdot (\pa^i\Phi\pm \epsilon^{ij}\Phi\times\pa_j\Phi)\right]\  dx
     \ \pm \ 4\pi k \ \notag
\end{align}
from which it is immediate that an absolute minimum
of the energy functional $V[\Phi]$ in a given topological sector $k$ must
be a solution of the equation:
\begin{equation}
    \pa_i \Phi \pm \epsilon_{i}^{\ \, j} \Phi\times
    \pa_j \Phi \ = \ 0 \ . \label{eq:Bog}
\end{equation}
The ground state $Q$ is precisely the representation of the unique co-rotational solution of 
\eqref{eq:Bog}.\\

\noindent
In \cite{RS} factorization \eqref{eq:factor} gave the basis for the $H^2$ and Morawetz type bounds for $w$, obtained by conjugating
the problem for $w$ with the help of the operator $A_\lambda$, so that 
$$
A_\lambda H_\lambda w=\tilde H_\lambda (A_\lambda w)
$$
with $\tilde H_\lambda= A_\lambda A^*_\lambda$, and exploiting the space-time repulsive properties
of $\tilde H_\lambda$ to derive the energy and Morawetz estimates for $A_\lambda w$.  Simultaneous use of pointwise in time energy bounds and space-time Morawetz estimates however runs into 
difficulties in the cases $k=1, 2$, which become seemingly insurmountable for $k=1$.\\
 
 \noindent
We propose here a new approach, still  based on the factorization of $H_\lambda$, yet relying 
{\bf only} on the appropriate {\it energy estimates} for the associated Hamiltonian 
$\tilde H_\lambda$, which retains its repulsive properties even in the most difficult cases of $k=1,2$.
We note that $\|\ep\|_{\tilde H}$ norm introduced above can be conveniently written in the form 
$$
\|\ep\|_{\tilde H}^2= \lambda^2 (\tilde H_\lambda A_\lambda w,A_\lambda w) + \lambda^2\|(\pa_tw,0\|_{H}^2.
$$
\\
One difficulty will be that the bound \fref{cbeoeooe} {\it is not sufficient} to derive the sharp blow up speed. The size $b^{k+1}$ in the RHS of \fref{cbeoeooe} is sharp and is induced by  a very slowly decaying term in $\Psi_{B_1}$ in \fref{esciheihee}, which arises from the localization of the profile $Q_b$. Such terms however are {\it localized} on $y\sim B_1>>\frac{1}{b}$ far away from the backward light cone with the vertex at the singularity. 
Another crucial new feature of our analysis here is a use of {\it localized} energy identities. 
It is based on the idea of writing the energy identity in the region bounded by the initial hypersurface
$t=0$ and the hypersurface 
$$
r=2\frac {\lambda(t)}{b(t)},\quad {\text{equivalently}}\quad y=\frac 2{b(t)}
$$
which, under the bootstrap blow up assumptions, is complete (the point $r=0$ is reached at the 
blow up time) and space-like. Such an energy identity effectively restricts 
the error term $\Psi_{B_1}$ to the region $y\le 2/b$, where it is better behaved, and leads to an 
improved bound: 
\be
\label{cbeoeooebis}
\|\e\|_{\tilde{H}(y\leq \frac{2}{b})}\lesssim \frac{b^{k+1}}{|\log b|},
\ee 
see Proposition \ref{propinside} in section \ref{sectionhtwo}. Note that the logarithmic gain from \fref{cbeoeooe} to \fref{cbeoeooebis} is typical of the $k=1$ case and can be turned to a polynomial gain for $k\geq 2$.\\

{\bf Step 3} The flux computation and the derivation of the sharp law.\\

The pointwise bounds \fref{cbeoeooe}, \fref{cbeoeooebis} are specific to the {\it almost self-similar regime} we are describing. They are 
derived by a bootstrap argument, which incidentally requires {\it only} an upper bound\footnote{Such an upper bound is already sufficient
to conclude the finite time blow up and establish a lower bound on the concentration scale $\lambda(t)$.} on $|b_s|$, see Lemma \ref{roughboundpointw}. To derive the precise law for $b$ we examine the equation for $\e$, which has the following approximate
form: 
\be
\label{cnoeoer}
\partial_s^2\e+H_{B_1}\e=-b_s\Lambda P_{B_1}+\Psi_{B_1}+\mbox{L.O.T.}
\ee where $H_{B_1}=-\Delta +k^2\frac{f'(P_{B_1})}{y^2} $. We consider an almost self-similar solution 
$P_{B_0}$ localized on the scale $B_0=\frac{c}{b}$ with a specific constant $0<c<1$ defined in \eqref{defbnot} and project this equation onto $\Lambda P_{B_0}$, which is almost in the null space of $H_{B_1}$. The result is the identity of the form: 
\be
\label{cnbioebeibep}
b_s|\Lambda P_{B_0}|_{L^2}^2=(\Psi_{B_1},\Lambda P_{B_0})+O(b^{k-1}\|\e\|_{\tilde{H}(y\leq \frac{2}{b})}).
\ee
The first term in the above RHS yields the leading order flux and tracks the nontrivial contribution of $T_p$ to the Pohozaev integration \fref{oeoe}:
$$(\Psi_{B_1},\Lambda P_{B_0})=-c_kb^{2k}(1+o(1))$$ for some universal constant $c_k$. This computation can be thought of as related to the derivation of the log-log law in \cite{MR4}. The 
$\epsilon$-term in \fref{cnbioebeibep} is treated with the help of \fref{cbeoeooebis}, observe that \fref{cbeoeooe} alone would not have been enough: $$O(b^{k-1}\|\e\|_{\tilde{H}(y\leq \frac{2}{b})})=o(b^{2k}).$$ Finally, from the behavior $$\Lambda Q\sim \frac{1}{y^k}\ \ \mbox{as} \ \ y\to +\infty$$ and $P_{B_0}\sim Q$ for $b$ small, there holds:
$$|\Lambda P_{B_0}|_{L^2}^2\sim \left\{\begin{array}{ll} c_k \ \ \mbox{for} \ \ k\geq 2\\
								            c_1 |\log b| \ \ \mbox{for} \ \ k=1
								            \end{array} \right .
								            $$ for some universal constant $c_k>0$. We hence get the following system of ODE's for the scaling law:
$$\frac{ds}{dt}=\frac{1}{\lambda}, \ \ b=-\lsl, \ \ b_s=-\left\{\begin{array}{ll} c_k(1+o(1))b^{2k}\ \ \mbox{for} \ \ k\geq 2\\
(1+o(1))\frac{b^2}{2|\log b|}\ \ \mbox{for} \ \ k=1
\end{array} \right .
$$							     							     
Its integration yields -- for the class of initial data under consideration -- the existence of $T<+\infty$ such that $\lambda(T)=0$ with the laws \fref{universallawkgeq}, \fref{universallawkgeq2} near $T$, thus concluding the proof of the sharp asymptotics \fref{universallawkgeq}, \fref{universallawkgeq2}. The non-concentration of the excess of energy \fref{convustarb}, \fref{qunitoief} now follows from the dispersive bounds obtained on the solution, hence concluding the proof of Theorem \ref{mainthm}.\\

This paper is organized as follows. In section \ref{sectionlineaire}, we recall some well known facts about the structure of the linear Hamiltonian $H$ close to $Q$ and the orbital stability bounds. In section \ref{sectiontwo}, we construct the approximate self similar profiles $Q_b$ with sharp estimates on their behavior, Proposition \ref{propqb} and Proposition \ref{lemmapsibtilde}. In section \ref{sectionthree}, we explicitly describe the set of initial data of Theorem \ref{mainthm}, Definition \ref{defoinitial}, and set up the bootstrap argument, Proposition \ref{bootstrap}, which proof relies on a rough bound on the blow up speed, Lemma \ref{lemmainitialdata}, and global and local $H^2$ bounds, Lemma \ref{propinside}. In section \ref{sectionfour}, we derive the sharp blow up speed from the obtained energy bounds and the flux computation, Proposition \ref{lemmaalgebra}, and this allows us to conclude the proof of Theorem \ref{mainthm}.\\
			     
{\bf Aknowldegments} This work was partly done while P.R. was visiting Princeton University and I.R. the Institut de Mathematiques de Toulouse, and both authors would like to thank these institutions for their hospitality. 
The authors also wish to acknowledge discussions with J. Sterbenz concerning early stages of this work.
P.R. is supported by the ANR Jeunes Chercheurs SWAP. I.R. is supported by the NSF grant DMS-0702270.						        								            				         	

\section{Ground state and the associated linear Hamiltonian}
\label{sectionlineaire}


The problem
\be\label{eq:nonl}
\pa_t^2 u -\pa_r^2 u-\frac 1r \pa_r u + k^2 \frac {f(u)}{r^2}=0,\qquad f=gg'
\ee
admits a special stationary solution $Q(r)$, and its dilates $Q_\lambda(r)=Q(r/\lambda)$, 
characterized as the global minimum of the corresponding 
energy functional
\begin{align}
E(u,\pa_t u) &=\int \left ((\pa_t u)^2 + (\pa_r u)^2 + k^2 \frac {g^2(u)}{r^2}\right)\notag\\&=
\int \left ((\pa_t u)^2 + (\pa_r u-k \frac {g(u)}r)^2\right)+2k G(u(r))|_{r=0}^{r=\infty},\label{eq:fac-B}
\end{align}
where $G(u)=\int_0^u g(u) du$. In view of such factorization of energy, $Q$ can be found as a solution of
the ODE
$$
r\pa_r Q=k g(Q),
$$
or alternatively
\be\label{formulakey}
\Lambda Q=kg(Q)
\ee
For the (WM) problem the function $g(u)=\sin u$ and for the (YM) equation $g(u)=\frac 12 (1-u^2)$. Therefore, 
$$Q(r)=2\tan^{-1}(r^k),\qquad Q(r)=\frac {1-r^2}{1+r^2}
$$
respectively.

For a solution $u(t,r)$ close to a ground state $Q_\lambda$ the nonlinear problem \eqref{eq:nonl}
can be approximated by a linear inhomogeneous evolution 
$$
\pa_t^2 w + H_\lambda w=F,\qquad u(t,r)=Q_\lambda(r) + w(t,r)
$$
with the linear Hamiltonian 
$$
H_\lambda = -\Delta + k^2 \frac {f'(Q_\lambda)}{r^2}.
$$
We denote the Hamiltonian associated to $Q$ by
$$
H=-\Delta_y +k^2 \frac {f'(Q(y))}{y^2}.
$$
and recall the factorization property \eqref{eq:factor} of  $H$:
\be
\label{defhatsr}
 H=A^*A
\ee 
with
\be
\label{deoperatoaone}
A=-\partial_y+\frac{V^{(1)}}{y}, \ \ A^*=\partial_y+\frac{1+V^{(1)}}{y}, 
\ee
with 
\be
\label{deoperatoa}
V^{(1)}(y)=kg'(Q(y)),
\ee
and:
\be
\label{defaloambinot}
A_{\lambda}=-\partial_r+\frac{V^{(1)}_{\lambda}}{r}, \ \ A^*_{\lambda}=\partial_r+\frac{1+V^{(1)}_{\lambda}}{r}.
\ee This factorization is a consequence of the Bogomol'nyi's factorization of the Hamiltonian \fref{general_lines} 
or, alternatively \eqref{eq:fac-B}. Since $Q$ is an energy minimizer we expect the Hamiltonian $H$ to be 
non-negative definite and possess a kernel generated by the function $\Lambda Q$ -- generator of dilations 
(scaling symmetry) of the ground state $Q$. Factorization of $H$ however leads to even a stronger property,
which on one hand confirms that the kernel of $H$ is one dimensional but also
leads to the fundamental cancellation: 
\be
\label{cancnelalq}
A(\Lambda Q)=0,
\ee
that is $\Lambda Q$ lies in the kernel of $A$. We note that for $k=1$ the function $\Lambda Q$ 
is not in $L^2({\Bbb R}^2)$ and thus formally does not belong to the domain of $H$. The structure of the 
kernel of $H$ leads to the following statement of orbital stability of the ground state.
\begin{lemma}[Orbital stability of the ground state, \cite{Cote}, \cite{RS}]
\label{orbstab}
For any initial data $(u_0,u_1)$ with the property that $u_0=Q_{\lambda_0} + w_0$
and $\|(w_0,u_1)\|_{H}<\epsilon$ with $\epsilon$ sufficiently small, 
and for any $t\in [0,T)$ with $0<T\leq +\infty$ the maximum time of existence of the classical solution with 
data $(u_0,u_1)$, there exists a unique 
decomposition of the flow $$u(t)=Q_{\lambda(t)}+w(t)$$ with $\lambda(t)\in {\mathcal{C}}^2([0,T), \RR^*_+)$ and 
$$\forall t\in [0,T), \ \  |\partial_tu|_{L^2}+|\lambda_t(t)|+\|w(t),0\|_{\mathcal H}\lesssim O(\epsilon)
$$ satisfying 
the orthogonality condition 
\be
\label{defhovohf}
\forall t\in [0,T), \ \ (w(t,\lambda(t)\cdot),\chi_M\Lambda Q)=0.
\ee 
\end{lemma}
\begin{remark}
The cut-off function $\chi_M(r)=\chi(r/M)$ equal to one on the interval $[0,M]$ and vanishing for 
$r\ge 2M$ for some sufficiently large universal constant $M$ is introduced to accommodate the
case $k=1$ in which $\Lambda Q(y)$ decays with the rate $y^{-1}$ and thus misses the space
of $L^2$ functions. The imposed orthogonality condition is not standard, however the arguments
in \cite{Cote}, \cite{RS} can be easily adapted to handle this case. 
The statement of the Lemma in particular implies the coercivity of the Hamiltonian $H_\lambda$
\be
\label{coercivityneergy}
(H_\lambda w,w)=|A_\lambda w|^2_{L^2}\ge c(M) \int\left ((\pa_r w)^2+\frac {w^2}{r^2}\right),
\ee
provided that $(w(\lambda \cdot),\chi_M\Lambda Q)=0$.
\end{remark}

We introduce the function 
\be
\label{defW}
W(t,r)=A_{\lambda(t)}w
\ee
The energy type bound on $W$ will lead us to the $H^2$ type bound on $w$. To be more precise, 
we will control the $\tilde H$ norm of the function $\ep(s,y)=w(t,r)$, introduced in \eqref{eq:tH}. \\

We next turn to the equation for $W=A_{\lambda}w$. Following \cite{RS}, an important observation is that the Hamiltonian driving the evolution of $W$ is the {\it conjugate} Hamiltonian 
\be
\label{defhatsrbis}
 \tilde{H}_\lambda=A_{\lambda}A^*_{\lambda}=-\Delta +\frac{k^2+1}{r^2}+\frac{2V^{(1)}_{\lambda}+V^{(2)}_{\lambda}}{r^2}, \ \ V_2(y)=k^2\left[(g')^2-gg''-1\right](Q)
 \ee
which, as opposed to $H$, displays space-time {\it repulsive} properties. 
Commuting the equation for $w$ with $A_{\lambda}$ yields:
\be
\label{Wequation'}
\partial_{tt}W+\widetilde{H}_\lambda W=A_{\lambda}F+\frac{\partial_{tt}V_{\lambda}^{(1)}w}{r}+\frac{2\partial_t\vul\partial_tw}{r}.
\ee 
Observe that in the (WM) case $V^{(2)}\equiv 0$ and
\be
\label{cnheoiheoeuy}
k^2+1+2V^{(1)}+V^{(2)}=(k-1)^2+2k(1+\cos(Q))\geq
\left\{\begin{array}{ll} 1,\,\,\,{\text{for}}\,\,k\ge 2,\\
\frac{1}{1+r^2},\,\,{\text{for}}\,\,k=1.
\end{array}\right . 
\ee
For the (YM) problem $V^{(2)}=-2(1-Q^2)$ and, with $k=2$,
\be
\label{cnheoiheoeuybis} 
k^2+1+ 2V^{(1)} + V^{(2)}=1+2(1-Q)^2\geq 1.
\ee
These inequalities imply that the Hamiltonian $\tilde H_\lambda$
is a positive definite operator with the property that 
\be\label{eq:coerc}
(\tilde H_\lambda W,W)=|A_{\lambda}^*W|_{L^2}^2 \ge C\left\{\begin{array}{ll} \int \left ((\partial_r W)^2 + \frac {W^2}{r^2}\right) \ \ \mbox{for} \ \ k\ge 2,\\
\int \left ((\partial_r W)^2 + \frac {W^2}{r^2(1+\frac{r^2}{\lambda^2})}\right)\ \ \mbox{for} \ \ k=1,\\
\end{array} \right .
\ee
It is important to note that unlike $H_\lambda$, $\tilde H_\lambda$ is unconditionally coercive. However, it 
provides weaker control at infinity in the case $k=1$. The expression 
$$
\lambda^2 (\tilde H_\lambda W,W) + \lambda^2\|(\pa_t W,0)\|_{\mathcal H}^2
$$
is precisely the norm $\|\ep\|^2_{\tilde H}$ we ultimately need to control. Moreover, it obeys the estimate 
$$
\lambda^2 (\tilde H_\lambda W,W) + \lambda^2\|(\pa_t W,0)\|_{\mathcal H}^2\lesssim \|\ep\|_{\mathcal H^2}^2
$$
\\
Associated to the Hamiltonian $\tilde{H}_{\lambda}$, we define global and local energies $\mathcal E(t),\mathcal E_{\sigma}(t)$ used extensively in the paper:
\bea
\label{poitnwiseboundWbis}
\nonumber \mathcal E(t) & = & \lambda^2\int\left[(\partial_tW)^2+(\nabla W)^2+\frac{k^2+1+2\vul+\vdl}{r^2}W^2\right]\\
& = &\lambda^2\left[\int |A_{\lambda(t)}^*W(t)|^2+\int|\partial_tW(t)|^2\right],
\eea
\be
\label{lcalizedenrgybis}
{\mathcal E}_{\sigma}(t) =\lambda^2\int \sigma_{B_c}\left[(\partial_tW)^2+(\nabla W)^2+\frac{k^2+1+2\vul+\vdl}{r^2}W^2\right]
\ee
where we let $B_c=\frac{2}{b}$, as in 
\fref{defbnot}, and $\sigma_{B_c}$ be a cut off function 
\be
\label{ceheoehi}
\sigma_{B_c}(r)=\sigma(\frac{r}{\lambda B_c}) \  \ \mbox{with} \ \ \sigma(r)=\left\{\begin{array}{ll} 1 \ \ \mbox{for} \ \ r\leq 2\\
							    0 \ \ \mbox{for} \ \ r\geq 3,
							    \end{array}
							    \right .
							  \ee				
 We finish this section with the discussion on the admissibility of the functions $u(t,r)$, $w(t,r)=u(t,r)-(P_B)_\lambda(r)$
where  $(P_{B}(r))_\lambda$ is a deformation of $Q_{\lambda}$ which will be defined in section \ref{sectiontwo}. 
 The criterium for admissibility of $w(t,r)=\ep(s,y)$ will be the finiteness of the ${\mathcal H}^2$ norm of 
 $\ep$. 
 
 \begin{proposition}\label{prop:sym}
 Let $\Phi$ be a smooth solution of the (WM)/(YM) problem on the time interval $[0,T(\Phi_0))$ 
 with co-rotational/equivariant initial data $(\Phi_0,\Phi_1)$. Then $(\Phi(t),\pa_t\Phi(t))$ remains co-rotational/equivariant
 for any $t\in [0,T(\Phi_0))$ and its symmetry reduction $u(t,r)$ coincides with the solution of 
the nonlinear problem \eqref{WMsphere}/\fref{YMradial}. Moreover, for any  $t\in [0,T(\Phi_0))$ the function
$u(t)\in {\mathcal H}^2_a$.
 \end{proposition}
 
 {\bf Proof of Proposition \ref{prop:sym}}: The first part of the Proposition is a standard statement of propagation of symmetry. 
 We omit its proof. It remains to show that $u(t)\in {\mathcal H}^2_a$.  We give the argument for the (WM) case, the (YM) is left to the reader. We note that $$
 |\pa_r u|=|\pa_r \Phi|,\qquad |\sin (u)|=|\pa_\theta\Phi|,\qquad |\pa_r^2 u|=|\pa_r\Phi +(\pa_r\Phi,\pa_r\Phi)\Phi|
 $$
 As a consequence, for a smooth map $\Phi(t)$ the finiteness of the ${\mathcal H}_a^2$ norm of $u(t)$
 can only fail at $r=0$. To eliminate this possibility it will be sufficient to show that for $k\ge 2$
 $|\pa_r u|\le C r$, while for $k=1$ the function $|u|\le Cr$ and $|\pa_r u -\frac ur|\le Cr$. 
 The desired statement for $k\ge 2$ is contained in \cite{RS}. For $k=1$, arguing as in \cite{RS} 
 we derive that the energy density
 $$
 e(\Phi)(t,r)=|\pa_t u|^2+|\pa_ru|^2+\frac {\sin^2 u}{r^2}
 $$ 
 is a smooth function of $r^2$, which leads to the requirement that $|u|\le Cr$. Moreover, 
 differentiability of $\Phi$ also implies that
 $$
 \lim_{r\to 0} |\pa_r u|=\lim_{r\to 0} \frac {|\sin u|}{r},
 $$ 
which immediately gives the existence of
$$
\lim_{r\to 0} (\pa_r u) =\lim_{r\to 0} (\frac ur).
$$
 On the other hand, the algebra of \eqref{general_lines} implies that 
 $$
 |\pa_r u -\frac{\sin u}{r}|^2=\frac 12 (\pa_i \Phi -\ep_{ij} \Phi\times \pa_j \Phi)\cdot  
 (\pa_i \Phi -\ep_{ij} \Phi\times \pa_j \Phi)=v(\Phi)
 $$
 is a smooth function of $r^2$. Since $(\pa_r u -\frac{\sin u}{r})$ vanishes at the origin we obtain that
 $ |\pa_r u -\frac{\sin u}{r}|$ and hence  $|\pa_r u -\frac{u}{r}|$ obey the estimate
 $$
 |\pa_r u -\frac{u}{r}|\le C r,
 $$
 and this concludes the proof of Proposition \ref{prop:sym}.

\section{Construction of the family of almost self-similar solutions}
\label{sectiontwo}


This section is devoted to the construction of approximate self-similar solutions $Q_b$. These describe the dominant part of the blow up profile inside the backward light cone from the singular point $(0,T)$ and display a slow decay at infinity, which is eventually responsible for the $log$ modifications to the blow up speed. A related construction was made in the (NLS) setting in \cite{Perelman}, \cite{MR2}, where the ground state is exponentially decreasing. A simpler version of the profiles $Q_b=Q+b^2 T_1$, terminating 
at a $2$-term expansion was used in \cite{RS}.  The key to this construction is the fact that the structure of the linear operator $H=-\Delta + k^2 \frac {f'(Q)}{y^2}$ is completely explicit due to the variational nature of $Q$ as the minimizer of the associated nonlinear problem. 


\subsection{Self-similar equation}


Fix a small parameter $b>0$. Given $T>0$, a self-similar solution to \fref{equation} is of the form: 
\be
\label{decomppsi}
u(t,r)=Q_b(\frac{r}{\lambda}), \ \ \lambda(t)=b(T-t).
\ee 
The stationary profile $Q_b$ should solve the nonlinear elliptic equation:
\be
\label{eqselfsimilairebis}
-\Delta Q_b+b^2D\Lambda Q_b+k^2\frac{f(Q_b)}{y^2}=0.
\ee
This equation however admits no finite energy solutions, see \cite{KavianWeissler} for related results. We therefore construct  approximate solutions of finite energy, 
which exhibit the fundamental slow decay behavior in the region $y\geq \frac{1}{b}$.

The approximate solution $Q_b$ will be of the form 
\be
\label{defQbpropeven}
Q_b=Q+\Sigma_{j=1}^{p+1} b^{2j}T_j.
\ee
We will require that the profiles $T_j$ verify the orthogonality condition
\be
\label{orthod}
(T_j,\chi_{M}\Lambda Q)=0
\ee
with $\chi_{M}$ given by (\ref{defchib}).
The error associated to $Q_b$ is defined according to the formula
 \be
\label{defpsib}
\Psi_b(y)=-\Delta Q_b+b^2D\Lambda Q_b+k^2\frac{f(Q_b)}{y^2},
\ee 
For a given homotopy index $k$ we define an auxiliary integer parameter $p$
\be
\label{defp}
p=\left\{ \begin{array}{ll}
	\frac{k}{2} \ \ \mbox{for $k$ even}\\
	\frac{k-1}{2} \ \ \mbox{for $k$ odd}
	\end{array}
	\right .
\ee

\begin{proposition}[Approximate solution to the self-similar equation]
\label{propqb}
Let $M>0$ be a large universal constant to be chosen later and let $C(M)$ denote a generic large 
increasing function of $M$. 
Then there exists $b^*(M)>0$ such that for all $0<b\leq b^*(M)$ the following holds true.
There exist smooth radial profiles $(T_j)_{1\leq j\leq p+1}$ satisfying \fref{orthod} with the following properties:
\begin{itemize}
\item {\bf $k\ge 4$ even}: For all sufficiently small $y$ and $0\leq m\leq 3$,
\be
\label{asymtotthorigin}
\frac{d^mT_j}{dy^m}(y)=\tilde{c}_{j,m}y^{k-m}(1+O(y^{2})).\ \,
\ee
For $y\ge 1$, 
\begin{align}
\label{esttjsharpeven}
&\frac{d^mT_j}{dy^m}(y)=c_j\frac{d^my^{2j-k}}{dy^m}(1+\frac{f_j}{y^2}+O(\frac{1}{y^{3}})), \ \ 1\leq j\leq p-1,\ \ 0\leq m\leq 3,\\
\label{esttjsharppeven}
& T_p(y)=c_p(1+\frac{f_p}{y^2}+O(\frac{1}{y^{3}})), \ \ \frac{d^mT_p}{dy^m}(y)=f_pc_p\frac{d^my^{-2}}{dy^m}+O(\frac{1}{y^{3+m}}), \ \ 1\leq m\leq 3,\\
\label{esttjsharppevenpplusone}
&T_{p+1}(y)=O(1), \ \  \frac{d^mT_{p+1}}{dy^m}(y)=O(\frac{1}{y^{m+1}}), \ \ 1\leq m \leq 3. 
\end{align}
For $0\leq m\le 1$ the error term verifies
\be
\label{estpsibeven}
|\frac{d^m\Psi_b}{dy^m}(y)|\lesssim b^{k+4}\frac{y^{k-m}}{1+y^{k+1}}.
\ee
\item
{\bf $k\ge 3$ odd}: $(T_j)_{1\leq j\leq p}$ obey the asymptotics  (\ref{asymtotthorigin}) near the origin, while for all
$y\ge 1$ and $0\leq m\leq 3$
\be
\label{esttjsharp}
\frac{d^mT_j}{dy^m}(y)=c_j\frac{d^my^{2j-k}}{dy^m}(1+\frac{f_j}{y^2}+O(\frac{1}{y^{3}})), \ \ 1\leq j\leq p,
\ee
\be
\label{esttjsharpbis}
\frac{d^m}{dy^m}T_{p+1}(y)=O(\frac{1}{y^{1+m}}).
\ee
For $0\leq m\leq 1$ the error term verifies
 \be
\label{estpsibodd}
|\frac{d^m\Psi_b}{dy^m}(y)|\lesssim b^{k+3}\frac{y^{k-m}}{1+y^{k+2}}.
\ee

\item {\bf $k=2$}: There exist smooth profiles $T_1, T_2$ verifying (\ref{orthod}) such that for all sufficiently small $y$ and $j=1,2$,
\be
\label{asymtotthorigin2}
\frac{d^mT_j}{dy^m}(y)=C(M) O(y^{k-m}), \ \ 0\leq m\leq 3,\ \,
\ee
while for all $y\ge 1$ and $0\leq m\leq 3$,
\be
\label{esttjsharp'}
\frac{d^mT_j}{dy^m}(y)=\left\{\begin{array}{ll} 
			c_j\delta_{0m}+C(M) O(\frac{1}{y^{k+m}}), \ \ j=1,\\
			C(M) O(\frac 1{y^m}), \ \ j=2,
			\end{array} \right .
\ee

For $0\leq m\leq 1$ the error term verifies
\be
\label{vheohoevbobvob}
 |\frac{d^m}{dy^m}\left[\Psi_b+c_bb^4\Lambda Q\right]|\lesssim C(M) b^{k+4}\frac{y^{k-m}}{1+y^{k+1}},
\ee
for some constant $c_b=O(1)$.

\item
{\bf $k=1$}: We can find $T_1$ satisfying (\ref{orthod}),  such that for all sufficiently small $y$ and $0\leq m\leq 3$,
\be
\label{asymtotthorigin1}
\frac{d^mT_1}{dy^m}(y)=C(M) O(y^{k-m}),\ \,
\ee
while for $1\leq y\leq \frac{1}{b^2}$ and $0\leq m\leq 3$,
\be
\label{asympttone}
|\frac{d^m}{dy^m}T_1(y)|\lesssim (1+y^{1-m})\frac{1+|\log (by)|}{|\log b|}{\bf 1}_{y\leq \frac{B_0}{2}}+\frac{1}{b^2|\log b|(1+y^{1+m})}{\bf 1}_{y\geq \frac{B_0}{2}} + \frac {C(M)}{1+y^{1+m}}
\ee
The error term $\Psi_b$ satisfies for $0\leq m\leq 1$ and $0\le y\le \frac 1{b^2}$,
\bea
\label{estpsiboddkequalone}
& & \left|\frac{d^m}{dy^m}\left(\Psi_b-c_bb^2\chi_{\frac{B_0}{4}}\Lambda Q\right)\right|\\
\nonumber & \lesssim & b^4\frac {y^{1-m}}{1+y^4}+b^{4}\frac{(1+|\log (by)|)}{|\log b|}y^{1-m}{\bf 1}_{1\leq y\leq \frac{B_0}{2}}+\frac{b^2}{|\log b|y^{1+m}}{\bf 1}_{y\geq \frac{B_0}{2}} .
\eea
with a constant 
$$
|c_b|\lesssim \frac{1}{|\log b|}
$$
\end{itemize}
The constants $(c_j)_{1\leq j\leq p}$ in (\ref{esttjsharpeven}), (\ref{esttjsharppeven}), (\ref{esttjsharp}) are given by the recurrence formula:
\be
\label{computationcp}
\forall j\in[2,p],  \ \ c_j=c_{j-1}\frac{(k-2j+2)(k-2j+1)}{4j(k-j)}, \ \ c_1=\frac{k}{2}.
\ee
			
\end{proposition}

In the construction of the profile $Q_b$ the term $T_p(y)$ is a radiative term displaying an anomalous slow decay at infinity according to \fref{esttjsharppeven}, \fref{esttjsharp}, \fref{asympttone}. It is the first term which yields an unbounded contribution to the Hamiltonian of the corresponding self-similar solution $u$. The term $T_{p+1}$ is introduced in the decomposition to refine the behavior of the
error term $\Psi_b$ on compact sets, i.e. finite values of $y$, without destroying its radiative behavior far out.
This turns out to be more delicate for $k=1,2$ which explains a slightly pathological behavior of the error $\Psi_b$ in these cases, \fref{vheohoevbobvob}, \fref{estpsiboddkequalone}. Note that this is particularly true for $k=1$ where $p=0$ and $Q$ itself is the radiative term. In that case, 
introduction of the term $T_1$, which is however badly behaved for $y\geq \frac{1}{b}$ according to \fref{asympttone}, allows us to gain a factor of $\frac{1}{|\log b|}$ in the region $y\leq \frac{1}{b}$ in \fref{estpsiboddkequalone}. This should be contrasted with the polynomial gain in $b$ we see 
for higher values of $k$.

\begin{remark} The orthogonality condition \fref{orthod} corresponds to a choice of gauge for $Q_b$ 
allowed by the kernel of $H$, given by \fref{defhinitiale}. This choice will be convenient for an additional decomposition of the flow
near $Q_b$, see in particular \fref{orthe}.
\end{remark}



\subsection{Construction of $Q_b$}
\label{sectionqb}


 {\bf Proof of Proposition \ref{propqb}}\\
 
 Let $p$ be given by (\ref{defp}).\\
   
 {\bf step 1} Construction of an expansion. \\
 
The case $k=1$ will be treated separately. Let thus $k\geq 2$, $j\in [1,p]$ and $(T_{l})_{1\leq l\leq j}$ be any smooth radial function vanishing sufficiently fast both at zero and infinity , as in say (\ref{asymptotujkeven}). Let $$Q_b=\Sigma_{l=0}^{j}b^{2l}T_l, \ \ T_0=Q.$$ From the Taylor expansion of $f$: 
$$
f(Q_b)=f(Q)+\Sigma_{l=1}^j\frac{f^{(l)}(Q)}{l!}(b^2T_1+\dots+b^{2j}T_j)^l+R_{1,j}(b,y)
$$
with 
\be
\label{defrij}
R_{1,j}(b,y)=\frac{(Q_b-Q)^{j+1}}{j!}\int_0^1(1-u)^{j}f^{(j+1)}(uQ_b+(1-u)Q)du.
\ee 
We then reorder the polynomial part in $b$ to get: 
\bea
\label{defpolymonial}
\nonumber f(Q_b) & = & f(Q)+\Sigma_{l=1}^{j}b^{2l}\left[f'(Q)T_l+P_{l}(T_1, \dots,T_{l-1})\right]\\
& + & R_{1,j}(b,y)+R_{2,j}(T_1,\dots,T_j).
\eea
Here $P_l$ is a polynomial of degree $l$ with the convention that $P_1=0$ and the term $T_m$
contributes $m$ to the degree of $P_l$.
$R_{2,j}$ is a polynomial in $(T_l)_{1\leq l\leq j}$ and contains the terms of order $(b^{2l})_{l\geq j+1}$. Hence:
\be
\label{cancealltionrthrre}
\forall 0\leq l\leq j, \ \ \frac{\partial^l R_{1,j}(b,y)}{\partial (b^2)^l}_{|b=0}=\frac{\partial^l R_{2,j}(b,y)}{\partial (b^2)^l}_{|b=0}=0.
\ee 
We now expand the self similar equation:
\bea
\label{expansionselfsim}
\nonumber & & -\Delta Q_b+b^2D\Lambda Q_b+k^2\frac{f(Q_b)}{y^2}=-\Delta\left(Q+\Sigma_{l=1}^jb^{2l}T_l\right)+\left(\Sigma_{l=1}^jb^{2l}D\Lambda T_{l-1}\right)+b^{2(j+1)}D\Lambda T_j\\
\nonumber& + & \frac{k^2}{y^2}\left\{f(Q)+\Sigma_{l=1}^{j}b^{2l}\left[f'(Q)T_l+P_{l}(T_1, \dots,T_{l-1})\right]+ R_{1,j}+R_{2,j}\right\}\\
& = &\Sigma_{l=1}^jb^{2l}\left[HT_l+D\Lambda T_{l-1}+\frac{k^2}{y^2}P_{l}(T_1, \dots,T_{l-1})\right]+\frac{k^2}{y^2}(R_{1,j}+R_{2,j})+b^{2(j+1)}D\Lambda T_j.
\eea
We claim by induction on $1\leq j\leq p$ that we may solve the system: 
\be
\label{eqeqsttl}
HT_l+D\Lambda T_{l-1}+\frac{k^2}{y^2}P_{l}(T_1, \dots,T_{l-1})=0, \ \ 1\leq l\leq j
\ee 
with $(T_l)_{1\leq j}$ satisfying the desired estimates and the orthogonality condition (\ref{orthod}). Indeed, for $j=1$, we solve:   
\be
 \label{eqtun}
  HT_1+D\Lambda Q=0, \ \ (T_1,\chi_M\Lambda Q)=0,
  \ee 
explicitely by setting 
\be
\label{defgeneraletun}
T_1=\frac{1}{4}y^2\Lambda Q-\frac{\int\chi_My^2(\Lambda Q)^2}{4\int \chi_M(\Lambda Q)^2}\Lambda Q.
\ee In the (WM) case for $k\ge 3$, it satisfies from (\ref {ezpansionLQ}) the asymptotics: 
 \be
 \label{asymptoictun}
T_1(y)=\left\{ \begin{array}{ll}
	 \tilde c_1y^k(1+O(y^{k})) \ \ \mbox{as} \ \ y\to 0\\
	 {c}_1\frac{y^{2}}{y^{k}}(1+\frac{f_1}{y^2}+O(\frac{1}{y^{3}}))\ \ \mbox{as} \ \ y\to +\infty,\\
	 \end{array}
	 \right .
 \ee
and for $k=1,2$:
\be
 \label{asymptoictunksmall}
T_1(y)=\left\{ \begin{array}{ll}
	 \tilde c_1y^k(1+O(y^{k})) \ \ \mbox{as} \ \ y\to 0\\
	 {c}_1\frac{y^{2}}{y^{k}}(1+C(M)O(\frac{1}{y^{2}}))\ \ \mbox{as} \ \ y\to +\infty,\\
	 \end{array}
	 \right .
 \ee
with 
$$C(M)\sim\left\{\begin{array}{ll} 
\log M   \ \ \mbox{for} \ \ k=2,\\
 \frac{M^2}{\log M} \ \mbox{for} \ \ k=1,
 \end{array} \right .
 $$ In the (YM) $k=2$ case
  \be
 \label{asymptoictun'}
T_1(y)=\left\{ \begin{array}{ll}
	 -\tilde c_1y^k(1+\log M O(y^{k})) \ \ \mbox{as} \ \ y\to 0\\
	 -{c}_1\frac{y^{2}}{y^{k}}(1+O(\frac{\log M}{y^{k}}))\ \ \mbox{as} \ \ y\to +\infty.\\
	 \end{array}
	 \right .
\ee
 In all cases,
 \be
 \label{defd1}
 c_1=\frac{k}{2}.
 \ee
 Hence $T_1$ satisfies \fref{orthod}, \eqref{asymtotthorigin}, \eqref{esttjsharpeven}, (\ref{esttjsharp}),
 \eqref{asymtotthorigin2} and (\ref{esttjsharp'}) for $j=1$.\\
 
 {\bf step 2} Induction for $k\ge 3$.\\
 
 For $k=3$, we have $p=1$ and $T_2=T_{p+1}$ will be constructed in step 4. We hence assume $k\geq 4$ and now argue by induction on $j$ using Lemma \ref{inversionl}. We assume that we could solve (\ref{eqeqsttl}) for $1\leq l\leq j-1$ with $(T_l)_{1\leq l\leq j-1}$ satisfying \fref{asymtotthorigin}, \fref{esttjsharpeven}, (\ref{esttjsharp}). In  order to apply Lemma \ref{inversionl}, we need to show the orthogonality:
 \be
 \label{esttoshow}
 \left(D\Lambda T_{j-1}+\frac{k^2}{y^2}P_{j}(T_1, \dots,T_{j-1}), \Lambda Q\right)=0.
 \ee
 Assume (\ref{esttoshow}). Then from Lemma \ref{inversionl}, we may solve (\ref{eqeqsttl}) for $l=j$ with $T_j$ satisfying (\ref{orthod}). Moreover, from the decay properties of $(T_1,\dots, T_{j-1})$ at infinity and the polynomial structure of $P_j(T_1,\dots,T_{j-1})$, the leading order term on the RHS of (\ref{eqeqsttl}) as $y\to +\infty$ is given by $D\Lambda T_{j-1}=2yT_{j-1}'+y^2T_{j-1}''$ that is: $$D\Lambda T_{j-1}+\frac{k^2}{y^2}P_{j}(T_1, \dots,T_{j-1})=(k-2j+2)(k-2j+1)c_{j-1}\frac{y^{2(j-1)}}{y^{k}}(1+O(\frac{1}{y^{2}})).$$ (\ref{asymptotuj}), (\ref{asymptotujkeven}), (\ref{asymptotujkevenderivative}) now allow us to derive the asymptotics of $T_j,T_j'$ near $+\infty$, and higher derivatives are controlled using the equation (\ref{eqeqsttl}).\\
 Estimates \fref{esttjsharpeven}, (\ref{esttjsharp}) follow with the recurrence formula: $$c_{j}=c_{j-1}\frac{(k-2j+2)(k-2j+1)}{4j(k-j)}, $$ which gives (\ref{computationcp}). Similarly, the $y^k$ vanishing of $(T_{l})_{1\leq l\leq j-1}$ at the origin ensures that the same vanishing holds for $\left(\frac{P_l(T_1,\dots,T_{l-1})}{y^2}\right)_{2\leq l\leq j}$, and (\ref{asymtotthorigin}) follows.\\
 
 \noindent
 Proof of (\ref{esttoshow}): Note that a direct algebraic proof seems hopeless due to the nonlinear structure of the problem. However, we claim that  (\ref{esttoshow}) is a simple consequence of the energy criticality of the problem and the cancellation provided by the Pohozaev identity. Let $(T_l)_{0\leq l\leq j-1}$ be the first constructed profiles and let $T_j$ be any smooth radial function vanishing sufficiently fast both at zero and infinity. Let $Q_b=\Sigma_{l=0}^jb^{2l}T_l$, then: $$F(b)=\left(-\Delta Q_b+b^2D\Lambda Q_b+k^2\frac{f(Q_b)}{y^2},\Lambda Q_b\right)=0.$$ Let us indeed recall that this holds true for any smooth $Q_b$ which decays enough both at the origin and infinity. Note also that we are implicitly using the condition $j\leq p$ which ensures from \fref{esttjsharpeven}, (\ref{esttjsharp}) that the integration by parts does not create any boundary terms for the $(T_l)_{1\leq l\leq j-1}$ terms. We conclude that the Taylor series of $F$ at $b=0$ vanishes to all orders. On the other hand, from the decomposition (\ref{expansionselfsim}), 
 \bee
F(b)
  & = & \Sigma_{l=1}^jb^{2l}\left[HT_l+D\Lambda T_{l-1}+\frac{k^2}{y^2}P_{l}(T_1, \dots,T_{l-1})\right]
+\frac {k^2}{y^2}R_{1,j}+\frac{k^2}{y^2}R_{2,j}\\ &+&b^{2(j+1)}D\Lambda T_j,\Lambda Q+\Sigma_{l=1}^jb^{2l}\Lambda T_l\\
& = & b^{2j}\left[HT_j+D\Lambda T_{j-1}+\frac{k^2}{y^2}P_{j}(T_1, \dots,T_{j-1})\right]+
\frac{k^2}{y^2}R_{1,j}+\frac{k^2}{y^2} R_{2,j}\\ &+&b^{2(j+1)}D\Lambda T_j,\Lambda Q+\Sigma_{l=1}^jb^{2l}\Lambda T_l
\eee
where we used that (\ref{eqeqsttl}) is satisfied for $1\leq l\leq j-1$. (\ref{cancealltionrthrre}) now implies: $$0=\frac{d^{2j}}{db^{2j}}F(b)_{|b=0}=\left(HT_j+D\Lambda T_{j-1}+\frac{k^2}{y^2}P_{j}(T_1, \dots,T_{j-1}), \Lambda Q\right).$$ Now $(HT_j,\Lambda Q)=(T_j,H\Lambda Q)=0$ for any $T_j$ and (\ref{esttoshow}) follows.\\

{\bf step 3} Estimate on the error at the order $p$.\\

Let now $\Psi^{(p)}_b$ be given by (\ref{defpsib}) for $Q_b=\Sigma_{l=0}^pb^{2l}T_l$, explicitly from  (\ref{expansionselfsim}):
\be
\label{estfkjeo}
\Psi^{(p)}_b=\frac{k^2}{y^2} R_{1,p}+\frac{k^2}{y^2} R_{2,p}+
b^{2(p+1)}D\Lambda T_p.
\ee
 $R_{1,p}$ given by (\ref{defrij}) and $R_{2,p}$ are given by (\ref{defpolymonial}) are estimated using the uniform bound on $(\|f^{(j)}\|_{L^{\infty}})_{1\leq j\leq p}$ and the behavior of $T_j$ near the origin and infinity:

For $k$ odd and $0\leq m\leq 1$:
\bea
\label{fiohohgoe}
|\frac{d^m y^{-2} R_{1,p}}{dy^m}(y)|&\lesssim& b^{2(p+1)}\frac {y^{(p+1)k-m-2}}{1+y^{2(p+1)(k-1)}}+b^{2p(p+1)}
\frac {y^{(p+1)k-m-2}}{1+y^{(p+1)(k+1)}}, \\
|\frac{d^m y^{-2}
R_{2,p}}{dy^m}(y)|&\lesssim& b^{2(p+1)}\frac {y^{2k-m-2}}{1+y^{3k-1}}+b^{2p^2}
\frac {y^{pk-m-2}}{1+y^{pk+p}}, \label{fiohohgoe1}.
\eea
Note that $R_{2,p}$ is non-trivial only for $k\geq 5$.\\
For $k$ even and $0\leq m\leq 1$:
\bea
\label{fiohohgoee}
|\frac{d^m y^{-2} R_{1,p}}{dy^m}(y)|&\lesssim& b^{2(p+1)}\frac {y^{(p+1)k-m-2}}{1+y^{2(p+1)(k-1)}}+b^{2p(p+1)}
\frac {y^{(p+1)k-m-2}}{1+y^{(p+1)k}}, \\
|\frac{d^m y^{-2}
R_{2,p}}{dy^m}(y)|&\lesssim& b^{2(p+1)}\frac {y^{2k-m-2}}{1+y^{3k-2}}+b^{2p^2}
\frac {y^{pk-m-2}}{1+y^{pk}}.\label{fiohohgoee1}
\eea
Note that $R_{2,p}$ is non-trivial only for $k\geq 4$.
\vskip 1pc
 It remains to estimate the leading order term $D\Lambda T_p$ in \fref{estfkjeo}. Recall the asymptotics of $T_p$ near $y+\infty$ from \fref{esttjsharppeven}, \fref{esttjsharp}:
 $$T_p(y)=c_p(1+\frac{f_p}{y^2}+O(\frac{1}{y^3})) \ \ \mbox{for k even},$$
 $$ T_p(y)=\frac{c_p}{y}(1+\frac{f_p}{y^2}+O(\frac{1}{y^3})) \ \ \mbox{for k odd}.$$
 We now use in a fundamental way the cancellation 
 \be
 \label{keycancleaation} D\Lambda(\frac{1}{y})=D\Lambda(1)=0
 \ee 
 which yields in particular as $y\to +\infty$: 
 \be
 \label{behavioprdltp}
 D\Lambda T_p(y)=\left\{\begin{array}{ll} \frac{\tilde{f_p}}{y^2}+O(\frac{1}{y^3}) \ \  \mbox{for k even},\\
 \frac{\tilde{f_p}}{y^3}+O(\frac{1}{y^4}) \ \  \mbox{for k odd},
 \end{array} \right .
 \ee
and the crude bounds:
 $$
 |\frac{d^mD\Lambda T_p}{dy^m}(y)|\lesssim \frac{y^{k-m}}{1+y^{k+2}}, \ \ 0\leq m\leq 1 \ \ \mbox{for k even},
 $$ 
 $$
 \label{neovhoeghobis}
 |\frac{d^mD\Lambda T_p}{dy^m}(y)|\lesssim \frac{y^{k-m}}{1+y^{k+3}}, \ \ 0\leq m\leq 1\ \ \mbox{for k odd}.
 $$
 These estimates together with (\ref{fiohohgoe})-(\ref{fiohohgoee1})  now yield: 
\be
\label{estpsiodd}
|\frac{d^m}{dy^m}\Psi^{(p)}_b|\lesssim \frac{b^{k+2}y^{k-m}}{1+y^{k+2}},  \ \ 0\leq m\leq 1,\ \ \mbox{for} \ \ \mbox{k even},
\ee
\be
\label{estpsiven}
|\frac{d^m}{dy^m}\Psi^{(p)}_b|\lesssim \frac{b^{k+1}y^{k-m}}{1+y^{k+3}},  \ \ 0\leq m\leq 1, \ \ \mbox{for} \ \ \mbox{k odd}.
\ee

{\bf step 4} Construction of $T_{p+1}$ for $k\geq 3$.
\vskip 1pc

Observe that for all $k\geq 1$, $T_p$ is the radiative term  in the sense that as $y\to+\infty$: $$T_p\sim \frac{1}{y} \ \ \mbox{for k odd}, \ \ T_p\sim 1 \ \ \mbox{for k even}.$$ 
Note that for $k=1$ we have $p=0$ and $T_0=Q$.\\

The estimates \fref{estpsiodd}, \fref{estpsiven} are not sufficient for our analysis. Therefore we add  an extra term $T_{p+1}$ by taking advantage of the cancellations \fref{keycancleaation}. The cases $k=1,2$ are degenerate and require a separate treatment.\\

For $k\geq 3$, we need to solve:
\be\label{eq:p+1}
LT_{p+1}+D\Lambda T_{p}+\frac{k^2}{y^2}P_{p+1}(T_1, \dots,T_{p})=0.
\ee
 To do this, we first need to verify the orthogonality condition for $k\geq 3$:
 \be
 \label{fosfohs}
 \left(D\Lambda T_{p}+\frac{k^2}{y^2}P_{p+1}(T_1, \dots,T_{p}), \Lambda Q\right)=0.
 \ee
 As before we may define $Q_b=\Sigma_{l=0}^{p+1} b^{2l}T_l$ with an arbitrary smooth rapidly decaying
 function $T_{p+1}$ and 
 $$
 F(b)=\left(-\Delta Q_b+b^2D\Lambda Q_b+k^2\frac{f(Q_b)}{y^2},\Lambda Q_b\right)
$$
so that $$ \left(D\Lambda T_{p}+\frac{k^2}{y^2}P_{p+1}(T_1, \dots,T_{p}), \Lambda Q\right)=\frac 1{(2(p+1))!}\, \frac{d^{2(p+1)}F(b)}{b^{2(p+1)}}|_{b=0}.$$
We now claim:
\be
\label{estineteko}
F(b)=\frac{c_p^2}{2}b^{2k}(1+o(1)) \ \ \mbox{as} \ \ b\to 0.
\ee
Indeed, let $R>0$ and recall the Pohozaev integration: for any smooth enough $\phi$, 
\be
\label{pohozaev}
\int_{r\leq R}\left(-\Delta \phi+b^2D\Lambda \phi+k^2\frac{f(\phi)}{y^2}\right)\Lambda \phi   =  \left[-\frac{1}{2}(r\phi')^2+\frac{b^2}{2}|r\Lambda \phi|^2+\frac{k^2g(\phi)}{2}\right](R).
\ee
Applying this with $\phi=Q_b$ yields:
$$\lim_{R\to +\infty} \int_{r\leq R}\left(-\Delta Q_b+b^2D\Lambda Q_b+k^2\frac{f(Q_b)}{y^2}\right)\Lambda Q_b   =   \lim_{R\to +\infty} \frac{b^2}{2}|r\Lambda Q_b|^2(R)+\frac{k^2}{2}|g(Q_b)|^2(R)$$
and hence:
$$F(b)= \left\{\begin{array}{ll} \frac{c_p^2b^{4p+2}}{2}=\frac{c_p^2b^{2k}}{2} \ \ \mbox{for k odd}\\
					       \frac{c_p^2b^{4p}}{2}=\frac{c_p^2b^{2k}}{2}  \ \ \mbox{for k even}
							\end{array} \right .
$$
where we used in the last step the asymptotics \fref{esttjsharppeven}, (\ref{esttjsharp}) for $j=p$ for $T_p$. Combining (\ref{estineteko}) with the analytic dependence of $F(b)$ on $b$, we conclude that for $k\geq 3$ (recall that $2(p+1)=k+1<2k$ for $k$ odd and $2(p+1)=k+2<2k$ for $k$ even):
$$
\frac {d^{2(p+1)}}{db^{2(p+1)}} F(b)|_{b=0}=0
$$
and the desired orthogonality condition follows. We now argue exactly as in the proof of Lemma \ref{inversionl} to construct $T_{p+1}$ solution to (\ref{eq:p+1}) satisfying from \fref{behavioprdltp} the estimate \fref{asymtotthorigin} near the origin and for $y\geq 1$:
$$T_{p+1}=c_{p+1}(1+O(\frac{1}{y})),   \ \ |\frac {d^m}{dy^m} T_{p+1}(y)|\lesssim\frac {y^{k-m}}{1+y^{k+1}}, \ \ 0\leq m\leq 2  \ \  \mbox{for k even}
$$
$$ T_{p+1}=\frac{c_{p+1}}{y}(1+O(\frac{1}{y})), \ \ |\frac {d^m}{dy^m} T_{p+1}(y)|\lesssim\frac {y^{k-m}}{1+y^{k+1}}, \ \ 0\leq m\leq 2, \ \ \mbox{for k odd}.
 $$
 In the even case, we used here the same cancellation which led to \fref{asymptotujkevenderivative} for the $\frac{1}{y^2}$ part of the behavior of $D\Lambda T_p$ in the asymptotics \fref{behavioprdltp} near $y\to +\infty$. We cannot retrieve the same cancellation on the part induced by the $O(\frac{1}{y^3})$ tail but we simply need the rough bound $|T_{p+1}'|\lesssim \frac{1}{y^2}$ at $+\infty$.\\
 Using the degeneracy \fref{keycancleaation}, this leads to the bound for $0\leq m\leq 1$:
 \be
 \label{cnokehohoefh}
 |\frac {d^m}{dy^m} D\Lambda T_{p+1}(y)|\lesssim \frac {y^{k-m}}{1+y^{k+2}} \ \ \mbox{for k odd}, 
\ee
\be
\label{cnokehohoefhbis}
|\frac {d^m}{dy^m} D\Lambda T_{p+1}(y)|\lesssim \frac {y^{k-m}}{1+y^{k+1}} \ \ \mbox{for k even}.
\ee

We now define 
$$
\Psi_b=\frac {k^2}{y^2} R_{1,p+1} +\frac {k^2}{y^2} R_{2,p+1}+b^{2(p+2)} D\Lambda T_{p+1}. 
$$
The estimates on the first two terms are already contained in (\ref{fiohohgoe})-(\ref{fiohohgoee1}), and \fref{cnokehohoefh}, \fref{cnokehohoefhbis} now imply  \fref{estpsibeven}, \fref{estpsibodd}.\\

{\bf step 5} Construction of $T_2$ for $k=2$.
\vskip 1pc

We now turn to the $k=2$ case. Observe that the fundamental cancellation \fref{keycancleaation} still holds, but the orthogonality condition \fref{fosfohs} fails. This failure is due to the fact that $2(\frac{k}{2}+1)=2k$. Let $T_1$ be given by \fref{defgeneraletun} and 
\be
\label{defcbbis}
c_b=\frac{(D\Lambda T_1+\frac{k^2}{2y^2}f''(Q)T_1^2,\Lambda Q)}{|\Lambda Q|_{L^2}^2}=\frac {c_1^2}{2|\Lambda Q|_{L^2}^2}\sim 1,
\ee
then $T_1$ satisfies the asymptotics \fref{asymtotthorigin2}, \fref{esttjsharp'} from \fref{asymptoictun}, \fref{asymptoictun'}. Let then $T_2$ be the solution given\footnote{Formally, Lemma \ref{inversionl} can be applied only in
the context of the (WM) problem and with $k\ge 3$. The argument however can be easily modified to
satisfy our current needs. We sketch the argument below.} by Lemma \ref{inversionl} to 
$$HT_2=-D\Lambda T_1-\frac{k^2}{2y^2}f''(Q)T_1^2+ c_b\Lambda Q=g.$$ 
Explicitely, from \fref{defulinearsolver}, $T_2=\tilde{T}_2-c_M\Lambda Q$ with:
$$
 \tilde{T}_2(y)=\G(y)\int_0^yg(x)J(x)xdx-J(y)\int_{1}^yg(x)\G(x)xdx.
$$
and
$$
c_M=\frac {(\tilde T_2, \chi_M \Lambda Q)}{(\Lambda Q, \chi_M \Lambda Q)}.
$$
The asymptotics (\ref{asymtotthorigin}) near the origin follow easily from \fref{asymptoictun}, \fref{asymptoictun'}.
For $y\geq 1$, we have from $(g,\Lambda Q)=0$: 
\bee
|\tilde{T}_2(y)| & = & \left|\G(y)\int_y^{+\infty} g(x)J(x)xdx+J(y)\int_{1}^yg(x)\G(x)xdx\right|\\
& \lesssim & y^2\int_{y}^{\infty}\frac{xdx}{(1+x^2)^2}+\frac{1}{y^2}\int_{1}^{y}\frac{x^3dx}{1+x^2}\lesssim C(M).
\eee
Therefore,
$$
c_M\lesssim C(M). 
$$
This leads to \fref{esttjsharp'} for $m=0$ and $j=2$. Higher order derivatives are estimated similarily. We now compute the error $\Psi_b$:
\bee
& & \Psi_b=-\Delta Q_b+b^2D\Lambda Q_b+k^2\frac{f(Q_b)}{y^2}\\
& = & b^4\left[LT_2+D\Lambda T_1+\frac{k^2}{2y^2}f''(Q)(T_1)^2\right]+b^6D\Lambda T_2\\
& + & \frac{k^2}{y^2}\left[f(Q+b^2T_1+b^4T_2)-f(Q)-b^2f'(Q)(T_1+b^2T_2)-\frac{b^4T_1^2}{2}f''(Q)\right]
\eee
from which:
$$
\left|\Psi_b+c_bb^4\Lambda Q\right| \lesssim  b^6\left[|D\Lambda T_2|+C(M)\frac{y^2}{1+y^4}\right] \lesssim  C(M) b^6\frac{y^2}{1+y^3}.
$$
This is \fref{vheohoevbobvob} for $m=0$, the case $m=1$ follows similarily.\\

{\bf step 6} Construction of $T_1$ for $k=1$.
\vskip 1pc

We now turn to the $k=1$ case. The cancellation \fref{keycancleaation} still holds, but the orthogonality condition \fref{fosfohs} fails since for $k=1$, $2(\frac{k-1}{2}+1)=2k$. This reflects the fact that $\Lambda Q\sim \frac{1}{y}$ is already the radiative term, and the non vanishing quantity on the LHS of \fref{fosfohs} is exactly the flux term driving the blow up speed. This can equivalently be seen in the anomalous growth of 
$$T_1^0=\frac{y^2}{4}\Lambda Q\sim y \,\,\ \ \mbox{solution of}\,\, \ \ HT_1^0+D\Lambda Q=0.$$ Let 
\be
\label{defcbbisbis}
c_b=\frac{(D\Lambda Q,\Lambda Q)}{(\Lambda Q,\chi_{\frac{B_0}{4}}\Lambda Q)}\sim \frac{C}{|\log b|}
\ee
and $T_1$ be the solution given by Lemma \ref{inversionl} to 
$$LT_1=-D\Lambda Q+c_b\Lambda Q\chi_{\frac{B_0}{4}}=g,$$ explicitely $T_1=\tilde{T}_1-c_M\Lambda Q$ with
$$
c_M=\frac {(\tilde T_1, \chi_M \Lambda Q)}{(\Lambda Q, \chi_M \Lambda Q)}
$$
and from \fref{defulinearsolver}:
$$
 \tilde{T}_1(y)=\G(y)\int_0^yg(x)J(x)xdx-J(y)\int_{1}^yg(x)\G(x)xdx.
$$
The asymptotics (\ref{asymtotthorigin}) near the origin follow easily. For $y\ge 1$, we first have from the orthogonality condition $(g,\Lambda Q)=0$, implied by (\ref{defcbbisbis}), and the degeneracy \fref{keycancleaation}, which implies that $|D\Lambda Q|\le y^{-3}$ for $y\ge 1$, that
for $\frac{1}{b^2}\ge y\geq \frac{B_0}{2}$, 
\bee
|\tilde{T}_1(y)| & = & \left|\G(y)\int_y^{+\infty} g(x)J(x)xdx+J(y)\int_{1}^yg(x)\G(x)xdx\right|\\
& \lesssim & (1+y)\int_{y}^{\infty}\frac{dx}{1+x^3}+\frac{1}{y}\left[\int_{1}^{y}\frac{x^2dx}{1+x^3}+|c_b|\int_1^{B_0}\frac{x^2dx}{1+x}\right]\\
& \lesssim & \frac{1}{b^2|\log b|}\frac{1}{1+y}.
\eee
On the other hand, for $1\leq y\leq \frac{B_0}{2}$:
\bee
|\tilde{T_1}(y)| & = & (1+y)\int_y^{+\infty} \frac{dx}{1+x^3}+|c_b|(1+y)\int_y^{B_0}\frac{dx}{1+x}+\frac{1}{1+y}\int_{1}^{y}x^2dx\left[\frac{1}{1+x^3}+\frac{|c_b|}{x}\right]\\
& \lesssim &  \frac{1+y}{|\log b|}(1+|\log (by)|){\bf 1}_{y\leq \frac{B_0}{2}}.
\eee
The constant $c_M$ can be then estimated:
$$
c_M\le C(M)
$$
This leads to \fref{asympttone} for $m=0$. Higher order derivatives are estimated similarily. We now compute the error $\Psi_b$:
\bee
& & \Psi_b=-\Delta Q_b+b^2D\Lambda Q_b+k^2\frac{f(Q_b)}{y^2}\\
& = & b^2(LT_1+D\Lambda Q)+b^4D\Lambda T_1+\frac{k^2}{y^2}\left[f(Q+b^2T_1)-f(Q)-b^2f'(Q)T_1\right].
\eee
Using the cancellation for the term $D\Lambda (c_M\Lambda Q)$ we then obtain
\bee
\left|\Psi_b+c_bb^2\Lambda Q\chi_{\frac{B_0}{4}}\right|& \lesssim & b^4\left[|D\Lambda T_1|+\frac{1}{y^2}T_1^2 
\int_0^1 \int_0^1 \tau f''(Q+\tau' \tau b^2 T_1) d\tau' d\tau\right]\\
& \lesssim & C(M) b^4\frac {y}{1+y^4}+ b^4\frac{1+y}{|\log b|}(1+|\log (by)|){\bf 1}_{1\leq y\leq \frac{B_0}{2}}+\frac{b^2}{|\log b|}\frac{{\bf 1}_{y\geq \frac{B_0}{2}}}{y},
\eee
where we used the behavior $|f''(y)|\lesssim y$ for $y\le 1$. This is \fref{estpsiboddkequalone} for $m=0$, the case $m=1$ follows similarily.

For future reference we also note the following improved behavior in the region $y\ge B_0$. First, we compute
\begin{align*}
\Lambda \tilde T_1&=-\Lambda \Gamma(y) \int_y^\infty g(x) J(x) x \, dx - \Lambda J(y) \int_1^y g(x) \Gamma (x) x\,dx,\\ 
D\Lambda \tilde T_1&=-D\Lambda \Gamma(y) \int_y^\infty g(x) J(x) x \, dx +\Lambda \Gamma(y) g(y) J(y) y^2
\\ &- D\Lambda J(y) \int_1^y g(x) \Gamma (x) x\,dx - \Lambda J(y) g(y) \Gamma(y) y^2
\end{align*}
We now observe that $|D\Lambda J(y)|\lesssim y^{-3}$ for $y\ge 1$ and that the worst term in $g$ is supported in
$y\le B_0/2$. Therefore, for $y\ge B_0$
\bee
|D\Lambda\tilde{T}_1(y)| & \lesssim &(1+y)\left (\int_{y}^{\infty}\frac{dx}{1+x^3}+\frac 1{1+y^2}\right)+
\frac{1}{y^3}\left (\left[\int_{1}^{y}\frac{x^2dx}{1+x^3}+|c_b|\int_1^{B_0}\frac{x^2dx}{1+x}\right]+\frac {y^3}{1+y^3}\right)\\
& \lesssim & \frac{1}{1+y}.
\eee
Repeating the calculation for $\Psi_b$, we obtain for $y\ge B_0$
\begin{align}
\left|\Psi_b\right|& \lesssim  b^4\left[|D\Lambda \tilde T_1|+\frac{1}{y^2}T_1^2 
\int_0^1 \int_0^1 \tau f''(Q+\tau' \tau b^2 T_1) d\tau' d\tau\right]\notag\\
& \lesssim  \frac {b^4}{1+y}+ b^4\frac{1}{y^5 b^4 \log^2 b}\lesssim \frac {b^4}{1+y}\label{eq:imp-Psi}
\end{align}

This concludes the proof of Proposition \ref{propqb}.


\subsection{Profile localization}


Observe from (\ref{esttjsharppeven}), (\ref{esttjsharp}) that the profiles $T_p$ possess tails slowly decaying at infinity. The
behavior of these tails, near the light cone $y\sim \frac{1}{b}$, are responsible for a leading order phenomenon in determining the blow up speed, but their slow decay becomes irrelevant  for $y>>\frac{1}{b}$, where $Q_b$ is no longer a good approximation of the solution. In this region, the nonlinear interaction is over and we simply match the profile to its asymptotic value $a$. Note that the existence of an exact {\it constant} self-similar stationary solution to the full nonlinear problem turns out to be important for the analysis for small $k$. We thus introduce a localized version of the $Q_b$ profile as follows. Recall the two different scales $B_0, B_1$ defined in \fref{defbnot} and 
let $$B\in \{B_0,B_1\}\ \ \mbox{with} \ \ 
B_0=\frac{1}{b\sqrt {3\int y\chi(y) dy}}, \ \ B_1=\frac{|\log b|}{b}.$$ We then define:
\be
\label{defqbtilde}
P_{B}=(1-\chi_B)a+\chi_BQ_b,
\ee
where $$a=\lim_{y\to +\infty} Q(y)=\left\{\begin{array}{ll} \pi \ \ \mbox{for (WM)}\\ a=-1 \ \ \mbox{for (YM)} \end{array}\right .
$$ and 
$Q_b$ is given by Proposition \ref{propqb}. We now collect the estimates on this localized profile $P_{B}$ which are a simple consequence of Proposition \ref{propqb}.

\begin{proposition}[Estimates on the localized profile]
\label{lemmapsibtilde}
Let
\be
\label{defpsitilda}
\Psi_B=-\Delta P_{B}+b^2D\Lambda P_{B}+k^2\frac{f(P_{B})}{y^2}.
\ee
Then 
\be
\label{supppsib}
Supp(\Psi_B)\subset \{y\leq 2B\}
\ee
and there holds the estimates:\\
(i) For $k\ge 4$ even,
\be
\label{estp[rgehdiob}
\left|\frac{d^m }{dy^m}\frac{\partial P_{B}}{\partial b}\right|\lesssim b\frac{y^{k-m}}{1+y^{2k-2}}{\bf 1}_{y\leq \frac{1}{b}}+\frac {b^{k-1}}{y^m}{\bf 1}_{\frac{1}{b}\leq y\leq 2B}, \ \ 0\leq m\leq 3,
\ee
\be
\label{estpsibeventilde}
|\frac{d^m\Psi_B}{dy^m}(y)|\lesssim b^{k+4}\frac{y^{k-m}}{1+y^{k+1}}{\bf 1}_{y\leq B}+\frac{b^{k+2}}{y^{m}}{\bf 1}_{B\leq y\leq 2B}, \ \ 0\leq m\leq 1,
\ee
(ii) For $k\geq 3$ odd,
\be
\label{estbodd}
\left|\frac{d^m }{dy^m}\frac{\pa P_{B}}{\partial b}\right|\lesssim b\frac{y^{k-m}}{1+y^{2k-2}}{\bf 1}_{y\leq \frac{1}{b}}+\frac{b^{k-2}}{y^{1+m}}{\bf 1}_{\frac{1}{b}\leq y\leq 2B}, \ \ 0\leq m\leq 3,
\ee
\be
\label{estpsiboddtilde}
|\frac{d^m\Psi_B}{dy^m}|\lesssim b^{k+3}\frac{y^{k-m}}{1+y^{k+2}}{\bf 1}_{y\leq B}+\frac{b^{k+1}}
{1+y^{m+1}}{\bf 1}_{B\leq y\leq 2B}, \ \ 0\leq m\leq 1.
\ee
(iii) For $k=2$
\be
\label{estp[rgehdiobisbis}
\left|\frac{d^m }{dy^m}\frac{\partial P_{B}}{\partial b}\right|\lesssim b\frac{y^{2-m}}{1+y^{2}}{\bf 1}_{y\leq \frac{1}{b}}+\frac {b}{y^m}{\bf 1}_{\frac{1}{b}\leq y\leq 2B} +C(M)b\frac{y^{2-m}}{1+y^{4}}{\bf 1}_{y\leq 2B} , \ \ 0\leq m\leq 3,
\ee
\be
\label{estpsibeventildektwo}
|\frac{d^m}{dy^m}\left[\Psi_B-c_bb^4\chi_B\Lambda Q\right]|\lesssim C(M) b^{k+4}\frac{y^{k-m}}{1+y^{k+1}}{\bf 1}_{y\leq B}+\frac{b^{k+2}}{y^{m}}{\bf 1}_{B\leq y\leq 2B}, \ \ 0\leq m\leq 1, \ \ k=2.
\ee
(iii) For $k=1$, 
\bea
\label{vhovhohvoh}
\left|\frac{d^m }{dy^m}\frac{\pa P_{B}}{\partial b}\right|\lesssim \frac{by^{1-m}(1+|\log {b(1+y)}|)}{|\log b|}{\bf 1}_{y\leq \frac{B_0}{2}}+\frac{1}{b|\log b| y^{1+m}}{\bf 1}_{\frac{B_0}{2}\le y\leq 2B} &+& \frac 1{by^{1+m}}{\bf 1}_{\frac{B}{2}\le y\leq 2B},\notag\\ &&\hskip -12pc +C(M) \frac {by}{1+y^{2+m}}, \hskip 2pc 0\leq m\leq 3\label{estboddkone}
\eea
and for $0\leq m\leq 1$:
\bea
\label{estpsiboddtildekone}
& & \left|\frac{d^m}{dy^m}\left(\Psi_b-c_bb^2\chi_{\frac{B_0}{4}}\Lambda Q\right)\right|\lesssim \frac{b^2}{y}{\bf 1}_{B\leq y\leq 2B}
+ C(M) b^4\frac{y^{1-m}}{1+y^4}{\bf 1}_{y\leq 2B}\notag \\ &&\hskip 4pc +b^{4}\frac{(1+|\log (by)|)}{|\log b|}y^{1-m}{\bf 1}_{1\leq y\leq \frac{B_0}{2}}+\frac{b^2}{|\log b|y^{1+m}}{\bf 1}_{ \frac{B_0}{2}\leq y\leq 2B},
\eea
\end{proposition}

The main consequence of the localization procedure is first that 
$$Supp(\Lambda P_{B})\subset\{0\leq y\leq 2B\}$$ and hence the possible growth in $b$ of weighted Sobolev norms of $P_B$ may be evaluated explicitely. Second, the localization procedure creates an unavoidable slowly decaying term in the error $\Psi_{B}$ arising from the commutator $[D\Lambda, \chi_B]\sim 1$ and the specific decay of the radiation $T_p$, leading to: 
\be
\label{aysptoticofpsib}
\forall y\in [B,2B], \ \ \Psi_B(y)\sim \left\{\begin{array}{ll} b^{k+2} \ \ \mbox{for k} \ \ even,\\
\frac{b^{k+1}}{y} \ \ \mbox{for k} \ \ odd,\\
\end{array} \right .
\ee
However, according to \fref{estpsibeventilde},  \fref{estpsiboddtilde}, \fref{estpsibeventildektwo}, \fref{estpsiboddtildekone}, $\Psi_B$ is {\it better behaved} on the set where $\chi_B=1$, 
thanks to the extra gains provided by the $T_{p+1}$ terms in Proposition \ref{propqb}. 

\begin{remark}
\label{orthosurprofil}
Observe that for $b<b^*(M)$ small enough, the localization does not destroy the orthogonality relation which we have built into $Q_b$. More precisely, \fref{orthod} ensures: 
\be
\label{orthopb}
\forall b\leq b^*(m), \ \ \forall B\geq \frac{1}{b}, \ \ (P_B-Q,\chi_M\Lambda Q)=0.
\ee
\end{remark}

{\bf Proof of Proposition \ref{lemmapsibtilde}}: First compute from (\ref{defqbtilde}) and (\ref{defpsib}): 
\be
\label{calcuderivpe}
\frac{\partial P_{B}}{\partial b}=\chi_B\frac{\partial Q_b}{\partial b}-\frac{\partial \log B}{\partial b}y\chi'_B(Q_b-\pi),
\ee
\bea
\label{formulapsitilde}
\nonumber \Psi_B & = & \chi_B\Psi_b+\frac{k^2}{y^2}\left\{f(P_{B})-\chi_Bf(Q_b)\right\}-(Q_b-a)\Delta \chi_B-2\chi_B'Q'_b\\
& + & b^2\left\{(Q_b-a)D\Lambda \chi_B+2y^2\chi_B'Q'_b\right\}
\eea 
and thus (\ref{supppsib}) follows from (\ref{defqbtilde}). We now consider separate cases:\\

\noindent
{\it case $k\ge 4$ even}: Recall that $2p=k$ for $k$ even.  From \fref{asymtotthorigin}, (\ref{esttjsharp}), there holds for $y\leq \frac{1}{b}$: 
$$|\frac{\partial P_{B}}{\partial b}|\lesssim b|T_1(y)|\lesssim \frac{by^k}{1+y^{2k-2}}.$$ 
On the other hand, in the region $\frac{1}{b}\leq y\leq 2B$:
$$|\frac{\partial P_{B}}{\partial b}|\lesssim b^{k-1}T_p(y)+\frac{b^{k}}{b}\lesssim b^{k-1}.$$ This proves (\ref{estbodd}) for $m=0$, other cases follow similarily.\\

\noindent
We now estimate $\Psi_B$. For $y\leq B$, $\Psi_B=\Psi_b$ and hence (\ref{estpsiboddtilde}), \fref{estpsibeventildektwo} follow for $y\leq B$ from (\ref{estpsibeven}), \fref{vheohoevbobvob}. For $B\leq y\leq 2B$, we estimate the RHS of (\ref{formulapsitilde}). First: 
$$
\frac{1}{y^2}\left\{|f(P_{B})-\chi_Bf(Q_b)|\right\} \lesssim  \frac{|Q_b-a|}{y^2}{\bf 1}_{B\leq y\leq 2B}\lesssim 
{b^{k+2}}{\bf 1}_{B\leq y\leq 2B}.$$
Similarily, 
$$\left |(Q_b-\pi)\Delta \chi_B-2\chi_B'(Q_b-a)'\right|\lesssim \frac{b^{k}}{B^2}{\bf 1}_{B\leq y\leq 2B}\lesssim b^{k+2}{\bf 1}_{B\leq y\leq 2B},$$
$$b^2\left|(Q_b-a)D\Lambda \chi_B+2y^2\chi_B'Q'_b\right|\lesssim b^{2}b^{k}{\bf 1}_{B\leq y\leq 2B}=b^{k+2}{\bf 1}_{B\leq y\leq 2B}.$$ 
 These estimates imply (\ref{estpsibeventilde}) for $m=0$. The cases $1\leq m\leq 3$ follow similarily and are left to the reader.\\
 The case $k=2$ follows similarily using \fref{asymtotthorigin2}, \fref{esttjsharp'}, \fref{vheohoevbobvob}, this is left to the reader.\\
 
 \noindent
{\it case $k\ge 3$ odd}: Recall that $2p+1=k$ for $k$ odd. From \fref{asymtotthorigin}, (\ref{esttjsharp}), \fref{esttjsharpbis}, the leading order behavior of $\frac{\partial P_{B}}{\partial b}$ in the region $y\leq \frac{1}{b}$ is given by: 
$$|\frac{\partial P_{B}}{\partial b}|\lesssim b|T_1(y)|\lesssim \frac{by^k}{1+y^{2k-2}}.$$ On the other hand, in the region $\frac{1}{b}\leq y\leq 2B$, there holds:
$$|\frac{\partial P_{B}}{\partial b}|\lesssim b^{k-2}T_p(y)+\frac{1}{b}\frac{b^{k-1}}{y}\lesssim \frac{b^{k-2}}{y}.$$ This proves (\ref{estbodd}) for $m=0$, other cases follow similarily.\\

\noindent
We now estimate the error $\Psi_B$ given in (\ref{formulapsitilde}). For $y\leq B$, $\Psi_B=\Psi_b$ and hence (\ref{estpsiboddtilde}) follows for $y\leq B$ from (\ref{estpsibodd}). In the region $B\leq y\leq 2B$, we estimate from (\ref{estbodd}) and $f(\pi)=0$: 
\bee
\frac{1}{y^2}\left\{|f(P_{B})-\chi_Bf(Q_b)|\right\} & \lesssim &  \frac{1}{y^2}\left\{|f(\pi+\chi_B(Q_b-\pi))-f(\pi)|+|f(Q_b)-f(\pi)|\right\}\\
& \lesssim & \frac{|Q_b-\pi|}{y^2}{\bf 1}_{B\leq y\leq 2B}\lesssim \frac{b^{k+1}}{y}{\bf 1}_{B\leq y\leq 2B},
\eee
$$\left |(Q_b-\pi)\Delta \chi_B-2\chi_B'(Q_b-\pi)'\right|\lesssim \frac{b^{k-1}}{B^2y}{\bf 1}_{B\leq y\leq 2B}\lesssim \frac{b^{k+1}}{y}{\bf 1}_{B\leq y\leq 2B},$$
$$b^2\left|(Q_b-\pi)D\Lambda \chi_B+2y^2\chi_B'Q'_b\right|\lesssim \frac{b^{2}b^{k-1}}{y}{\bf 1}_{B\leq y\leq 2B}=\frac{b^{k+1}}{y}{\bf 1}_{B\leq y\leq 2B}.$$ These estimates together with (\ref{estpsibodd}) now imply (\ref{estpsiboddtilde}) for $m=0$. The case $m=1$ follow similarily.\\

\noindent 
{\it case $k=1$}:\,\, We estimate from \eqref{calcuderivpe}:
$$
\frac{\partial P_{B}}{\partial b}=\chi_B\frac{\partial Q_b}{\partial b}-\frac{\partial \log B}{\partial b}\frac{y}{B}\chi'_B(Q_b-\pi),
$$
Therefore,
$$
|\frac{\partial P_{B}}{\partial b}|\le |\frac{\partial (b^2 T_1)}{\partial b}| {\bf 1}_{y\le 2B} + b^{-1} 
|Q_b-\pi|{\bf 1}_{\frac B2 \le y\le 2B}. 
$$
Estimate  \fref{estboddkone} is a direct consequence of the construction of $T_1$ and the 
bound $|Q_b-\pi |\lesssim (1+y)^{-1}$. The derivative estimates follow in a similar fashion.\\

\noindent
We now turn to the estimate of $\Psi_B$. From \fref{formulapsitilde}:
$$\frac{1}{y^2}\left\{|f(P_{B})-\chi_Bf(Q_b)|\right\} \lesssim \frac{|Q_b-\pi|}{y^2}{\bf 1}_{B\leq y\leq 2B}\lesssim \frac{b^{2}}{y}{\bf 1}_{B\leq y\leq 2B},
$$
\begin{align}\label{eq:qpsi}
&\left |(Q_b-\pi)\Delta \chi_B-2\chi_b'(Q_b-\pi)'\right|\lesssim \frac{1}{B^2y}{\bf 1}_{B\leq y\leq 2B}\lesssim \frac{b^{2}}{y}{\bf 1}_{B\leq y\leq 2B},\\
&b^2\left|(Q_b-\pi)D\Lambda \chi_B+2y^2\chi_B'Q'_b\right|\lesssim \frac{b^{2}}{y}{\bf 1}_{B\leq y\leq 2B}.\label{eq:qpsi2}
\end{align} 
These estimates yield (\ref{estpsiboddtildekone}) for $m=0$, the case $m=1$ follows similarily.\\

\noindent
This concludes the proof of Proposition \ref{lemmapsibtilde}.


\section{Decomposition of the flow}
\label{sectionequation}


Having constructed the almost self similar localized profiles $P_{B}$, we introduce a decomposition of the flow: $$u(t,r)=\left(P_{B_1(b(t))}+\e\right)(t,\frac{r}{\lambda(t)})=\left(P_{B_1(b(t))}\right)_{\lambda(t)}+w(t,r)$$ where $$B_1=\frac{|\log b|}{b}.$$ The time dependent parameters $b(t),\lambda(t)$ will be determined from the modulation theory in section \ref{geomdecomp}. The perturbative $w(t)$ is what is refered to in the paper as the ``radiation term''. Since $(P_{B_1})_{\lambda(t)}\in \mathcal H^2_a$, it implies \footnote{Observe that fr $k=1$, $Q_{\lambda(t)}$ \emph{does not} belong to $\mathcal H^2_a$ due its slow convergence at infinity.}  that $w(t,r)\in \mathcal H^2$.\\
We now derive the equations for $w$ and $\e$. Let
\be
\label{defreacledtime}
s(t)=\int_0^t\frac{d\tau}{\lambda(\tau)} \ee
be the rescaled time\footnote{Note that $s(t)$ will be proved to be a global time $s(t)\to +\infty$ as $t\to T$} . We shall make an intensive use of the following rescaling formulas: for
$$u(t,r)=v(s,y), \ \ y=\frac{r}{\lambda}, \ \ \frac{ds}{dt}=\frac{1}{\lambda},$$
\be
\label{dliatationone}
\partial_tu=\frac{1}{\lambda}\left(\partial_sv+b\Lambda v\right)_{\lambda},
\ee
\be
\label{dilationtowo}
\partial_{tt}u=\frac{1}{\lambda^2}\left[\partial_s^2v+b(\partial_sv+2\Lambda\partial_sv)+b^2D\Lambda v+b_s\Lambda v\right]_{\lambda}.
\ee
In particular, using \fref{defpsitilda} and \fref{dilationtowo}, we derive from \fref{equation} the equation for $\e$:
\bea
\label{eqeqb}
\nonumber \partial_s^2\eb+H_{B_1}\eb & = & -\Psi_{B_1}-b_s\Lambda P_{B_1}-b(\partial_sP_{B_1}+2\Lambda\partial_sP_{B_1})-\partial^2_sP_{B_1}\\
& - & b(\partial_s\eb+2\Lambda\partial_s\eb)-b_s \Lambda \eb-\frac{k^2}{y^2}N(\eb)
\eea
where $H_{B_1}$ is the linear operator associated to the profile  $P_{B_1}$
\be
\label{defhbl}
H_{B_1} \eb=-\Delta \eb+b^2D\Lambda\eb+k^2\frac{f'(P_{B_1})}{y^2}\eb,
\ee
and the nonlinearity:
\be
\label{defnw}
N(\eb)=\frac{1}{y^2}\left[f(P_{B_1}+\eb)-f(P_{B_1})-f'(P_{B_1})\eb\right].
\ee
Alternatively, the equation for $w$ given by \fref{defet} takes the form:
$$\partial_t^2w+H_{B_1} w=-\left[\partial_t^2(P_{B_1})_{\lambda}-\Delta(P_{B_1})_{\lambda}+k^2\frac{f((P_{B_1})_{\lambda})}{r^2}\right]-\frac{k^2}{r^2}N(w)$$
with 
\be
\label{defhbltwo}
H_{B_1} w=-\Delta w+k^2\frac{f'((P_{B_1})_{\lambda})}{r^2},
\ee
\be
\label{defnwbis}
N(w)=\frac{1}{r^2}\left[f(P_{B_1}+w)-f(P_{B_1})-f'((P_{B_1})_{\lambda})w\right].
\ee
We then expand using (\ref{dliatationone}), (\ref{dilationtowo}) and (\ref{defpsitilda}):
\bee
\partial_t^2(P_{B_1})_{\lambda}-\Delta(P_{B_1})_{\lambda}+k^2\frac{f((P_{B_1})_{\lambda})}{r^2}& = & \frac{1}{\lambda^2}\left[\partial_{ss}P_{B_1}+b(\partial_sP_{B_1}+2\Lambda\partial_sP_{B_1})+b_s\Lambda P_{B_1}+\Psi_B\right]_{\lambda}\\
& = &  \frac{1}{\lambda^2}\left[b\Lambda\partial_sP_{B_1}+b_s\Lambda P_{B_1}+\Psi_B\right]_{\lambda}+\partial_t\left[\frac{1}{\lambda}(\partial_sP_{B_1})_{\lambda}\right]
\eee
and rewrite the equation for $w$:
\be
\label{eqwfinal}
\partial_t^2w+H_{B_1}w =   -\frac{1}{\lambda^2}\left[b\Lambda\partial_sP_{B_1}+b_s\Lambda P_{B_1}+\Psi_B\right]_{\lambda}-\partial_t\left[\frac{1}{\lambda}(\partial_sP_{B_1})_{\lambda}\right]-\frac{k^2}{r^2}N(w).
\ee
For most of our arguments we prefer to view the linear operator $H_{B_1}$ acting on $w$ in (\ref{eqwfinal}) as a perturbation of the linear operator $H_\lambda$ associated to $Q_\lambda$. Then
\bea
\label{oeioehoe}
& & \partial_t^2w+H_\lambda w=F_{B_1}\\
\nonumber &   = &   -\frac{1}{\lambda^2}\left[b\Lambda\partial_sP_{B_1}+b_s\Lambda P_{B_1}+\Psi_{B_1}\right]_{\lambda}-\partial_t\left[\frac{1}{\lambda}(\partial_s\qbt)_{\lambda}\right] \\
\nonumber & + &  \frac{k^2}{r^2}\left[f'(Q_{\lambda})-f'((P_{B_1})_{\lambda})\right]w-\frac{k^2}{r^2}N(w)
\eea
with 
\be
\label{defh}
H_{\lambda} w=-\Delta w+k^2\frac{f'(Q_{\lambda})}{r^2}.
\ee

\begin{remark}
\label{cboieheo} We note that absence of satisfactory {\it pointwise} in time estimates for the 
$b_{ss}$ type of terms appearing on the RHS of \fref{oeioehoe} (see also \fref{eqeqb}) requires that we rewrite such
terms as full time derivatives and consistently integrate them by parts in all of our estimates.
\end{remark}
Our analysis will require control of ${\mathcal H}^2$ norm of $w$. This will be achieved via energy estimates 
for the function 
\be
\label{defwwinifo}
W=A_\lambda w.
\ee We recall that the operator $A_\lambda$ factorizes the Hamiltonian 
$H_\lambda=A_\lambda^* A_\lambda$ and the function $W$ is a solution of the wave equation
\be
\label{Wequation}
\partial_{tt}W+\widetilde{H}_\lambda W=A_{\lambda}F_{B_1}+\frac{\partial_{tt}V_{\lambda}^{(1)}w}{r}+\frac{2\partial_t\vul\partial_tw}{r}.
\ee 
with the conjugate Hamiltonian $\tilde H_\lambda=A_\lambda A_\lambda^*$, see \eqref{Wequation'}.


\section{Initial data and the bootstrap assumptions}
\label{sectionthree}


In this section we describe the set of estimates which govern the blow up dynamics stated in Theorem \ref{mainthm}. We begin with the prescription of the set ${\mathcal{O}}$ of initial data and 
consequently show that, under bootstrap assumptions, they evolve to a trapped regime leading
to a finite time blow up.\\


\subsection{Description of the set $\mathcal O$ of initial data}


Let us recall the orbital stability statement of Lemma \ref{orbstab}: for all sufficiently small $\eta>0$ such that for $(u_0,u_1)\in H^1_r\times L^2$ with $E(u_0,u_1)<E(Q)+\eta,$
there exists $\lambda(t)>0$ such that the corresponding solution $u(t)$ to \fref{equation} satisfies: 
$$u(t,r)=(Q+\e)(\frac{r}{\lambda(t)}) \  \ \mbox{with} \ \ \|\e(t),\partial_tu\|_{\mathcal H}= o(\eta).$$ This decomposition is not unique. Uniqueness can be achieved, using standard modulation theory, by for example fixing an orthogonality condition on $\e$, see Lemma \ref{orbstab}. The class of initial data which lead to the blow up dynamics of Theorem \ref{mainthm} have energy just above $E(Q)$ and 
are excited in a specific direction of the $Q_b$ deformation of $Q$. 

\begin{definition}[Description of the set of initial data ${\mathcal{O}}$]
\label{defoinitial}
Let $M$ be a sufficiently large constant and let $b_0^*(M)>0$ be small enough. We define ${\mathcal{O}}$ to be the set of initial data $(u_0,u_1)$ of the form:
\be
\label{decompointi}
u_0(r)=\left(P_{B_1(b_0)}\right)_{\lambda_0}+w_0(r)=\left(P_{B_1(b_0)}+\e_0\right)_{\lambda_0},
\ee
\be
\label{dlecompointipartialt}
u_1(r)=\frac{b_0}{\lambda_0}\left(\Lambda P_{B_1(b_0)}\right)(\frac{r}{\lambda_0})+w_1(r),
\ee
where $\e_0$ satisfies the orthogonality condition:
\be
\label{ortheinit}
(\e_0,\chi_M\Lambda Q)=0
\ee
We require that the following bounds are satisfied:
\begin{itemize}
\item Smallness of $b_0$: 
\be
\label{smallbzero}
0<b_0<b_0^*;
\ee
\item Smallness of $\lambda_0$ with respect to $b_0$:
\be
\label{smalllmba}
\lambda^2_0<b^{2k+4}_0;
\ee
\item Smallness of the excess of energy:
\be
\label{eq:sm-en}
\|w_0,w_1\|_{\mathcal H}\lesssim {b_0^{10k}}
\ee
and
\be
\label{oeiohepwe}
\|w_0,w_1\|_{\H^2}\lesssim \frac{b_0^{10k}}{\lambda_0}.
\ee
\end{itemize}
\end{definition}

\begin{remark} Note that by the implicit function theorem 
${\mathcal{O}}$ is a non-empty {\it open} set of $\H^{2}$.
\end{remark}


\subsection{Decomposition of the flow and modulation equations}
\label{geomdecomp}


Let us now consider $(u_0,u_1)\in {\mathcal{O}}$ and let $u(t)$ be the corresponding solution to \fref{equation} with life time $T=T(u_0)\leq +\infty$ defined as the maximal time interval on which 
$u\in{\mathcal C}([0,T),\H^{2}_a)$. It now easily follows from the orbital stability of Lemma \ref{orbstab} that
for any $(u_0,u_1)\in {\mathcal{O}}$ and $t\in [0,T(u_0))$ there exists a unique 
decomposition of the flow $$u(t)=(Q+\e_1)_{\lambda(t)}$$ with $\lambda(t)\in {\mathcal{C}}^2([0,T), \RR^*_+)$ and 
\be
\label{orbstabbound}
\forall t\in [0,T), \ \  |\partial_tu|_{L^2}+|\lambda_t(t)|+\|\e_1(t),0\|_{\mathcal H}\lesssim o(1)_{b^*_0\to 0}
\ee satisfying 
the orthogonality condition 
\be
\label{defhovohfbis}
\forall t\in [0,T), \ \ (\e_1(t),\chi_M\Lambda Q)=0.
\ee 
Based on this decomposition we define 
\be
\label{defbt}
b(t)=-\lambda_t \ \ \mbox{so that} \ \ b(t)=o(1)_{b^*_0\to 0}
\ee
and for $b^*_0$ small enough define the new decomposition with the profile $P_{B_1(b(t)})$ and 
``the excess'' $\e(t,y)=w(t,r)$: 
\be
\label{defet}
u(t,r)=\left(P_{B_1(b(t))}+\e\right)(t,\frac{r}{\lambda(t)})=\left(P_{B_1(b(t))}\right)_{\lambda(t)}+w(t,r).
\ee
Observe from \fref{defhovohfbis} and the choice of gauge \fref{orthopb} in the construction of $Q_b $  that:
\be
\label{orthe}
\forall t\in [0,T), \ \ (\e(t),\chi_M\Lambda Q)=0,\qquad (w(t),(\chi_M\Lambda Q)_{\lambda(t)})=0
\ee
According, to section \ref{sectionequation}, $w,\e$ and $W$ given by \fref{defwwinifo} satisfy respectively the equations \fref{eqeqb}, \fref{oeioehoe} and \fref{Wequation}. The modulation equation for $b$ is based on the orthogonality condition \fref{orthe} and will be derived in section \ref{forstboundnbs}. The precise control of the parameter $b$ is at the heart of our analysis. According to the modulation equation for $\lambda$ \fref{defbt}, the behavior determines the blow up speed and measures the deviation from the self similar blow up.


\subsection{Initial bounds for $(\lambda,b,w)$}


We have now began the process of recasting the original flow for the function $u$ in terms of the dynamics of the new 
variables $(\lambda, b, w)$. Although the equations for $\lambda(t), b(t)$ are yet to be derived, 
we reinterpret the assumptions on the initial data $(u_0,u_1)\in{\mathcal O}$ as assumptions on $(\lambda(0),b(0),w(0), W(0))$ and claim the following initial estimates:

\begin{lemma}[Initial bounds for the $(\lambda,b,w)$ decomposition]
\label{lemmainitialdata}
We have
\be
\label{hofhieoepeih}
\lambda_0=\lambda(0), \ \  b_0-b(0)=0(b_0^{10k}),
\ee
\be
\label{ivnovnhdlbvjepjvepvo}
\|w(0), \partial_tw(0)\|_{\mathcal H}=o(1)_{b^*_0\to 0},
\ee
\be
\label{inoshohgoer}
|b_s(0)|+\lambda_0\|W(0),\partial_tW(0)\|_{\mathcal H}\lesssim \frac{b_0^{k+1}}{|\log b_0|}.
\ee
\end{lemma}

\noindent
{\bf Proof of Lemma \ref{lemmainitialdata}}\\

{\bf step 1} Estimates for $\lambda(0), b(0)$ and spatial derivatives of $w$.\\

Let us 
first show that 
\be
\label{hofhieoepeihbis}
\lambda_0=\lambda(0), \ \  b_0-b(0)=0(b_0^{10k}),
\ee
\be
\label{estwntinino}
\int (\partial_r w(0))^2+\int\frac{(w(0))^2}{r^2}\lesssim b_0^{5k},
\ee
\be
\label{cnoeneone}
\|W(0),0\|_{\mathcal H}\lesssim \frac{b_0^{5k}}{\lambda(0)}.
 \ee
Indeed, first compare \fref{decompointi} and \fref{defet} at $t=0$ to get: $$u_0=(Q+(P_{B_1(b_0)}-Q)+\e_0)_{\lambda_0}=(Q+(P_{B_1(b(0))}-Q)+\e(0))_{\lambda(0)}$$ with $$ ((P_{B_1(b_0)}-Q)+\e_0,\chi_M\Lambda Q)=((P_{B_1(b(0))}-Q)+\e(0),\chi_M\Lambda Q)=0$$ and hence the uniqueness of the geometric decomposition ensures:
\be
\label{cnonconceoeno}
\lambda(0)=\lambda_0\ \  \mbox{and} \ \ \e(0)=\e_0+P_{B_1(b_0)}-P_{B_1(b(0))}.
\ee 
and
\be\label{eq:con}
w(0)=w_0+(P_{B_1(b_0)}-P_{B_1(b(0))})_{\lambda_0}
\ee
We now compute the $\partial_t$ derivative at $t=0$: 
\be
\label{oehoevoe}
\partial_tu(0)=\frac{1}{\lambda_0}\left(b_s(0)\frac{\partial P_{B_1}}{\partial b}+b(0)\Lambda P_{B_1(b(0))}\right)_{\lambda_0}+\partial_tw(0).
\ee
We take a scalar product of this relation with $(\chi_M\Lambda Q)_{\lambda_0}$ and first observe from \fref{orthe} that: $$(\partial_tw,(\chi_M\Lambda Q)_{\lambda})=-\frac{b}{\lambda}(w,\Lambda (\chi_M\Lambda Q)_{\lambda})$$ and hence from \fref{cnonconceoeno}:
\bee
\label{cnockeoneooe}
\nonumber 
\left|(\partial_tw(0),(\chi_M\Lambda Q)_{\lambda_0})\right| & \lesssim &  |b(0)|\lambda_0|(\e_0+ P_{B_1(b_0)}-P_{B_1(b(0))},\Lambda(\chi_M\Lambda Q))\\
& \lesssim &  C(M) \lambda_0|b(0)|(b_0^{10k}+|b^2(0)-b^2_0|).
\eee
The last line uses the initial bound \fref{eq:sm-en} and the results of Proposition \ref{lemmapsibtilde}.\\

\noindent
Furthermore, $$(\frac{\partial P_{B_1}}{\partial b}, \chi_M\Lambda Q)=0$$ and hence from \fref{oehoevoe}: 
\be
\label{nononeoneo}
(\partial_tu(0),(\chi_M\Lambda Q)_{\lambda_0}) = \lambda_0\left[b(0)(\Lambda P_{B_1(b(0))},\chi_M\Lambda Q)+ O(|b(0)|(b_0^{10k}+|b^2(0)-b^2_0|)\right].
\ee
Performing the same computation on \fref{dlecompointipartialt} using \fref{oeiohepwe} yields: 
$$(\partial_t u(0),(\chi_M\Lambda Q)_{\lambda_0})=\lambda_0\left[b_0(\Lambda P_{B_1(b_0)},\chi_M\Lambda Q)+O(b_0^{10k})\right]$$ which together with \fref{nononeoneo} now implies:
$$
b_0-b(0)=0(b_0^{10k}).
$$
This gives \fref{hofhieoepeihbis}. Estimate \fref{estwntinino} now follows by inserting \fref{eq:sm-en} and \fref{hofhieoepeih} into \fref{eq:con}. \\

\noindent
Finally,
\bea
\label{bjebeofo}
\nonumber \|W(0),0\|_{\mathcal H}^2 &= & \int |\partial_r A_{\lambda_0} w(0)|^2+\int\frac{(A_{\lambda_0}w(0))^2}{r^2}\\
&  \lesssim & \frac {\|w(0),0\|_{\mathcal H}^2}{\lambda_0^2}+\|w(0),0\|_{\H^{2}}^2\\
\nonumber &  \lesssim &  \frac {\|w_0,0\|_{\mathcal H}^2}{\lambda_0^2}+\|w_0,0\|_{\H^{2}}^2+\frac{(b_0-b(0))^2}{\lambda_0^2}\lesssim \frac{b_0^{10k}}{\lambda_0^2}
\eea
where we used the uniform boundedness of the $Q_b$ profile in the $\H^2$ norm (note asymptotic behavior \fref{asymtotthorigin}, \fref{asymtotthorigin1} at the origin). Thus \fref{eq:con}, \fref{hofhieoepeih} and the initial bounds \fref{eq:sm-en}, \fref{oeiohepwe}, and \fref{cnoeneone} follow. Note that for $k=1$, the bound \fref{bjebeofo} requires some care and uses the fact that $|V^{(1)}(y)-1|\lesssim y $ for $y\leq 1$ and hence:
\bee
& & \int_{r\leq \lambda_0}|\partial_r A_{\lambda_0}w(0)|^2+\int_{r\leq \lambda_0}\frac{(A_{\lambda_0}w(0))^2}{r^2}\\
& = & \int_{r\leq \lambda_0}\left|\partial_r\left(-\partial_rw(0)+\frac{V_{\lambda_0}^{(1)}}{r}w(0)\right)\right|^2+\int_{r\leq \lambda_0}\frac{1}{r^2}\left|-\partial_rw(0)+\frac{V_{\lambda_0}^{(1)}}{r}w(0)\right|^2\\
& \lesssim & \int_{r\leq \lambda_0}(\partial_r^2w(0))^2+\int_{r\leq \lambda_0}\frac{1}{r^2}\left(\partial_rw(0)-\frac{w(0)}{r}\right)^2+\int_{r\leq \lambda_0}\frac{(w(0))^2}{\lambda_0^2r^2}\\
& \lesssim & \|w(0),0\|_{\H^{2}}^2+\frac {\|w(0),0\|_{\mathcal H}^2}{\lambda_0^2}
\eee
while 
\bee
& & \int_{r\geq \lambda_0}|\nabla A_{\lambda_0}w(0)|^2+\int_{r\geq \lambda_0}\frac{(A_{\lambda_0}w(0))^2}{r^2} \lesssim \int_{r\geq \lambda_0}(\partial_r^2w(0))^2+  \int_{r\geq \lambda_0}\left (\frac{(\nabla w(0))^2}{r^2}+\frac {w(0)^2}{r^4}\right)\\
&\lesssim & \int\left ((\partial_r^2w(0))^2+\frac{(\partial_r w(0))^2}{r^2}+\frac {w(0)^2}{\lambda_0^2 r^2}\right) \lesssim  \|w(0),0\|_{\H^2}^2+\frac {\|w(0),0\|_{\mathcal H}^2}{\lambda_0^2},
\eee
which yield \fref{bjebeofo} for $k=1$.\\

{\bf step 2} Time derivative estimates.\\

From \fref{dlecompointipartialt}, \fref{oehoevoe}, \eqref{hofhieoepeihbis}:
$$ \lambda_0\partial_tw(0)=\left(b_0\Lambda P_{B_1(b_0)}-b(0)\Lambda P_{B_1(b(0))}-b_s(0)\frac{\partial P_{B_1}}{\partial b}\right)_{\lambda_0}+w_1.$$ Therefore, 
$$
\lambda_0 \partial_t W(0) = \lambda_0 A_{\lambda_0} \partial_ t w(0) + \lambda_0 (\partial_t A_\lambda)
w(0) 
$$
Using \eqref{deoperatoaone} and \eqref{deoperatoa} we have
$$
(\partial_t A_\lambda)= \frac {\partial_t V_\lambda^{(1)}}{r}=\frac{k b(0)}{\lambda_0} 
\frac {(\Lambda Q g''(Q))_{\lambda_0}}{r}
$$
This implies from \eqref{eq:sm-en}, \fref{oeiohepwe}, \eqref{eq:con} and \fref{hofhieoepeih}:
$$|\partial_tw(0)|_{L^2}+\lambda_0|\partial_tW(0)|_{L^2}\lesssim (|b_s(0)|+b_0^{10k})
\left(|\frac{\partial P_{B_1}}{\partial b}|_{L^2}+|A\frac{\partial P_{B_1}}{\partial b}|_{L^2}\right)+O(b_0^{4k}).$$
We now derive from Proposition \ref{lemmapsibtilde} the rough bound: 
\be
\label{cnoeoheir}
|A\frac{\partial P_{B_1}}{\partial b}|_{L^2}+|\frac{\partial P_{B_1}}{\partial b}|_{L^2}\lesssim \left\{\begin{array}{ll} 1\ \ \mbox{for} \ \ k\geq 2\\
												            \frac{1}{b_0}\ \ \mbox{for} \ \ k=1
												            \end{array} \right .
\ee
and hence:

\be
\label{nvonvornpnorepn}
|\partial_tw(0)|_{L^2}+\lambda_0|\partial_tW(0)|_{L^2}\lesssim O(b_0^{4k})+\left\{\begin{array}{ll}|b_s(0)|\ \ \mbox{for} \ \ k\geq 2\\
												            \frac{|b_s(0)|}{b_0}\ \ \mbox{for} \ \ k=1
												            \end{array} \right .
\ee
It remains to compute $b_s(0)$. This computation relies on the orthogonality relation \fref{orthe} and is done in full detail in the proof of Proposition \ref{roughboundpointw}. In particular, we may extract from the explicit formula \fref{elgebraba} evaluated at $t=0$ the crude bound:
\begin{align}
\label{cncbeobvi}|b_s||\Lambda Q|_{L^2(y\leq 2M)}^2&\lesssim |(\Psi_{B_1},\chi_M\Lambda Q)|+|b(0)||\partial_tw(0)|_{L^2}||\frac{y^k}{1+y^{2k}}|_{L^2(y\leq 2M)}\\ &
+M^C\left (|A_{\lambda_0} w(0)|_{L^2(y\leq 2M)}+|\frac{w(0)}r|_{L^2(y\leq 2M)}\right)\notag
\end{align}
We now examine separately:\\
{\bf case $k\geq 2$}: We first have from Proposition \ref{lemmapsibtilde}:
$$|(\Psi_{B_1},\chi_M\Lambda Q)|\lesssim M^Cb^{k+2}.$$ We insert 
this together with \fref{hofhieoepeihbis}, \fref{estwntinino}, \fref{cnoeneone} into \fref{cncbeobvi} to get:
$$|b_s(0)|\lesssim |b_0||\partial_tw(0)|_{L^2}+O(b_0^{k+2}).$$ 
Combining this with \fref{nvonvornpnorepn} concludes the proof of \fref{ivnovnhdlbvjepjvepvo}, \fref{inoshohgoer}.\\
{\bf case $k=1$}: From \fref{estpsiboddtildekone}, $$|(\Psi_{B_1},\chi_M\Lambda Q)|\lesssim M^C\frac{b^2}{|\log b|}$$ and hence \fref{hofhieoepeihbis}, \fref{estwntinino}, \fref{cnoeneone} and \fref{cncbeobvi} yield:
$$|\log M||b_s(0)|\lesssim |b_0|\sqrt{\log M}|\partial_tw(0)|_{L^2}+O(\frac{b_0^2}{|\log b_0|}).$$ 
Combining this with \fref{nvonvornpnorepn} now concludes the proof of \fref{ivnovnhdlbvjepjvepvo}, \fref{inoshohgoer} for $M$ large enough and $b_0<b^*_0(M)$ sufficiently small.\\

This concludes the proof of Lemma \ref{lemmainitialdata}.


\subsection{The set of bootstrap estimates}


Let  $K=K(M)>0$ be a large universal constant  to be chosen later, and let $\mathcal E(t),\mathcal E_{\sigma}(t)$ be the global and local energies as defined in \fref{poitnwiseboundWbis}, \fref{lcalizedenrgybis}. From the continuity $u\in {\mathcal C}([0,T),\H^{2})$, the initial bounds \fref{smalllmba} and \fref{ivnovnhdlbvjepjvepvo}, \fref {inoshohgoer} of Lemma \ref{lemmainitialdata}, we may find a maximal time $T_1\in (0,T)$ such that the following estimates hold on $[0,T_1)$:
\begin{itemize}
\item Pointwise control of $\lambda$ by $b$:
\be
\label{controllambda}
\lambda^2<10b^{2k+4}.
\ee
\item Pointwise bound on $b_s$:
\be
\label{poitwisebs}
|b_s|\leq \sqrt{K}\frac{b^{k+1}}{|\log b|}.
\ee
\item {\it Global} ${\mathcal H}^2$ bound:
\be
\label{poitnwiseboundW}
\nonumber \mathcal E(t)\leq K b^{2k+2}.
\ee
\item {\it Local} ${\mathcal H}^2$ bound: 
\be
\label{localizegviojdo}
 {\mathcal E}_{\sigma}(t) \leq K\frac{b^{2k+2}}{(\log b)^2}.
\ee
\end{itemize}
\begin{remark}
The large bootstrap constant $K(M)$ does not depend on the small constant $b_0^*$, which provides
an upper bound for possible values of the parameter $b$. It therefore allows us to assume that 
$$
o(1)_{b^*_0\to 0} K(M)= o(1)_{b^*_0\to 0}.
$$
In particular, if $C(M)$ is an even larger universal constant dependent on $M$ and $K$ and $\eta$ is the constant in the
orbital stability bound \eqref{ortbialstabai}, we may assume that
$$
\eta^{\frac 1{10}} C(M)<1,
$$
\end{remark}
\begin{remark}[Coercivity of $\mathcal E$] The potential part of the energy ${\mathcal E}$ is
the quadratic form of the Hamiltonian $\tilde{H}_\lambda$ given by \fref{defhatsrbis}. As a consequence ${\mathcal E}$, as well as ${\mathcal E}_\sigma$, is coercive. However, the norm under control degenerates at infinity for $k=1$. In fact, from \eqref{cnheoiheoeuy}, \fref{cnheoiheoeuybis}:
\be
\label{contorlocoervcie}
\frac{\mathcal E_{\sigma}}{\lambda^2}\geq \int \sigma_{B_c}\left[(\partial_tW)^2+(\partial_rW)^2+\frac{W^2}{r^2}\right] \ \ \mbox{for} \ \ k\geq 2,
\ee
and thus controls the Hardy norm both at the origin and at infinity, while 
\be
\label{contorlocoervciebis}
\frac{\mathcal E_{\sigma}}{\lambda^2}\geq \int \sigma_{B_c}\left[(\partial_tW)^2+(\partial_rW)^2+\frac{W^2}{r^2(1+\frac{r^2}{\lambda^2})}\right] \ \ \mbox{for} \ \ k=1
\ee
and thus is not as strong at infinity. This difficulty will be handled with the help of logarithmic Hardy inequalities, see Lemma \ref{lemmaloghrdy} in the Appendix. However, logarithmic losses in Hardy type inequalities are 
potentially dangerous, since for $k=1$ all possible gains are themselves merely logarithmic in
the parameter $b$.
This explains why many estimates for $k=1$ will require a very detailed, careful and sometimes subtle analysis, which in particular will keep track of $\log$ losses and $\log\, b$ gains. 
\end{remark}

Our first result is the contraction of the bootstrap regime, described by \eqref{controllambda}-\eqref{localizegviojdo}, under the nonlinear flow. 

\begin{proposition}[Bootstrap control of $\lambda, b_s,W$]
\label{bootstrap}
Assume that $K=K(M)$ in \fref{controllambda}, \fref{poitwisebs}, \fref{poitnwiseboundW}, \fref{localizegviojdo} has been chosen large enough, then $\forall t\in [0,T_1)$, 
\be
\label{controllambdaboot}
\lambda^2\leq b^{2k+4},
\ee
\be
\label{poitwisebsboot}
|b_s|\leq\frac{\sqrt{K}}{2}\frac{b^{k+1}}{|\log b|},
\ee
\be
\label{poitnwiseboundWboot}
\mathcal E(t)\leq \frac{K}{2} b^{2k+2},
\ee
\be
\label{localizegviojdoboot}
 {\mathcal E}_{\sigma}(t)\leq \frac{K}{2}\frac{b^{2k+2}}{(\log b)^2}.
\ee
As a consequence $T_1=T$. Moreover, the solution blows up in finite time $$T<+\infty.$$
\end{proposition}

\begin{remark} The bootstrap bounds of Proposition \ref{bootstrap} are not enough yet to provide a sharp law for the blow up speed. The fact that a sharp description of the singularity formation {\it is not needed} to prove finite time blow up was already central in \cite{MM1}, \cite{MR1}, \cite{R1} and \cite{RS}. This 
conveniently separates  the analysis required for the proof of a finite time blow up and 
an upper bound on the blow up rate from obtaining a lower bound on the blow up rate, which relies on finer dispersive effects. \end{remark}

The next section is devoted to the proof of the key dynamical estimates which imply Proposition \ref{bootstrap}. 


\section{The excess of energy and finite time blow up}


This section is devoted to the proof of the bootstrap bounds \fref{poitnwiseboundWboot}, \fref{localizegviojdoboot}. The proof consists of two steps. First is to derive a crude bound on the blow up speed in the form of a pointwise control on $|b_s|$. This follows directly from the construction of the profile $P_{B_1}$. The second step is a pointwise in time bound on the excess of energy of $W$ in the region containing 
the backward light cone of a future singularity. Combination of these two estimates will establish \fref{poitnwiseboundWboot}, \fref{localizegviojdoboot}. This will be already sufficient  to prove finite time blow up with an explicit  non-sharp upper bound on blow up rate. Note that the statements of a finite time blow up and stability of the blow up regime do not require the knowledge of the precise blow up speed.


\subsection{First bound on $b_s$}
\label{forstboundnbs}

The first step in the proof of the bootstrap estimates \fref{poitnwiseboundWboot}, \fref{localizegviojdoboot} is the derivation of a crude bound on $b_s$ which will allow us to obtain control on the scaling parameter $\lambda$ and to derive suitable energy estimates on the solution. This bound is a simple consequence of the construction of the profile $Q_b$ and the choice of the orthogonality condition \fref{orthe}.\\

Let $M>0$ be a large enough universal constant to be chosen later and $|b|\leq b_0^*(M)$ small enough. Let us start with observing the following orbital stability bound:

\begin{lemma}[Orbital stability bound]
\label{orbitalstability}
There holds:
\be
\label{ortbialstabai}
\forall t\in [0,T_1], \ \ |b|+\|w,\partial_tw\|_{H} <\eta=o(1)_{b^*_0\to 0}.
\ee
\end{lemma}
\begin{remark}
We note that $\|w,\partial_tw\|_{H}$ norm provides an $L^\infty$ bound for $w$ and $\epsilon$
$$
|w(t)|_{L^\infty}=|\epsilon(s)|_{L^\infty}<\eta.
$$
This is a consequence of the simple inequality 
$$
w^2(r)\le\int\left ((\pa_r w)^2+\frac {w^2}{r^2}\right), 
$$
which holds true for smooth functions vanishing at the origin.
\end{remark}
{\bf Proof of Lemma \ref{orbitalstability}} \\

First recall from \fref{orbstabbound}, \fref{defet} that $|b|=|\lambda_t|\lesssim o(1)_{b^*_0\to 0}$ and hence: 
\be
\label{prbtisstabboundone}
\|w,0\|_{H}\lesssim \|\e_1,0\|+\|P_{B_1}-Q,0\|_{H}\lesssim o(1)_{b^*_0\to 0}.
\ee 
It remains to prove the smallness of the time derivative for which we use \fref{orbstabbound}, the estimates of 
Proposition  \ref{lemmapsibtilde}, \fref{cnoeoheir} and the bootstrap bound \fref{poitwisebs} on $b_s$:
\bee
\nonumber \|\partial_tw\|_{L^2} & \lesssim & \|\partial_tu\|_{L^2}+\|b_s\frac{\partial_b P_{B_1}}{\partial b}+b\Lambda P_{B_1}\|_{L^2}\lesssim o(1)_{b_0^*\to 0}+|b_s|\|\frac{\partial_b P_{B_1}}{\partial b}\|_{L^2}\\
& \lesssim & o(1)_{b_0^*\to 0}+|b_s|\left\{\begin{array}{ll} 1 \ \ \mbox{for} \ \  k\geq 2\\
\frac{1}{b} \ \ \mbox{for} \ \ k=1 \end{array} \right . \lesssim o(1)_{b_0^*\to 0}+\sqrt{K(M)}\frac{|b|}{|\log b|}\\
& \lesssim & o(1)_{b_0^*\to 0}
\eee
and \fref{ortbialstabai} follows. This concludes the proof of Lemma \ref{orbitalstability}.\\

We now claim the first refined bound on $b_s$:

\begin{lemma}[First bound on $b_s$]
\label{roughboundpointw}
The following bound on $b_s$ holds true on $[0,T_1)$:
 \be
\label{casekoddbigone}
  |b_s|^2\lesssim \frac{1}{\log  M}\left[\int_{y\leq 2M}|\nabla (A\e)|^2+\int_{y\leq 1}\frac{|A\e|^2}{y^2}\right]+\frac{b^{2k+2}}{|\log b|^2}+b^2 M^C {\mathcal E}.
 \ee
 In particular,
 \be
 \label{estbds}
   |b_s|^2\lesssim \frac{1}{\log M}  {\mathcal E}_\sigma+\frac{b^{2k+2}}{|\log b|^2}+b^2 M^C {\mathcal E}
 \ee
\end{lemma}

\begin{remark} Observe that the upper bound on $b_s$ given by Lemma \ref{roughboundpointw} is sharp for $k=1$ but very lossy for large k compared with the expected behavior $|b_s|\sim b^{2k}$. At this stage, sharp bounds could have been derived by further improving the profile inside the light cone as we did for $k=1,2$, but this is not needed for large $k$.
\end{remark}

{\bf Proof of Lemma \ref{roughboundpointw}}\\  

Let us recall that the equation for $\e$ in rescaled variables is given according to \fref{eqeqb}, \fref{defhbl}, \fref{defnw} by:
\bee
\nonumber \partial_s^2\eb+H_{B_1}\eb & = & -\Psi_{B_1}-b_s\Lambda P_{B_1}-b(\partial_sP_{B_1}+2\Lambda\partial_sP_{B_1})-\partial^2_sP_{B_1}\\
& - & b(\partial_s\eb+2\Lambda\partial_s\eb)-b_s \Lambda \eb-\frac{k^2}{y^2}N(\eb)
\eee
with
$$
H_{B_1} \eb=-\Delta \eb+b^2D\Lambda\eb+k^2\frac{f'(P_{B_1})}{y^2}\eb,
$$
$$
N(\eb)=\frac{1}{y^2}\left[f(P_{B_1}+\eb)-f(P_{B_1})-f'(P_{B_1})\eb\right].
$$
Note that from (\ref{adjoinctionfrimula}), the adjoint of  $H_{B}$ with respect to the $L^2(ydy)$ inner product   is given by: 
\be
\label{deflstar}
H_B^*=H_B+2b^2D.
\ee 
To compute $b_s$ we take the scalar product of (\ref{eqeqb}) with $\chi_M\Lambda Q$. Using the orthogonality relations $$(\e,\chi_M\Lambda Q)=(\pa_s^m (P_{B_1}-Q),\chi_M\Lambda Q)=
0,\qquad \forall m\ge 0$$ we integrate by parts to get the algebraic identity:
\bea
\label{elgebraba}
\nonumber & & b_s\left[(\Lambda P_{B_1},\chi_M\Lambda Q)+b(\frac{\partial P_{B_1}}{\partial b}+2\Lambda \frac{\partial P_{B_1}}{\partial b},\chi_M\Lambda Q)+(\Lambda\ebo,\chi_M\Lambda Q)\right]=-(\Psi_{B_1},\chi_M\Lambda Q)\\
&  - & (\ebo,H_{B_1}^*(\chi_M\Lambda Q))+b(\partial_s\ebo,3\chi_M\Lambda Q+\Lambda(\chi_M\Lambda Q))-k^2(\frac{N(\ebo)}{y^2},\chi_M\Lambda Q).
\eea
On the support of $\chi_M$ and for $b<b_0^*(M)$ small enough, the term $\Lambda Q$ dominates the remaining terms in the expansion
$$
\Lambda P_{B_1}=\Lambda Q_b=\Lambda Q + \sum_{j=1}^{p+1} b^{2j} \Lambda T_j.
$$
The orbital stability bound then yields:
\bee
\nonumber 
|b_s|^2\left (\int_{y\leq M}|\Lambda Q|^2\right)^2 & \lesssim & (\Psi_{B_1},\chi_M\Lambda Q)^2+ \left|(\ebo,H_{B_1}^*(\chi_M\Lambda Q))\right|^2+b^2|(\partial_s\ebo,3\chi_M\Lambda Q+\Lambda(\chi_M\Lambda Q))|^2\\
& + & |(\frac{N(\ebo)}{y^2},\chi_M\Lambda Q)|^2.
\eee
We now treat each term in the above RHS. The last two terms may be estimated in 
a straightforward fashion using the $\chi_M$ localization:
\bee
& & b^2|(\partial_s\ebo,3\chi_M\Lambda Q+\Lambda(\chi_M\Lambda Q))|^2\\
& \lesssim & b^2|(\partial_s\e+by\cdot\nabla \e,3\chi_M\Lambda Q+\Lambda(\chi_M\Lambda Q))|^2+b^4|(y\cdot\nabla\ebo,3\chi_M\Lambda Q+\Lambda(\chi_M\Lambda Q))|^2\\
& \lesssim & b^2\lambda^2M^C\left[|\frac{\partial_tw}{r}|_{L^2}^2+|\frac w{r^2(1+|\log r|)}|_{L^2}^2\right]\lesssim b^2\lambda^2M^C\left[|\partial_tW|_{L^2}^2+|A_{\lambda}^*W|_{L^2}^2\right]
\eee
where we used the estimates of Lemma \ref{lemmahardy1}, Lemma \ref{lemmahardy} and \fref{controldt}. Similarily, from \fref{estdeux}: 
\bee
& & |(\frac{N(\e)}{y^2},\chi_M\Lambda Q)|^2  \lesssim  \left(\int_{y\leq 2M}|\ebo|^2\frac{y}{y^2(1+y^{2})}\right)^2\lesssim M^C |\e|_{L^{\infty}(y\leq 2M)}^2|A^* A\e|_{L^2}^2\\
& \lesssim & M^C|\nabla\e|_{L^2(y\leq 2M)}|\frac{\e}{y}|_{L^2(y\leq 2M)}|A^*A\e|_{L^2}^2\lesssim M^C|A^* A\e|_{L^2}^4 \lesssim  b^2\lambda^2 |A^*_{\lambda}W|_{L^2}^2
\eee
where we used \fref{poitnwiseboundW} in the last step. The first two terms in \fref{fhgfhjskofoh} require more attention. First observe that the $\chi_M$ localization ensures that $$\Psi_B\chi_M=\Psi_b\chi_M.$$ Next, we rewrite the linear term in $\e$ as follows. Using $H=A^*A$ and the cancellation $A(\Lambda Q)=0$ from \fref{cancnelalq} we derive:
\bea
\label{estlineaterm}
\nonumber (\e,H_{B_1}^*(\chi_M\Lambda Q))^2 & = & \left(\e,H(\chi_M\Lambda Q)+2b^2D(\chi_M\Lambda Q)+\frac{1}{y^2}(f'(P_{B_1})-f'(Q))(\chi_M\Lambda Q)\right)^2\\
& \lesssim & \left(A\e,(\Lambda Q)\partial_y\chi_M\right)^2+b^2\lambda^2M^C|A^*_{\lambda}W|_{L^2}^2
\eea
where we used \fref{estdeux} and the rough bound $|P_{B_1}-Q|_{L^{\infty}}\lesssim b$. We have thus obtained the preliminary estimate:
\be
\label{fhgfhjskofoh}
|b_s|^2\left (\int_{y\leq M}|\Lambda Q|^2\right)^2  \lesssim  (\Psi_{B_1},\chi_M\Lambda Q)^2+ \left(A\e,(\Lambda Q)\partial_y\chi_M\right)^2+ b^2\lambda^2M^C\mathcal E.
\ee 
We now separate cases:\\
{\it case $k$ odd, $k\geq 3$}:  We estimate from \fref{estpsibodd}
$$(\Psi_B,\chi_M\Lambda Q)^2\lesssim b^{2k+6}\left (\int\frac{y^{k}}{1+y^{k+2}}\frac{y^{k}}{1+y^{2k}}\right)^2\lesssim b^{2k+6},$$
\bea
\label{vgigvivgiev}
\nonumber \left(A\e,(\Lambda Q)\partial_y\chi_M\right)^2& \lesssim & \left(\int_{y\leq 2M} \frac{(A\e)^2}{y^2}\right)\int_{M\leq y\leq 2M}|\Lambda Q|^2\\
& \lesssim & \frac{1}{M^{2k-3}}\left(\int_{y\leq 2M}|\nabla A\e|^2+\int_{y\leq 1}\left|\frac{A\e}{y}\right|^2\right).
\eea
where we used \eqref{harybis} in the last step. 
This concludes the proof of \fref{casekoddbigone}.\\
{\it case $k$ even, $k\geq 4$}: From \fref{estpsibeven}:
$$(\Psi_B,\chi_M\Lambda Q)^2|\lesssim b^{2k+8}\left (\int\frac{y^{k}}{1+y^{k+1}}\frac{y^{k}}{1+y^{2k}}\right)^2\lesssim b^{2k+8},$$ and \fref{vgigvivgiev} still holds. This concludes the proof of \fref{casekoddbigone}.\\
{\it case $k=2$}: From \fref{vheohoevbobvob}:
\bee
(\Psi_B,\chi_M\Lambda Q)^2\lesssim \left(\int_{y\leq 2M} \left[b^4\Lambda Q+b^{6}\frac{y^{k}}{1+y^{k+1}}\right]\Lambda Q\right)^2\lesssim b^8,
\eee
and \fref{vgigvivgiev} still holds. This concludes the proof of \fref{casekoddbigone}.\\
{\it case $k=1$}: From \fref{estpsiboddkequalone}:
\bee
& & \left(\Psi_B,\chi_M\Lambda Q\right)^2\\
 &\lesssim &\left(\int_{y\leq 2M} \frac{y}{1+y^2}\left[\frac{b^2}{|\log b|}\frac{y}{1+y^2}+b^4y{\bf 1}_{y\leq 1}+b^{4}\frac{(1+|\log (by)|)}{|\log b|}y+\frac{b^4}{(\log M)^2} \frac {M^4}{1+y^4}\right]\right)^2\\
& \lesssim & (\log M)^2\frac{b^4}{|\log b|^2}.
\eee
For the linear term, we use \fref{harybis} to derive:
\bee
\left(A\e,(\Lambda Q)\partial_y\chi_M\right)^2 &\lesssim& \left(\int_{M\leq y\leq 2M} \frac{(A\e)^2}{y^2}\right)\int_{M\leq y\leq 2M}|\Lambda Q|^2\\
& \lesssim & \log M\left(\int_{y\leq 2M}|\nabla (A\e)|^2+\int_{1\leq y\leq 1}|A\e|^2\right).
\eee
It is now crucial to observe the growth on the LHS of \fref{fhgfhjskofoh}, specific to the $k=1$ case: 
$$|b_s|^2\left (\int_{y\leq 2M}|\Lambda Q|^2\right)^2\geq C(\log M)^2|b_s|^2$$ and \fref{casekoddbigone} follows.\\

This concludes the proof of Lemma \ref{roughboundpointw}.


\subsection{Global and local ${\mathcal H}^2$ bounds}  
\label{sectionhtwo}


In this section we establish ${\mathcal H}^2$ type bounds on the solution $w$. The global bound corresponds to the
energy ${\mathcal E}(t)$, while the local bound is connected to the energy ${\mathcal E}_\sigma(t)$ and 
provides an $H^2$ type estimate for the solution in a region slightly larger than the backward light cone from a future singularity. These bounds
rely on  non-characteristic energy type identities for  \fref{Wequation} and specific repulsive properties of the 
time-dependent conjuguate Hamiltonian $\tilde{H}_\lambda$ given by \fref{defhatsrbis}.
 This estimate is the second step in the proof of Proposition \ref{bootstrap}. 
  
\begin{lemma}[$H^2$ type energy inequalities]
\label{propinside}
In notations of \fref{poitnwiseboundWbis}, \fref{lcalizedenrgybis} and for $b<b_0^*(M)$ small enough, we have the 
following inequalities:
\bea
\label{vhoheor}
\nonumber & &\frac{d}{dt}\left\{\frac{\mathcal E}{\lambda^2}+O\left(\frac{|b_s|^2}{\lambda^2}+\frac{|b_s|\sqrt{\mathcal E}}{\lambda^2}+\frac{\eta^{\frac{1}{4}}\mathcal E}{\lambda^2}\right)\right\}\\
& \lesssim &  \frac{b}{\lambda^3}\left[|b_s|^2+b^{2k+2}+ (|b_s|+b^{k+1})\sqrt{\mathcal E} +\eta^{\frac{1}{4}}\mathcal E\right],
\eea
\bea
\label{vbboebvnkonovrpvrvo}
\nonumber & & \frac{d}{dt}\left\{\frac{\mathcal E_{\sigma}}{\lambda^2}+O\left(\frac{|b_s|^2}{\lambda^2}+\frac{|b_s|\sqrt{\mathcal E_{\sigma}}}{\lambda^2}+\frac{b^{\frac{1}{4}}\mathcal E}{\lambda^2}\right)\right\}\\
& \lesssim &  \frac{b}{\lambda^3}\left[|b_s|^2+\frac{b^{2k+2}}{|\log b|^2}+ (|b_s|+\frac{b^{k+1}}{|\log b|})\sqrt{\mathcal E_{\sigma}}+\frac{\mathcal E}{|\log b|^2} \right]
\eea
\end{lemma}

\begin{remark}
\label{remarkkey}  It is critical that the constants involved in the bounds \fref{vhoheor}, \fref{vbboebvnkonovrpvrvo} {\it do not depend on M} provided $b_0<b^*_0(M)$ has been chosen sufficiently  small.
\end{remark}

\begin{remark} Note that the logarithmic gain from the global bound \fref{vhoheor} to the local bound \fref{vbboebvnkonovrpvrvo}  can be turned into polynomial gain for $k\geq 2$.
\end{remark}
 
{\bf Proof of Lemma \ref{propinside} }\\

The proof is a consequence of the energy identity on \fref{Wequation} and the bootstrap control of the geometric parameters. The key is the space-time repulsive properties of the operator $\tilde{H}_\lambda$.\\

{\bf step 1} Algebraic energy identity.\\

We recall the definition of the cut-off function $\sigma_{B_c}$ given by \fref{ceheoehi} and of the localized energy $\mathcal E_{\sigma}$ given by \fref{lcalizedenrgybis}. In the sequel, we shall use the notation $\sigma$ generically for both $\sigma\equiv 1$ and $\sigma\equiv \sigma_{B_c}$ given by \fref{ceheoehi}.\\  
We claim the following algebraic energy identity:
 \bea
 \label{computationeergy}
& &  \nonumber \frac{1}{2}\frac{d}{dt}\left\{\int \sigma\left[(\partial_tW)^2+(\nabla W)^2+\frac{k^2+1+2\vul+\vdl}{r^2}W^2-\frac{4}{r}\partial_t\vul\partial_twW\right]\right\}\\
\nonumber & = & \frac{3b}\lambda\int\frac{\sigma W^2}{r^2}\left[\Lambda Q\left(kg''+\frac{k^2}2(g'g''-gg''')\right)(Q)\right]_{\lambda}- b\int\partial_r \sigma\frac{W^2}{r}(k\Lambda Qg''(Q))_{\lambda}\\
\nonumber & + & \frac{1}{2}\int \partial_t\sigma\left[(\partial_tW)^2+(\partial_rW)^2+\frac{k^2+1+2\vul+\vdl}{r^2}W^2\right]\\
 & - &2 \int \partial_t\sigma \frac{W}{r}\partial_t \vul \partial_tw-\int\partial_r\sigma\partial_rW\partial_tW\\ 
\nonumber   & + & \int\frac{\sigma\partial_{tt}V_{\lambda}^{(1)}}{r}\left[w\partial_tW-2W\partial_tw\right]-2\int\frac{\sigma W}{r}\partial_tV_{\lambda}^{(1)}F_{B_1}+\int \sigma \partial_tW A_{\lambda}F_{B_1}.
 \eea
 
{\it Proof of \fref{computationeergy}}: We proceed with the help of  \fref{defhatsrbis}, \fref{Wequation}:
\bea
\label{dowhohwo}
 & & \frac{1}{2}\frac{d}{dt}\left\{\int  \sigma[(\partial_tW)^2+(\partial_r W)^2+\frac{k^2+1+2\vul+\vdl}{r^2}W^2]\right\}\\
\nonumber &= &  \frac{1}{2}\int \partial_t\sigma[(\partial_tW)^2+(\partial_r W)^2+\frac{k^2+1+2\vul+\vdl}{r^2}W^2]-\int\nabla \sigma\cdot\nabla W\partial_tW\\
\nonumber & + & \int\sigma\partial_tW(\partial_{tt}W+\widetilde H_{\lambda}W)+\frac{1}{2}\int\frac{\sigma W^2}{r^2}\left(2\partial_t \vul+\partial_t\vdl\right)\\
\nonumber & = &  \frac{1}{2}\int \partial_t\sigma[(\partial_tW)^2+(\partial_r W)^2+\frac{k^2+1+2\vul+\vdl}{r^2}W^2]-\int\partial_r \sigma\partial_r W\partial_tW\\
\nonumber & + & \int\sigma\partial_tW\left[A_{\lambda}F_{B_1}+\frac{\partial_{tt}V_{\lambda}^{(1)}w}{r}+\frac{2\partial_t\vul\partial_tw}{r}\right]+\frac{1}{2}\int\frac{\sigma W^2}{r^2}\left(2\partial_t \vul+\partial_t\vdl\right).
\eea
The third term on the last line above requires integration by parts:
\bea
\label{choeohoe}
\nonumber  & & \int \sigma\partial_tW\frac{2\partial_t\vul\partial_tw}{r}\\
\nonumber & = & \frac{d}{dt}\left\{\int \sigma W\frac{2\partial_t\vul\partial_tw}{r}\right\}-2\int\frac{W}{r}\left[\partial_t\sigma\partial_t\vul \partial_tw+\sigma\partial_{tt}\vul\partial_tw+\sigma \partial_t\vul\partial_{tt}w\right]\\
\nonumber & = & \frac{d}{dt}\left\{\int \sigma W\frac{2\partial_t\vul\partial_tw}{r}\right\}- 2\int\partial_t\sigma \frac{W\partial_tw}{r}\partial_t\vul\\
& - & 2\int\frac{\sigma W}{r}\left[\partial_{tt}\vul \partial_tw+\partial_t\vul F_{B_1}\right]+2\int \frac{\sigma W}{r}\partial_t\vul H_{\lambda}w
\eea
where we used (\ref{oeioehoe}) in the last step. We now integrate the last term above  by parts in space using \fref{deoperatoaone}:
\bee
& & 2\int \frac{\sigma W}{r}\partial_t\vul Hw=2\int \frac{\sigma W}{r}\partial_t\vul A^*_{\lambda}W=2\int \frac{\sigma W}{r}\partial_t\vul\left(\partial_rW+\frac{1+\vul}{r}W\right)\\
& = & 2\int \sigma \frac{W^2}{r^2}\left[(1+\vul)\partial_t\vul-\frac{r}{2}\partial_t\partial_r\vul\right]-\int \frac{W^2}{r}\partial_r\sigma\partial_t\vul.
\eee
Inserting this together with \fref{choeohoe} into \fref{dowhohwo} yields:
\bee
\nonumber & & \frac{1}{2}\frac{d}{dt}\left\{\int \sigma[(\partial_tW)^2+(\partial_rW)^2+\frac{k^2+1+2\vul+\vdl}{r^2}W^2-\frac{4}{r}\partial_t\vul\partial_twW]\right\}\\
& = & \int \sigma \frac{W^2}{r^2}\left[\frac{1}{2}(2\partial_t \vul+\partial_t\vdl)+2((1+\vul)\partial_t\vul-\frac{r}{2}\partial_t\partial_r\vul)\right]\\
& + & \int \sigma \partial_tW\left[A_{\lambda}F_{B_1}+\frac{\partial_{tt}V_{\lambda}^{(1)}w}{r}\right]-2\int\sigma \frac{W}{r}\left[\partial_{tt}\vul \partial_tw+\partial_t\vul F_{B_1}\right]\\
& + &  \frac{1}{2}\int \partial_t\sigma[(\partial_tW)^2+(\partial_rW)^2+\frac{k^2+1+2\vul+\vdl}{r^2}W^2]-\int\partial_r \sigma\partial_rW\partial_tW\\
& - & \int \frac{W^2}{r}\partial_r\sigma\partial_t\vul-2\int\partial_t\sigma\partial_t\vul\frac{W\partial_tw}{r}.
\eee
An explicit computation from \fref{deoperatoa}, \fref{defhatsrbis} yields:
\be
\label{cjalkeuoepr}
\partial_t\vul=k\frac b\lambda(\Lambda Qg''(Q))_{\lambda}, \ \ \partial_t\vdl=k^2 \frac {b}{\lambda}(\Lambda Q[g'g''-gg'''](Q))_{\lambda}
\ee
and $$\vul\partial_t\vul-\frac{r}{2}\partial_t\partial_r\vul=\frac{bk^2}{2\lambda}(\Lambda Q(g'g''-gg''')(Q))_{\lambda}=\frac{1}{2}\partial_t\vdl,$$ and \fref{computationeergy} follows.
\begin{remark}\label{rem:V} A fundamental feature of \eqref{computationeergy} is that the first term on the RHS of \eqref{computationeergy} which could not be treated perturbatively {\it has a sign}. Indeed, in the (WM) case, $g(u)=\sin (u)$ and thus from \fref{formulakey}:
$$
\frac{3b}\lambda\int\frac{\sigma W^2}{r^2}\left[\Lambda Q\left(kg''+\frac{k^2}2(g'g''-gg''')\right)(Q)\right]_{\lambda}=-\frac{3k^2b}{\lambda}\int \sigma \frac {W^2}{r^2}\sin^2(Q)<0.
$$
In the (YM), we compute from $g(u)=\frac{1}{2}(1-u^2)$ and \fref{formulakey}:
$$\frac{3b}\lambda\int\frac{\sigma W^2}{r^2}\left[\Lambda Q\left(kg''+\frac{k^2}2(g'g''-gg''')\right)(Q)\right]_{\lambda}=-\frac{3b}{\lambda}\int \sigma \frac {W^2}{r^2}(1-Q)(1-Q^2)<0.$$
\end{remark}
For future reference, we record here an estimate on $\partial_t\vul$:
\be \label{eq:V1}
|\partial_t\vul(r)|\lesssim \frac b\lambda \left (\frac {r^k}{1+r^{2k}}\right)_{\lambda},
\end{equation}
which applies in both the (WM) and (YM) case. In the former, however, we also have a strengthened
estimate 
\be \label{eq:V2}
|\partial_t\vul(r)|\lesssim \frac b\lambda \left (\frac {r^{2k}}{1+r^{4k}}\right)_{\lambda},
\end{equation}
which follows from the vanishing properties of $g(Q)=\sin(Q)$. We can unify them in the following bound
\be \label{eq:V3}
|\partial_t\vul(r)|\lesssim \frac b\lambda \left(\frac {r^{2}}{1+r^{4}}\right)_{\lambda}.
\end{equation}
\vskip 2pc
As a consequence the last term on the LHS of \fref{computationeergy} can be estimated as follows:
\begin{align*}
\left |\int \sigma \frac 2r \pa_t V_\lambda^{(1)} \pa_t w W\right |&\lesssim \frac {b}{\lambda} \left (\int
\frac {(\pa_t w)^2}{r^2}\right)^{\frac 12} \left (\int W^2 (\frac {r^4}{1+r^8})_\lambda\right)^{\frac 12}\\ &\lesssim
C(M) b\left(|\partial_tW|_{L^2}+|A_\lambda^* W|_{L^2}\right) |A_\lambda^* W|_{L^2} \lesssim C(M)  \frac {b}{\lambda} {\mathcal E}\\
& \lesssim \frac{b^{\frac{1}{4}}\mathcal E}{\lambda^2}
\end{align*}
where we used \fref{eq:coerc},  \eqref{controldt}.\\

We now aim at estimating all the terms in the RHS \fref{computationeergy}\\

{\bf step 3} Control of the boundary terms in $\sigma$.\\

We treat the boundary terms in $\sigma$ which appear in the third line of the RHS \fref{computationeergy}. Observe from the explicit choice of $\sigma_{B_c}$ with $B_c=\frac{2}{b}$ and \fref{poitwisebs} that $$\partial_t\sigma_{B_c}=\frac{1}{\lambda}\left[b+\frac{b_s}{b}\right](y\partial_y\sigma)(\frac{r}{\lambda B_c})\leq -\frac{b(1-\eta)}{\lambda}|\partial_y\sigma|(\frac{r}{\lambda B_c}),$$ $$\left|\partial_r\sigma_{B_c}\right|=\frac{1}{\lambda B_c}|\partial_y\sigma|(\frac{r}{\lambda B_c})\leq \frac{b}{2\lambda}|\partial_y\sigma|(\frac{r}{\lambda B_c})$$ and hence $$\partial_t\sigma_{B_c}\leq -\frac 32|\partial_r\sigma_{B_c}|.$$ 
This reflects the fact that $r= C\lambda b^{-1}$ are space-like hypersurfaces for any choice of constant 
$C\ge 1$. Recall also from \fref{cnheoiheoeuy}, \fref{cnheoiheoeuybis} that $$k^2+1+2V_{\lambda}^{(1)}+V_{\lambda}^{(2)}\geq 0$$ and hence:
\bea
\label{controlboundary}
\nonumber & & \frac{1}{2}\int \partial_t\sigma\left[(\partial_tW)^2+(\partial_rW)^2+\frac{k^2+1+2\vul+\vdl}{r^2}W^2\right]-\int\partial_r \sigma\partial_r W\partial_tW\\
& \leq &  -\frac{1}{4}\int \partial_t\sigma\left[(\partial_tW)^2+(\partial_rW)^2\right].
\eea
The other term is estimated by brute force:
\bee
\left|2 \int \partial_t\sigma \frac{W}{r}\partial_t \vul \partial_tw\right| & \lesssim & \frac{b^2}{\lambda^2}\int\frac{W|\partial_t w|}{r}\left(\frac{r^{2}}{1+r^{4}}\right)_{\lambda}\\
& \lesssim & \frac{b^2}{\lambda}\left(\int \frac{(\partial_t w)^2}{r^2}\right)^{\frac{1}{2}}|A^*_{\lambda}W|_{L^2}\lesssim C(M) \frac{b^2}{\lambda^3}\mathcal E\\
& \lesssim & \frac{b}{\lambda^3}b^{\frac 14}\mathcal E
\eee
 where we used \fref{eq:coerc}, \fref{controldt}. Finally, observe that $\Lambda Q g''(Q)\leq 0$ and $\partial_r\sigma\leq 0$ imply 
 that $$- b\int\partial_r \sigma\frac{W^2}{r}(k\Lambda Qg''(Q))_{\lambda}\leq 0.$$

{\bf step 4} $\partial_{tt}\vul $ terms.\\

We compute:
$$
\partial_{tt}\vul = k \frac {b_s+b^2}{\lambda^2} (\Lambda Qg''(Q))_{\lambda} +k^2\frac {b^2}{\lambda^2}
(\Lambda Q(g'(Q)g''(Q)+g(Q) g'''(Q))_{\lambda}
$$
and hence using the bootstrap bound \fref{poitwisebs}:
\be
\label{poitwodevtt}
|\partial_{tt}\vul|\lesssim \frac{|b_s|+b^2}{\lambda^2}\left(\frac{r^{2k}}{1+r^{4k}}\right)_{\lambda}\lesssim \frac{b^2}{\lambda^2}\left(\frac{r^{2k}}{1+r^{4k}}\right)_{\lambda}
\ee 
in the (WM) case and 
\be
\label{poitwodevtt'}
|\partial_{tt}\vul|\lesssim \frac{|b_s|+b^2}{\lambda^2}\left(\frac{r^{2}}{1+r^{4}}\right)_{\lambda}\lesssim \frac{b^2}{\lambda^2}\left(\frac{r^{k}}{1+r^{2k}}\right)_{\lambda}
\ee 
for the (YM) $k=2$ case. We can unify them in the following bound:
\be
\label{eq:V'}
|\partial_{tt}\vul|\lesssim \frac{b^2}{\lambda^2}\left(\frac{r^{2}}{1+r^{4}}\right)_{\lambda}.
\ee 
As a consequence, we obtain using \fref{eq:coerc}, \fref{estdeux}, \fref{controldt}:
\bee
& & \left|\int\frac{\sigma\partial_{tt}V_{\lambda}^{(1)}}{r}\left[w\partial_tW-2W\partial_tw\right]\right|\\
 & \lesssim&  \frac{b^2}{\lambda^2}\left(\int (\partial_t W)^2\right)^{\frac{1}{2}}\left(\int w^2\left(\frac{r^{4}}{r^2(1+r^{8})}\right)_{\lambda}\right)^{\frac{1}{2}}
 +  \frac{b^2}{\lambda^2}\left(\int \frac{(\partial_t w)^2}{r^2}\right)^{\frac{1}{2}}\left(\int W^2\left(\frac{r^{4}}{1+r^{8}}\right)_{\lambda}\right)^{\frac{1}{2}}\\
& \lesssim &   \frac{b^2}{\lambda^2}\|\partial_tW\|_{L^2}\left(\lambda^2\int\frac{\e^2}{y^4(1+|\log y|^2)}\right)^{\frac{1}{2}}+\frac{b^2}{\lambda^2}\left(\|\partial_tW\|^2_{L^2}+\|A^*_{\lambda}W\|^2_{L^2}\right)^{\frac{1}{2}}\left(\lambda^2\|A^*_{\lambda}W\|_{L^2}\right)^{\frac{1}{2}}\\
& \lesssim &  C(M) \frac{b^2}{\lambda}\left[\int (\partial_tW)^2+(A_{\lambda}^*W)^2\right]\lesssim C(M) \frac{b^2}{\lambda^3}\mathcal E\lesssim \frac{b}{\lambda^3}b^{\frac{1}{4}}\mathcal E.
\eee

{\bf step 5} Decomposition of $F_{B_1}$ terms.\\

We now decompose the term involving $F_{B_1}$, given by \fref{oeioehoe} in \fref{computationeergy}, as follows. 
We first write:
\be
\label{firstdecomp}
F_{B_1}=F_1-\partial_tF_2 \ \ \mbox{with} \ \ F_2=\frac{1}{\lambda}(\partial_s P_{B_1})_{\lambda}.
\ee
Recall from Remark \ref{cboieheo} that there is no satisfactory pointwise bound for $b_{ss}$ and hence for $\partial_tF_2$. We thus have to integrate by parts in time:
 \bee
& & -2\int\frac{\sigma W}{r}\partial_tV_{\lambda}^{(1)}F_{B_1}+\int \sigma \partial_tW A_{\lambda}F_{B_1}\\
& = & -2\int\frac{\sigma W}{r}\partial_tV_{\lambda}^{(1)}(F_1-\partial_tF_2)+\int \sigma \partial_tW A_{\lambda}(F_1-\partial_tF_2)\\
& = & \frac{d}{dt}\left\{2\int\frac{\sigma W}{r}\partial_t\vul F_2-\int\sigma\partial_tWA_{\lambda}F_2\right\}\\
& - & 2\int F_2\partial_t(\frac{\sigma W}{r}\partial_t\vul)+\int A_{\lambda}F_2\left(\sigma\partial_{tt}W+\partial_t\sigma \partial_tW\right)\\
& - & 2\int \frac{\sigma W}{r}\partial_t\vul F_1+\int \sigma \partial_tWA_{\lambda}F_1+\int\sigma\partial_tW\frac{\partial_t \vul}{r}F_2\\
\eee
We then use the equation \fref{Wequation} to compute:
\bee
& & \int \sigma A_{\lambda}F_2\partial_{tt}W\\
& = & -\int \sigma A_{\lambda}F_2\tilde{H}_\lambda W+\int\sigma A_{\lambda}F_2\left(A_{\lambda}F_1-A_{\lambda}\partial_tF_2+\frac{\partial_{tt}V_{\lambda}^{(1)}w}{r}+\frac{2\partial_t\vul\partial_tw}{r}\right)\\
& = & -\int (A_{\lambda}^*W)A^*_{\lambda}(\sigma A_{\lambda}F_2)+\int\sigma A_{\lambda}F_2\left(A_{\lambda}F_1+\frac{\partial_t\vul }{r}F_2+\frac{\partial_{tt}V_{\lambda}^{(1)}w}{r}+\frac{2\partial_t\vul\partial_tw}{r}\right)\\
& - & \frac{d}{dt}\left\{\frac{1}{2}\int \sigma (A_{\lambda}F_2)^2\right\} +\frac{1}{2}\int \partial_t\sigma (A_{\lambda} F_2)^2.
\eee
We finally arrive at the following identity:
\bea
\label{prnvoirohr}
& & -2\int\frac{\sigma W}{r}\partial_tV_{\lambda}^{(1)}F_{B_1}+\int \sigma \partial_tW A_{\lambda}F_{B_1}\\
\nonumber & = & \frac{d}{dt}\left\{2\int\frac{\sigma W}{r}\partial_t\vul F_2-\int\sigma\partial_tWA_{\lambda}F_2-\frac{1}{2}\int \sigma (A_{\lambda}F_2)^2\right\}\\
\nonumber& - & 2\int \frac{\sigma W}{r}\partial_t\vul F_1+\int \sigma \partial_tWA_{\lambda}F_1\\
\nonumber& - & \int F_2\left[2\partial_t \sigma \frac{W}{r}\partial_t\vul+\sigma\frac{\partial_tW}{r}\partial_t\vul+2\sigma\frac{W}{r}\partial_{tt}\vul\right]\\
\nonumber& + & \int \sigma A_{\lambda}F_2\left[A_{\lambda}F_1+ \frac{\partial_t\vul }{r}F_2+ \frac{\partial_{tt}V_{\lambda}^{(1)}w}{r}+\frac{2\partial_t\vul\partial_tw}{r}\right]\\
\nonumber& + & \int \partial_t\sigma A_{\lambda}F_2\left[\partial_tW+\frac{1}{2}A_{\lambda}F_2\right] -\int (A_{\lambda}^*W)A^*_{\lambda}(\sigma A_{\lambda}F_2)
\eea
We now treat all terms on the RHS \fref{prnvoirohr}.\\

{\bf step 6} $F_2$ terms.\\

In what follows we use the crude bounds:
\begin{eqnarray}\begin{array}{c}
\label{crudeboundb}
|\partial_bP_{B_1}|\lesssim \frac{y^k}{(1+y^k)|\log b|}{\bf 1}_{y\leq 2B_1} + \frac 1{by^k}{\bf 1}_{\frac{ B_0}{2}\le y\leq 2B_1}, 
\\  |\partial_b\partial_yP_{B_1}|\lesssim\frac{y^{k-1}}{(1+y^k)|\log b|}{\bf 1}_{y\leq 2B_0}+ \frac 1{by^{1+k}}{\bf 1}_{\frac {B_0}{2}\le y\leq 2B_1},
\end{array}
\end{eqnarray}
\vskip .5pc

We treat all $F_2$ terms on the RHS of  \fref{prnvoirohr}.\\
\vskip 1pc

{\it First line in the RHS of \fref{prnvoirohr}}:  The crude bound $|\partial_bP_{B_1}|_{L^{\infty}}\lesssim 1$ follows from
 \fref{crudeboundb}. Therefore, from \fref{contorlocoervciebis}, \fref{eq:V3}:
\bee
 \left|\int\frac{\sigma W}{r}\partial_t\vul F_2\right| & \lesssim &  \frac{b|b_s|}{\lambda^2} \left(\int_{r\leq 2\lambda B_1}\sigma \frac{W^2}{r^2(1+\frac{r^2}{\lambda^2})}\right)^{\frac{1}{2}}\left(\int_{y\leq 2B_1} \left(\frac{r^{4}(1+r^2)}{1+r^{8}}\right)_{\lambda}\right)^{\frac{1}{2}}\\
& \lesssim &\frac{|b_s| b\sqrt{|\log b|}}{\lambda^2}\sqrt{\mathcal E_{\sigma}}\lesssim \frac{|b_s| }{\lambda^2}\sqrt{\mathcal E_{\sigma}}
\eee
\bee
 \left|\int\sigma\partial_tWA_{\lambda}F_2\right|& \lesssim&  \frac{|b_s|}{\lambda}|\sqrt{\sigma}\partial_tW|_{L^2}\left(\int\left (\frac{1}{(1+y^2)\log^2 b} {\bf 1}_{y\le 2B_1} + \frac 1{b^2 y^4} {\bf 1}_{\frac {B_0}{2}\le y\le 2B_1}\right)\right)^{\frac{1}{2}}\\
& \lesssim &  \frac{|b_s|}{\lambda}|\sqrt{\sigma}\partial_tW|_{L^2}\lesssim  \frac{|b_s|}{\lambda^2}\sqrt{\mathcal E_{\sigma}},
\eee
\begin{equation}\label{eq:F2}
\int \sigma (A_{\lambda}F_2)^2\lesssim \frac{|b_s|^2}{\lambda^2} \int\left (\frac{1}{(1+y^2)\log^2 b} {\bf 1}_{y\le 2B_1} + \frac 1{b^2 y^4} {\bf 1}_{\frac {B_0}{2}\le y\le 2B_1}\right)\leq \frac{|b_s|^2}{\lambda^2}.
\end{equation}
{\it Third line in the RHS of \fref{prnvoirohr}}: From \fref{eq:V3}:
\bee
\left|\int F_2\partial_t \sigma \frac{W}{r}\partial_t\vul\right|& \lesssim&  \frac{b^2|b_s|}{\lambda^3}\left(\int_{2\lambda B_c\leq r\leq 3\lambda B_c}\frac{W^2}{r^2(1+(\lambda r)^2)}\right)^{\frac{1}{2}}\left(\int_{y\leq 2B_1} \left(\frac{r^{4}(1+r^2)}{1+r^{8}}\right)_{\lambda}\right)^{\frac{1}{2}}\\
& \leq & \frac{b^2|\log b||b_s|}{\lambda^2}|A_{\lambda}^*W|_{L^2}\leq \frac{b}{\lambda^3}\left(|b_s|^2+\frac{\mathcal E}{|\log b|^2}\right)
\eee
\bee
\left|\int F_2\sigma\frac{\partial_tW}{r}\partial_t\vul\right| &  \lesssim &  \frac{b|b_s|}{\lambda^2}|\sqrt{\sigma}\partial_tW|_{L^2}\left(\int_{y\leq 2B_1}\left(\frac{r^4}{r^2(1+r^8)}\right)_{\lambda}\right)^{\frac{1}{2}}\\
& \leq &  \frac{b|b_s|}{\lambda^2}|\sqrt{\sigma}\partial_tW|_{L^2}\lesssim \frac{b}{\lambda^3}|b_s|\sqrt{\mathcal E_{\sigma}},
\eee
and from \fref{eq:V'}:
\bee
\left|\int F_2\sigma\frac{W}{r}\partial_{tt}\vul\right|& \lesssim & \frac{|b_s|b^2}{\lambda^3}\left(\int_{r\leq 3\lambda B_c}\frac{W^2}{r^2(1+\frac{r^2}{\lambda^2})}\right)^{\frac{1}{2}}\left(\int_{y\leq 2B_1} \left(\frac{r^{4}(1+r^2)}{1+r^{8}}\right)_{\lambda}\right)^{\frac{1}{2}}\\
& \leq & \frac{b^2|\log b||b_s|}{\lambda^2}|A^*_{\lambda}W|_{L^2}\lesssim \frac{b}{\lambda^3}\left(|b_s|^2+\frac{\mathcal E}{|\log b|^2}\right).
\eee
{\it Fourth line in the RHS of \fref{prnvoirohr}}: We leave aside the term involving $F_1$ which will be treated in the next step. From \fref{eq:V3}:
$$
\left|\int \sigma A_{\lambda}F_2 \frac{\partial_t\vul }{r}F_2\right| \lesssim  \frac{b|b_s|^2}{\lambda^5}\int_{y\leq 2B_1}\left(\frac{r^{2}}{r(1+r^{4})}\right)_{\lambda}\leq \frac{b}{\lambda^3}|b_s|^2.
$$
From \fref{crudeboundb}:
\bee
\left| \int \sigma A_{\lambda}F_2\frac{\partial_{tt}V_{\lambda}^{(1)}w}{r}\right| & \lesssim & \frac{|b_s|b^2}{\lambda^4}\left(\int w^2\left(\frac{r^{4}}{r^2(1+r^{7})}\right)_{\lambda}\right)^{\frac{1}{2}} \left(\int_{y\leq 2B_1}\left(\frac{1}{(1+r^{3})}\right)_{\lambda}\right)^{\frac{1}{2}}\\
& \lesssim & \frac{|b_s|b^2}{\lambda^4}\left(\lambda^2\int\frac{\e^2}{y^4(1+|\log y|^2)}\right)^{\frac{1}{2}}\lesssim C(M)\frac{|b_s|b^2}{\lambda^2}|A^*_{\lambda}W|_{L^2}\\
& \lesssim & \frac{b}{\lambda^3}\left(|b_s|^2+\frac{\mathcal E}{|\log b|^2}\right)
\eee
where we used \fref{estdeux} in the last steps. Finally, from \fref{controldt} and with the help of slightly stronger
bounds
\begin{align*}
|\partial_bP_{B_1}|&\lesssim \frac{y^k}{(1+y^k)|\log b|}\frac {(b(1+y))^{\frac 12}}{1+(b(1+y))^{\frac 12}}{\bf 1}_{y\leq 2B_1} + \frac 1{by^k}{\bf 1}_{\frac{ B_0}{2}\le y\leq 2B_1}, 
\\  |\partial_b\partial_yP_{B_1}|&\lesssim\frac{y^{k-1}}{(1+y^k)|\log b|}\frac {(b(1+y))^{\frac 12}}{1+(b(1+y))^{\frac 12}}{\bf 1}_{y\leq 2B_0}+ \frac 1{by^{1+k}}{\bf 1}_{\frac {B_0}{2}\le y\leq 2B_1},
\end{align*}
we obtain
\bee
& & \left| \int \sigma A_{\lambda}F_2\frac{\partial_{t}V_{\lambda}^{(1)}\partial_tw}{r}\right| \\
& \lesssim &\frac{b|b_s|}{\lambda^2}\left(\int \frac{(\partial_tw)^2}{r^2}\right)^{\frac{1}{2}}\left(\int_{y\leq 2B_1}\frac{y^{4}}{(1+y^{8})}
\left (\frac {by}{y^2 \log^2 b} {\bf 1}_{y\le 2B_1} + \frac 1{b^2 y^4} {\bf 1}_{\frac {B_0}{2}\le y\le 2B_1}\right)\right)^{\frac{1}{2}}\\
& \lesssim & C(M)\frac{b^{\frac 32}|b_s|}{|\log b|\lambda^2}\left(|\partial_tW|_{L^2}^2+|A_{\lambda}^*W|_{L^2}^2\right)^{\frac{1}{2}}\lesssim  \frac{b}{\lambda^3}\left(|b_s|^2+\frac{\mathcal E}{|\log b|^2}\right).
\eee

{\it Fifth line in the RHS of \fref{prnvoirohr}}: From \fref{crudeboundb}:
\bee
& & \left|\int\partial_t\sigma A_{\lambda}F_2[\partial_tW+\frac{1}{2}A_{\lambda}F_2]\right|\\
& \lesssim &  \frac{b^{\frac 12}|b_s|}{\lambda^{\frac 52}}|\sqrt{\partial_t\sigma}\partial_tW|_{L^2}\left(\int_{2B_c\le y\leq 3B_c}\left(\frac{1}{(1+r^2)\log^2 b}\right)_{\lambda}\right)^{\frac{1}{2}}+\frac{b^2|b_s|^2}{\lambda^5}\int_{2B_c\leq y\leq 3B_c}\left(\frac{1}{1+r^2}\right)_{\lambda}\\
& \lesssim & \left[\frac{|\sqrt{\partial_t\sigma}\partial_tW|^2_{L^2}}{|\log b|}+\frac{b |b_s|^2}{\lambda^3}\right]
\eee
which is absorbed thanks to \fref{controlboundary}.

For the last term, we need to exploit an additional cancellation in the case $k=1$. We compute from \fref{calcuderivpe}:
\begin{align*}
A^* (\sigma A \pa_b P_{B_1})&= \sigma H  (\pa_b P_{B_1}) +  \pa_y \sigma A \pa_b P_{B_1}\\ &=\sigma H \left (\chi_{B_1} \frac {\pa_b (b^2 T_1)}{\pa b} - \frac {\pa \log B_1}{\pa b} \frac y {B_1} \chi'_{B_1}(Q_b-\pi)\right) +
 \pa_y \sigma A \pa_b P_{B_1}
\end{align*}
Using the estimate \eqref{estboddkone} on $\pa_b P_{B_1}$ and its derivatives 
\begin{align*}
\left|\frac{d^m }{dy^m}\frac{\partial P_{B}}{\partial b}\right|\lesssim \frac{by^{1-m}(1+|\log {by}|)}{|\log b|}{\bf 1}_{y\leq \frac{B_0}{2}}+\frac{1}{b|\log b| y^{1+m}}{\bf 1}_{\frac{B_0}{2}\le y\leq 2B_1} &+ \frac 1{by^{1+m}}{\bf 1}_{\frac{B_1}{2}\le y\leq 2B_1}\\ &+ C(M) \frac b{1+y^{1+m}},
\end{align*}
as well as \eqref{asympttone} for $T_1$, we can easily conclude that 
$$
A^* (\sigma A \pa_b P_{B_1})=\sigma  \chi_{B_1} \frac {\pa_b (b^2 H T_1)}{\pa b} 
+\frac{1}{b|\log b| y^{3}}{\bf 1}_{\frac{B_0}{2}\le y\leq 2B_1} +  \frac{1}{by^{3}}{\bf 1}_{\frac{B_1}{2}\le y\leq 2B_1}\delta_{\sigma\equiv 1}
$$
We use that $HT_1$ verifies the equation
$$
HT_1=-D\Lambda Q+ c_b \Lambda Q \chi_{\frac {B_0}4},
$$
which immediately implies from $D\Lambda (\frac{1}{y})=0$ that $|D\Lambda Q|\lesssim \frac y{1+y^4}$ and
$$
 \frac {\pa_b (b^2 H T_1)}{\pa b}\le \frac {by}{1+y^3} + \frac {by}{(1+y^2)|\log b|} \chi_{\frac {B_0}2}
$$
As a consequence,
\bea\label{eq:A*A}
\nonumber |A^* (\sigma A \pa_b P_{B_1})| & \lesssim & \sigma\left[\frac {by}{1+y^3} {\bf 1}_{y\le 2B_1} + \frac {by}{(1+y^2)|\log b|} {\bf 1}_{y\le {2B_1}}\right]\\
& +  &\frac{1}{b|\log b|y^{3}}{\bf 1}_{\frac{B_0}{2}\le y\leq 2B_1}+\frac{1}{by^{3}}{\bf 1}_{\frac{B_1}{2}\le y\leq 2B_1}\delta_{\sigma\equiv 1}
\eea
For $\sigma\equiv 1$, this yields:
\bee
& & \nonumber \left|\int(A_{\lambda}^*W)A^*_{\lambda}(\sigma A_{\lambda} F_2)\right|\\
& \lesssim &  \frac{|b_s|}{\lambda ^2}|A_{\lambda}^*W|_{L^2} \left(\int_{y\leq 2B_1}\frac {b^2 y^2}{(1+y^6)} +\frac {b^2 y^2}{(1+y^4)(\log b)^2} +\frac 1{b^2 y^6}{\bf 1}_{\frac{B_1}{2}\leq y\leq 2B_1}\right)^{\frac{1}{2}}\\
 & \lesssim & \frac{b|b_s|}{\lambda ^2}|A_{\lambda}^*W|_{L^2}\lesssim \frac{b}{\lambda ^3}|b_s|\sqrt{\mathcal E}.
 \eee
For $\sigma\equiv \sigma_{B_c}$, observe that \fref{eq:A*A} on the set $y\le B_0/2$ is an improvement relative to a more straightforward estimate 
$$
|A^* (\sigma A \pa_b P_{B_1})| \lesssim \frac{by(1+|\log {by}|)}{(1+y^2)|\log b|}{\bf 1}_{y\leq \frac{B_0}{2}}+\frac{1}{b|\log b| y^{3}}{\bf 1}_{\frac{B_0}{2}\le y\leq 2B_1} +  C(M)\frac{b}{1+y^3}{\bf 1}_{y\leq 2B_0}
$$
which follows from \eqref{estboddkone}. Such an estimate would imply that 
$$
\int |A^* (\sigma A \pa_b P_{B_1})|^2\lesssim b^2 |\log b|,
$$
as opposed to the improved bound 
\be\label{eq:imp-b}
\int |A^* (\sigma A \pa_b P_{B_1})|^2\lesssim b^2
\ee
We also note that \eqref{eq:A*A} and thus \eqref{eq:imp-b} follow similarily from 
Proposition \ref{lemmapsibtilde} for all $k\ge 2$.
Hence:
\bee
& & \nonumber \left|\int(A_{\lambda}^*W)A^*_{\lambda}(\sigma A_{\lambda} F_2)\right| \lesssim \frac{|b_s|}{\lambda ^2}|A_{\lambda}^*W|_{L^2} \left(\int _{\frac{ B_0}{2}\le y\le 2B_1}\frac 1{b^2|\log b|^2 y^6}\right)^{\frac{1}{2}}\\
& + &  \frac{|b_s|}{\lambda ^2}|\sqrt{\sigma}A_{\lambda}^*W|_{L^2} \left(\int_{y\leq 2B_1}\frac {b^2 y^2}{(1+y^6)} {\bf 1}_{y\le 2B_1}+\frac {b^2 y^2}{(1+y^4)(\log b)^2} {\bf 1}_{y\le 2B_1} \right)^{\frac{1}{2}}\\
 & \lesssim & \frac{b|b_s|}{\lambda ^2}\left(| \sqrt{\sigma}A_{\lambda}^*W|_{L^2}+\frac{|A^*W|_{L^2}}{|\log b|}\right)\lesssim \frac{b}{\lambda ^3}\left(|b_s|\sqrt{\mathcal E_{\sigma}}+|b_s|^2+\frac{\mathcal E}{|\log b|^2}\right)
\eee
In the last step, we used the inequality
\be
\label{estpotentialvone}
(1+V^{(1)})^2\lesssim k^2+1+2V^{(1)}+V^{(2)},
\ee 
which can be verified by a direct computation. Hence: 
\bee
\int \sigma (A_{\lambda}^*W)^2 & = & \int\sigma\left[\partial_rW+\frac{1+V_{\lambda}^{(1)}}{r}W\right]^2\lesssim \int \sigma\left[(\partial_rW)^2+\frac{k^2+1+2V_{\lambda}^{(1)}+V_{\lambda}^{(2)}}{r^2}W^2\right]\\
& \lesssim & \lambda^{-2}\mathcal E_{\sigma}.
\eee

{\bf step 7} $F_1$ terms.\\

We now turn to the control of $F_1$ terms appearing in the RHS \fref{prnvoirohr}. For this, we first split $F_1$ into four different components: 
\be
\label{contribfone}
F_1=F_{1,1}+F_{1,2}+F_{1,3}-\frac{1}{\lambda^2}(\Psi_{B_1})_{\lambda}
\ee 
with $$F_{1,1}=-\frac{1}{\lambda^2}\left[b\Lambda \partial_sP_{B_1}+b_s\Lambda P_{B_1}\right]_{\lambda},  \ \ F_{1,2}=\frac{k^2}{r^2}\left[f'(Q)-f'(P_{B_1})\right]_{\lambda}w, \ \ F_{1,3}=\-\frac{k^2}{r^2}N(w).$$ 

{\it $F_{1,1}$ terms}: We estimate from Proposition \ref{propqb}
\be
\label{estimatelambdpbone}
|\Lambda P_{B_1}|\lesssim \frac{b y^k}{(1+y^k)|\log b|}{\bf 1}_{y\leq 2B_1} +\frac {y^k}{1+y^{2k}}{\bf 1}_{y\le 2B_1}
\ee which together with \fref{crudeboundb} yields:
$$|\frac{d^m}{dy^m}F_{1,1}|\lesssim \frac{|b_s|}{\lambda^{2}}\left(\frac{b y^{k-m}}{(1+y^k)|\log b|}{\bf 1}_{y\leq 2B_1} +\frac {y^{k-m}}{1+y^{2k}}{\bf 1}_{y\le 2B_1}\right), \ \ 0\leq m\leq 1.
$$
Next, the cancellation $A(\Lambda Q)=0$ implies  the bound 
$$
|A\Lambda P_{B_1}|\lesssim \frac {by^{k-1}}{(1+y^k)|\log b|} {\bf 1}_{y\le 2B_1} +\frac 1{y^{k+1}} {\bf 1}_{\frac {B_0}{2}\le y\le 2B_1}
$$
and thus:
$$|A_{\lambda} F_{1,1}|\lesssim \frac{|b_s|}{\lambda^3}\left(\frac {by^{k-1}}{(1+y^k)|\log b|} {\bf 1}_{y\le 2B_1} +\frac 1{y^{k+1}} {\bf 1}_{\frac {B_0}{2}\le y\le 2B_1}\right).$$
From \fref{contorlocoervciebis}, \fref{eq:V3}:
\bee
& & \left|\int \frac{\sigma W}{r}\partial_t\vul F_{1,1}\right|\\
& \lesssim & \frac{|b_s|b}{\lambda^3}\left(\int \sigma\frac{W^2}{r^2(1+\frac{r^2}{\lambda^2})}\right)^{\frac{1}{2}}\left(\int_{y\le 2B_1} \left(\frac{r^4(1+r^2)}{1+r^8}\left[\frac{b^2}{|\log b|^2}+\frac{r^{2k}}{1+r^{4k}}\right]\right)_{\lambda}\right)^{\frac{1}{2}}\\
& \lesssim & \frac{b}{\lambda^3}|b_s|\sqrt{\mathcal E_{\sigma}}
\eee
\bee
\left|\int \sigma \partial_tWA_{\lambda} F_{1,1}\right| & \lesssim & \frac{|b_s|}{\lambda^3}|\sqrt{\sigma}\partial_tW|_{L^2}\left(\int_{y\leq 2B_1} \left(\frac{b^2}{(1+r^2)\log^2 b} + \frac 1{r^4} {\bf 1}_{\frac{B_0}{2}\le r\le 2B_1} \right)_{\lambda}\right)^{\frac{1}{2}}\\
& \lesssim & \frac{b |b_s|}{\lambda^2}|\sqrt{\sigma}\partial_tW|_{L^2}\lesssim \frac{b}{\lambda^3}|b_s|\sqrt{\mathcal E_{\sigma}},
\eee
\begin{align*}
& \left|\int \sigma A_{\lambda}F_2A_{\lambda}F_{1,1}\right|  \lesssim  \frac{|b_s|^2}{\lambda^3}\\ &\times
 \int\frac 1{y^2} \left(\frac {by}{(1+y)|\log b|} {\bf 1}_{y\le 2B_1} + \frac 1y {\bf 1}_{\frac {B_0}{2}\le y \le 2B_1} \right) 
\left(\frac {y}{(1+y)|\log b|} {\bf 1}_{y\le 2B_1} + \frac 1{by} {\bf 1}_{\frac {B_0}{2}\le y \le 2B_1} \right)\\&
 \lesssim  \frac{|b_s|^2}{\lambda^3}\int_{y\leq 2B_1} \left(\frac{b}{(1+y^2)(\log b)^2}\right)\lesssim  \frac{b }{\lambda^3}|b_s|^2.
\end{align*}

{\it $F_{1,2}$ terms}: Take note that the term $F_{1,2}$ is not localized inside the ball $y\leq 2B_1$. 

We first recall the estimate:
$$|P_{B_1}-Q|=|(1-\chi_{B_1})(a-Q)+\chi_{B_1} (Q_b-Q)|\lesssim C(M)\frac{b^2y^{k}}{1+y^{2k-2}}{\bf 1}_{y\leq 2B_1}+\frac{1}{y^k}{\bf 1}_{y\geq \frac{B_1}{2}},$$
which follows from Proposition \ref{propqb}. It implies:
$$|f'(P_{B_1})-f'(Q)|\lesssim |P_{B_1}-Q| \int_0^1 |f''(\tau P_{B_1}+(1-\tau)Q)|d\tau
\lesssim C(M)\frac{b^2y^2}{1+y^2}{\bf 1}_{y\leq 2B_1}+\frac{1}{y^2}{\bf 1}_{y\geq \frac{B_1}{2}}.$$ 
In the last inequality we used that for the (WM) problem $|f''(\pi+R)|\lesssim R$ while the (YM) bound 
$|f''(R)|\lesssim 1$ only applies to the case $k=2$. 
Hence from \fref{estdeux}:
\bea
\label{estalhjdfitwo}
\nonumber \int|A_{\lambda}F_{1,2}|^2 & \lesssim &\frac{C(M)}{\lambda^3} \int_{y\leq 2B_1} b^4\left(\frac{y}{y^2(1+y^2)}\right)^2|\e|^2+\frac{C(M)}{\lambda^3} \int_{y\leq 2B_1} b^4\left(\frac{y^2}{y^2(1+y^2)}\right)^2|A\e|^2\\ 
\nonumber &+&\frac{1}{\lambda^3} \int_{y\geq \frac{B_1}{2}}\frac{|\e|^2}{y^{10}}+\frac{1}{\lambda^3} \int_{y\geq \frac{B_1}{2}}\frac{|A\e|^2}{y^8}\\
 & \lesssim & \frac{C(M)}{\lambda^3} \left[b^4|A^*A\e|_{L^2}^2+b^5|A^*A\e|_{L^2}^2\right] \lesssim  C(M)\frac{b^4}{\lambda^2}|A^*_{\lambda}W|_{L^2}^2.
\eea
This implies:
\bee
\left|\int \frac{\sigma W}{r}\partial_t\vul F_{1,2}\right|& \lesssim & \frac b\lambda\int\frac{|wW|}{r}\left(\frac{y^2}{1+y^4}\right)_{\lambda}\frac{1}{r^2}\left[C(M)\frac{b^2y^2}{1+y^2}{\bf 1}_{y\leq B}+\frac{1}{y^2}{\bf 1}_{y\geq B}\right]_\lambda\\
& \lesssim &C(M) \frac{b^3}{\lambda^3}\int\frac{|\e A\e|}{1+y^5}\\
& \lesssim & C(M)\frac{b^3}{\lambda^3}\left(\int \frac{|A\e|^2}{1+y^5}\right)^{\frac{1}{2}}\left(\int \frac{|\e|^2}{1+y^5}\right)^{\frac{1}{2}}\lesssim C(M) \frac{b^3}{\lambda}|A^*_{\lambda}W|_{L^2}^2\\
& \lesssim & C(M) \frac{b^3}{\lambda^3}\mathcal E\lesssim \frac{b}{\lambda^3}b\mathcal E
\eee
from \fref{estdeux}. Similarily, from \fref{estalhjdfitwo}:
$$\left|\int \sigma \partial_tW A_{\lambda} F_{1,2}\right| \lesssim C(M)\frac{b^2}{\lambda} |A^*_{\lambda}W|_{L^2}|\sqrt{\sigma}\partial_tW|_{L^2}\lesssim C(M)\frac{b^2}{\lambda^3}\sqrt{\mathcal E\mathcal E_{\sigma}}\lesssim \frac{b}{\lambda^3}b^{\frac 12}\mathcal E.$$ Finally, from \fref{estalhjdfitwo}:
\bee
 \left|\int \sigma A_{\lambda}F_2A_{\lambda}F_{1,2}\right| & \lesssim &  C(M)\frac{b^2}{\lambda^2}|A^*_{\lambda}W|_{L^2}|b_s|\left(\int_{y\leq 2B_1}\frac{1}{1+y}\right)^{\frac{1}{2}}\\
 & \lesssim & C(M)\frac{b^{\frac{3}{2}}\sqrt{|\log b|}|b_s|}{\lambda^2}
|A^*_{\lambda}W|_{L^2}\lesssim b\frac{b^{\frac{1}{4}}|b_s|}{\lambda^3}\sqrt{\mathcal E}\\
& \lesssim & \frac{b}{\lambda^3}(|b_s|^2+b^{\frac 12}\mathcal E).
\eee

{\it $F_{1,3}$ terms}: We now turn to the control of the nonlinear term. In this section we will also use the
bootstrap assumption \fref{poitnwiseboundWboot} in the form:
\be\label{eq:boot-A}
\lambda^2 \left (|A^*W|^2_{L^2}+ |\pa_t W|^2_{L^2}\right) \le  C b^4
\ee
for some positive constant $C$. We may assume that $C$ is dominated by the constant $C(M)$, which in turn,
as before, can be assumed to satisfy $C(M)<\eta^{-\frac 1{10}}$.

\vskip 1pc

We claim the following preliminary nonlinear estimates: 
\be
\label{estnonlinearglobal}
 \int\frac{|w|^4}{r^4}\leq \eta^{\frac 12} |A^*_{\lambda}W|_{L^2}^2     
\ee
and 
\be
\label{estnonlineargloballocalized}
\int_{r\leq 3\lambda B_1}\frac{|w|^4}{r^4}\leq   b^{\frac{3}{2}}|A^*_{\lambda}W|_{L^2}^2.
\ee
{\it Proof of \fref{estnonlinearglobal}, \fref{estnonlineargloballocalized}}: We rewrite 
$$ \int\frac{|w|^4}{r^4}=\frac{1}{\lambda^2}\int \frac{|\e|^4}{y^4}$$ and split the integral in three zones. Near the origin, we rewrite: 
$$A\e=-\partial_y\e+\frac{V^{(1)}}{y}\e=-y\partial_y\left(\frac{\e}{y}\right)+\frac{V^{(1)}-1}{y}\e$$ from which: $$\int_{y\leq 1}\left|\partial_y\left(\frac{\e}{y}\right)\right|^2\lesssim \int_{y\leq 1}\frac{(A\e)^2}{|y|^2}+\int_{y\leq 1}\frac{|V^{(1)}-1|^2}{y^4}\e^2.
$$
We now estimate for $k\geq 2$ from \fref{eq:coerc}, \fref{estun}:
\bee
\int_{y\leq 1}\frac{(A\e)^2}{|y|^2}+\int_{y\leq 1}\frac{|V^{(1)}-1|^2}{y^4}\e^2 & \lesssim &  \int\frac{(A\e)^2}{|y|^2}+\int_{y\leq 1}\frac{\e^2}{y^4}\lesssim C(M)\int\frac{(A\e)^2}{y^2}\\
& \lesssim & C(M)|A^*A\e|_{L^2}^2.
\eee
In the $k=1$ case, we use the cancellation $|V^{(1)}(y)-1|\lesssim y$ (in fact $y^2$)
for $y\leq 1$ and \fref{eq:coerc}, \fref{estdeux}:
$$
 \int_{y\leq 1}\frac{(A\e)^2}{|y|^2}+\int_{y\leq 1}\frac{|V^{(1)}-1|^2}{y^4}\e^2  \lesssim \int_{y\leq 1}\frac{(A\e)^2}{|y|^2}+\int_{y\leq 1}\frac{\e^2}{y^2}\lesssim C(M)|A^*A\e|_{L^2}^2.
$$
We thus conclude from the standard interpolation estimates 
\bea
\label{hohoovnoeno}
\nonumber \int_{y\leq 1}\frac{(\e)^4}{y^4} & \lesssim & \left[\int_{y\leq 2}\left|\partial_y\left(\frac{\e}{y}\right)\right|^2+\int_{y\leq 2}\frac{(\e)^2}{y^2}\right]\int_{y\leq 2}\frac{(\e)^2}{y^2}\lesssim |A^*A\e|_{L^2}^4\\
& \lesssim & C(M) b^4 |A^*A\e|_{L^2}^2\lesssim b^{\frac{3}{2}}|A^*A\e|_{L^2}^2
\eea 
where we used \fref{estdeux} and \eqref{eq:boot-A} in the last step. For $1\leq y\leq 4B_1$, we have from  \fref{estmedium}, \fref{estdeux} and \eqref{eq:boot-A} that:
 
\be
\label{estlinfty}
|\e|_{L^{\infty}(1\leq y\leq 4B_1)}^2\lesssim B_1^2|\log b|^2|A^*A\e|_{L^2}^2\lesssim C(M) b^{2}|\log b|^{4}
\ee 
 and hence:
 \bea
 \label{vboeneoe}
\nonumber \int_{1\leq y\leq 4B_1} \frac{|\e|^4}{y^4} &\lesssim& |\e|_{L^{\infty}(1\leq y\leq 4B_1)}^2\int_{y\leq 4B_1}\frac{\e^2}{y^4}\lesssim C(M) b^{2}|\log b|^{6}|A^*A\e|_{L^2}^2\\
& \leq & C(M) b^{\frac{5}{3}}|A^*A\e|_{L^2}^2\lesssim b^{\frac{3}{2}}|A^*A\e|_{L^2}^2
\eea
where we used \fref{estdeux}. This concludes the proof of \fref{estnonlineargloballocalized}. It remains to control the integral in \fref{estnonlinearglobal} for $y\geq 4B$. For $k\geq2$, we have from \fref{estun}, the orbital stability bound
\eqref{ortbialstabai} and \eqref{eq:coerc}: $$\int\frac{|\e|^4}{y^4}\lesssim |\e|_{L^{\infty}}^2\int\frac{|\e|^2}{y^4}\lesssim C(M) \eta |A^*A\e|_{L^2}^2$$ which yields \fref{estnonlinearglobal} for $k\geq 2$. For $k=1$, we need to deal with the logarithmic losses in \fref{estdeux} and have to sharpen the control. We argue as follows. Let $\psi_{B_1}(y)=\psi(\frac{y}{B_1})$ be a cut-off function with $\psi(y)=0$ for $y\leq 1$ and $\psi(y)=1$ for $y\geq 2$. We compute: 
\bee
\int \psi_{B_1}\frac{(\e)^4}{y^4}& = & -\frac{1}{2}\int \psi_{B_1}(\e)^4\partial_y\left(\frac{1}{y^2}\right) dy=\frac{1}{2}\int\frac{1}{y^3}\left[(\e)^4\partial_y\psi_{B_1}+4\psi_{B_1}(\e)^3\partial_y\e\right]\\
& \leq & C\int_{B_1\leq y\leq 2B_1}\frac{(\e)^4}{y^4}+2\int \psi_{B_1}\frac{(\e)^3}{y^3}\left[\frac{V_1}{y}\e-A\e\right]\\
& \leq & C\int_{B_1\leq y\leq 2B_1}\frac{(\e)^4}{y^4}+C|\e|^2_{L^{\infty}}\int\psi_{B_1}\frac{|\e|^2}{y^5}-2\int\psi_{B_1}\frac{(\e)^3}{y^3}\left[\frac{1}{y}\e+A\e\right]\\
& \leq & C(M) \eta^2|A^*A\e|_{L^2}^2-2\int\psi_{B_1}\frac{(\e)^3}{y^3}\left[\frac{1}{y}\e+A\e\right]
\eee
where we used that $|V_1(y)+1|\lesssim \frac{1}{y}$ (in fact $\frac 1{y^2}$) for $y\geq 1$, the orbital stability bound 
\eqref{ortbialstabai}and \fref{estdeux}, \fref{vboeneoe}. We now use H\"older and Sobolev inequalities to derive:
\bee
\label{hoheioh}
\nonumber 3\int \psi_B\frac{(\e)^4}{y^4}&  \lesssim &C(M) \eta^2|A^*A\e|_{L^2}^2+ \int \psi_{B_1}\frac{(\e)^3}{y^3}|A\e|\lesssim \eta|A^*A\e|_{L^2}^2+ \left(\int \psi_{B_1}\frac{(\e)^4}{y^4}\right)^{\frac{3}{4}}|A\e|_{L^4}\\
& \lesssim & \eta|A^*A\e|_{L^2}^2+ \left(\int \psi_{B_1}\frac{(\e)^4}{y^4}\right)^{\frac{3}{4}}|A\e|_{L^2}^{\frac{1}{2}}|\nabla (A\e)|^{\frac{1}{2}}_{L^2}\\
& \lesssim &  \eta \left(\int \psi_{B_1}\frac{(\e)^4}{y^4}+|A^*A\e|_{L^2}^2\right)
\eee
where we used the orbital stability bound \eqref{ortbialstabai}which implies $$|A\e|_{L^2}^2\lesssim |\nabla \e|^2_{L^2}+|\frac{\e}{y}|_{L^2}^2\lesssim \eta^2.$$ This concludes the proof of the global bound \fref{estnonlinearglobal} for $k=1$.\\

We now claim the following controls:
\be
\label{estfunldeux}
\int|F_{1,3}|^2\lesssim \eta^{\frac 12} |A^*_{\lambda}W|_{L^2}^2, 
\ee
\be
\label{estpartialfun}
\int_{r\leq 3\lambda B}|F_{1,3}|^2\lesssim b^{\frac{3}{2}}|A^*_{\lambda}W|_{L^2}^2,
\ee
\be
\label{estpartialtfun}
\int|\partial_tF_{1,3}|^2\lesssim \frac{b^2 b^{\frac 32} }{\lambda^2}|A_{\lambda}^*W|^2_{L^2},
\ee
{\it Proof of \fref{estfunldeux}, \fref{estpartialfun}, \fref{estpartialtfun}}: First recall the formula: 
$$F_{1,3}=\frac{k^2}{r^2}\left[f((P_{B_1})_{\lambda}+w)-f((P_{B_1})_{\lambda})-f'((P_{B_1})_{\lambda})w\right].$$ We thus derive the crude bound $$|F_{1,3}|\lesssim \frac{|w|^2}{r^2}$$ and hence \fref{estfunldeux}, \fref{estpartialfun} directly follow from \fref{estnonlinearglobal}, \fref{estnonlineargloballocalized}. Next, we compute:
\bee
\partial_tF_{1,3} & = & \frac{k^2}{r^2}\partial_t(P_{B_1})_{\lambda}\left[f'((P_{B_1})_{\lambda}+w)-f'((P_{B_1})_{\lambda})-f''((P_{B_1})_{\lambda}w\right]\\
& + & \frac{1}{r^2}\partial_tw\left[f'((P_{B_1})_{\lambda}+w)-f'((P_{B_1})_{\lambda})\right]
\eee
which yields the bound:
$$|\partial_tF_{1,3}|\lesssim \frac{1}{r^2}|\partial_t(P_{B_1})_{\lambda}||w|^2+\frac{1}{r^2}|f''((P_{B_1})_{\lambda})||w||\partial_tw|+\frac{1}{r^2}|\partial_tw||w|^2.$$
We now square this identity, integrate and estimate all terms. From \fref{crudeboundb}, \fref{estimatelambdpbone},  \eqref{poitwisebs}:
$$|\partial_t(P_{B_1})_{\lambda}|_{L^{\infty}}\lesssim\frac{1}{\lambda}\left|\left( b_s\frac{\partial P_{B_1}}{\partial b}+b\Lambda P_{B_1}\right)_{\lambda}\right||
\lesssim \frac{|b_s|+b}{\lambda}\lesssim \frac{b}{\lambda},$$ and thus from \fref{estnonlinearglobal}:
$$
\int \frac{|\partial_t(P_{B_1})_{\lambda}|^2|w|^4}{r^4}\lesssim \frac{ b^2}{\lambda^2}\int_{r\leq 2\lambda B_1}\frac{|w|^4}{r^4}\lesssim \frac{b^2b^{\frac 32}}{\lambda^2}|A^*_{\lambda}W|_{L^2}^2,
$$
Next, we have from \fref{controldt}: 
\be
\label{estp[onctap}
|\partial_tw|^2_{L^{\infty}}\lesssim |\nabla \partial_tw|_{L^2}|\frac{\partial_tw}{r}|_{L^2}\lesssim C(M) \left (|A^*_{\lambda}W|^2_{L^2}+|\pa_t W|^2_{L^2}\right).
\ee 
For $k\geq 2$, we then use the fact that $|f''(P_{B_1})|\lesssim 1$ and is supported in $y\le 2B_1$, with 
additional help of \eqref{eq:boot-A} and \fref{estun} followed by \eqref{contorlocoervcie} to estimate:
\bee
\int \frac{1}{r^4}|f''((P_{B_1})_{\lambda})|^2|w|^2|\partial_tw|^2 & \lesssim & C(M) \left (|A^*_{\lambda}W|^2_{L^2}+|\pa_t W|^2_{L^2}\right)\frac{1}{\lambda^2}\int_{y\le 2B} \frac{1}{y^4}|\e|^2\\
& \lesssim & C(M) \frac{b^4}{\lambda^4} |A^*A\e|_{L^2}^2\lesssim \frac{b^2b^{\frac 32}}{\lambda^2}|A^*_{\lambda}W|_{L^2}^2.
\eee 
For $k=1$, we use the improved bound  $|f''(P_{B_1}(y))|\lesssim \frac{y}{1+y^2}$ and  \fref{estdeux}:
\bee
\int\frac{1}{r^4}|f''((P_{B_1})_{\lambda})|^2|w|^2|\partial_tw|^2 & \lesssim & C(M) \left (|A^*_{\lambda}W|^2_{L^2}+|\pa_t W|^2_{L^2}\right)\frac{1}{\lambda^2}\int_{1\leq y\le 2B} \frac{y^2}{y^4(1+y^4)}|\e|^2\\
& \lesssim & C(M) \frac{b^4}{\lambda^4} |A^*A\e|_{L^2}^2\lesssim \frac{b^2b^{\frac 32}}{\lambda^2}|A^*_{\lambda}W|_{L^2}^2.
\eee Finally, from \fref{estnonlinearglobal} and \fref{estp[onctap}:
$$
\int\frac{|\partial_tw|^2|w|^4}{r^4} \lesssim C(M)\eta^{\frac 12} \left (|A^*_{\lambda}W|^2_{L^2}+|\pa_tW|^2_{L^2}\right)
\int\frac{|w|^4}{r^4}\lesssim C(M)\frac{b^4}{\lambda^2}|A^*_{\lambda}W|^2_{L^2}\lesssim  \frac{b^2b^{\frac 32}}{\lambda^2}|A^*_{\lambda}W|_{L^2}^2.
$$
This concludes the proof of \fref{estpartialtfun}.\\

We are now in position to control the $F_{1,3}$ terms in \fref{prnvoirohr}. First from \fref{estfunldeux}, \fref{harfylog}:
\bea
\label{coeohei}
\nonumber \left|\int \frac{W}{r}\partial_t\vul F_{1,3}\right|  & \lesssim & \frac {b}{\lambda^2} |F_{1,3}|_{L^2}\left(\int \left(\frac{r^4}{r^2(1+r^8)}\right)_{\lambda}W^2\right)^{\frac{1}{2}}\\
& \lesssim & \eta^{\frac 14} \frac{ b}{\lambda}|A_{\lambda}^*W|^2_{L^2}\lesssim  \frac{ b}{\lambda^3} \eta^{\frac 14}\mathcal E,
\eea
\bea
\label{dhohofhohics}
\nonumber \left|\int \sigma_{B_c} \frac{W}{r}\partial_t\vul F_{1,3}\right| &  \lesssim & \frac{b}{\lambda^2}|F_{1,3}|_{L^2(r\leq 3\lambda B_c)}\left(\int\left(\frac{r^4}{r^2(1+r^8)}\right)_{\lambda}W^2\right)^{\frac{1}{2}}\\
&  \lesssim & \frac{ bb^{\frac{3}{4}}}{\lambda}|A_{\lambda}^*W|^2_{L^2}\lesssim \frac{b}{\lambda^3}b^{\frac{3}{4}}\mathcal E
\eea
The second term in \fref{prnvoirohr} requires an integration by parts in time:
\bee
&  & \int \sigma \partial_tWA_{\lambda}F_{1,3}= \frac{d}{dt}\left\{\int \sigma WA_{\lambda}F_{1,3}\right\}-\int W\left[\sigma A_{\lambda}\partial_tF_{1,3}+\sigma \frac{\partial_t\vul}{r}F_{1,3}+\partial_t\sigma A_{\lambda}F_{1,3}\right]\\
& = &  \frac{d}{dt}\left\{\int F_{1,3}\left[\sigma A_{\lambda}^*W+\partial_r\sigma W\right]\right\}-\int\partial_tF_{1,3}\left[\sigma A_{\lambda}^*W+\partial_r\sigma W\right]-\int \sigma\frac{W}{r}\partial_t\vul F_{1,3}\\
& - & \int F_{1,3}\left[\partial_t\sigma A_{\lambda}^*W+\partial^2_{rt}\sigma W\right].
\eee
{\it case $\sigma\equiv 1$}: From \fref{estfunldeux}:
$$\left|\int F_{1,3}A_{\lambda}^*W\right|\lesssim \eta^{\frac 14} |A^*_{\lambda}W|_{L^2}^2\lesssim \frac{\eta^{\frac{1}{4}}\mathcal E}{\lambda^2}.$$ From \fref{estpartialtfun} and \fref{coeohei}:
$$\left|\int\partial_t F_{1,3}A_{\lambda}^*W\right|+\left|\int \frac{W}{r}\partial_t\vul F_{1,3}\right| \lesssim \eta^{\frac 14} \frac{b}{\lambda}|A^*_{\lambda}W|_{L^2}^2\lesssim \frac{b}{\lambda^3} \eta^{\frac 14} \mathcal E.$$ 
{\it case $\sigma=\sigma_{B_c}$}: From \fref{estpartialfun},
$$
\left|\int \sigma_{B_c} F_{1,3}A_{\lambda}^*W\right|  \lesssim  b^{\frac{3}{4}}|A^*_{\lambda}W|_{L^2}^2\lesssim \frac{b^{\frac{3}{4}}\mathcal E.}{\lambda^2}.$$ From \fref{estpartialfun} and \fref{harfylog}: 
\bee
\left|\int F_{1,3}\partial_r\sigma_{B_c} W\right| & \lesssim & b^{\frac{3}{4}}|A^*_{\lambda}W|_{L^2}\left(\int_{\lambda B_c\leq r\leq 3\lambda B_c}\frac{W^2}{r^2}\right)^{\frac{1}{2}}\lesssim b^{\frac{3}{4}}|\log b||A^*_{\lambda}W|^2_{L^2}\\
& \lesssim &  b^{\frac{1}{2}}|A^*_{\lambda}W|^2_{L^2}\lesssim \frac{b^{\frac{1}{2}}\mathcal E}{\lambda^2}.
\eee
Arguing similarily from \fref{estpartialtfun} and \fref{harfylog} yields:
\bee
\left|\int\partial_tF_{1,3}\left[\sigma_{B_c} A_{\lambda}^*W+\partial_r\sigma_{B_c} W\right]\right| & \lesssim &  \frac{bb^{\frac{3}{4}}}\lambda |A_{\lambda}^*W|_{L^2}^2+\frac{bb^{\frac{3}{4}}}\lambda |A_{\lambda}^*W|_{L^2}\left(\int_{\lambda\leq r\leq 3\lambda B}\frac{W^2}{r^2}\right)^{\frac{1}{2}}\\
& \lesssim & \frac{bb^{\frac{1}{2}}}\lambda |A^*_{\lambda}W|^2_{L^2}\lesssim \frac{b}{\lambda^3}b^{\frac{1}{2}}\mathcal E.
\eee
From \fref{dhohofhohics}:
$$\left|\int \sigma_{B_c} \frac{W}{r}\partial_t\vul F_{1,3}\right|  \lesssim \frac{ bb^{\frac{3}{4}}}{\lambda}|A_{\lambda}^*W|^2_{L^2}\lesssim  \frac{b}{\lambda^3}b^{\frac{3}{4}}\mathcal E.$$ From \fref{estpartialfun}:
$$\left|\int F_{1,3}\partial_t\sigma A_{\lambda}^*W\right|\lesssim \frac{b}{\lambda} \left(\int_{r\leq 3\lambda B_c}|F_{1,3}|^2\right)^{\frac{1}{2}}|A_{\lambda}^*W|_{L^2}\lesssim \frac{bb^{\frac{3}{4}}}{\lambda}|A_{\lambda}^*W|^2_{L^2}\lesssim  \frac{b}{\lambda^3}b^{\frac{3}{4}}\mathcal E,$$
\bee
\left|\int F_{1,3}\partial^2_{tr}\sigma W\right| & \lesssim & \frac{b}{\lambda} \left(\int_{r\leq 3\lambda B_c}|F_{1,3}|^2\right)^{\frac{1}{2}}\left(\int_{\lambda \leq r\leq 3\lambda B_c}\frac{W^2}{r^2}\right)^{\frac{1}{2}}|\lesssim \frac{bb^{\frac{3}{4}}|\log b|}{\lambda}|A_{\lambda}^*W|^2_{L^2}\\
& \lesssim & \frac{bb^{\frac{1}{2}}}{\lambda}|A_{\lambda}^*W|^2_{L^2}\lesssim  \frac{b}{\lambda^3}b^{\frac{1}{2}}\mathcal E.
\eee
The last $F_{1,3}$ term to bound in \fref{prnvoirohr} is estimated for the either choice of $\sigma\equiv 1$ and 
$\sigma=\sigma_{B_c}$ with the help of \eqref{eq:imp-b}, 
using \fref{estpartialfun} and the fact that $F_2$ is supported in $y\le 2B_1$:
\bee
\left|\int \sigma A_{\lambda}F_2A_{\lambda}F_{1,3}\right| & = & \left|\int F_{1,3}A^*_{\lambda}(\sigma A_{\lambda}F_2)\right|\lesssim \frac{|b_s|}{\lambda^3}|F_{1,3}|_{L^2(r\leq 2\lambda B_1)}
|A^*_{\lambda}(\sigma A_{\lambda} \partial_b P_{B_1})|_{L^2(r\leq 2\lambda B_1)}\\
& \lesssim & \frac{|b_s|bb^{\frac{3}{4}}}{\lambda^2}|A_{\lambda}^*W|_{L^2}\lesssim\frac{b}{\lambda^3} |b_s|b^{\frac 34}\sqrt{\mathcal E}\lesssim \frac{b}{\lambda^3}(|b_s|^2+b^{\frac 32}\mathcal E).
\eee

{\bf step 8} $F_1$ terms involving $\Psi_{B_1}$.\\

We now turn to the control of the leading order term on the RHS of \fref{prnvoirohr} which is given by 
$\Psi_{B_1}$ in the decomposition \fref{contribfone}. These estimates will be sensitive to the choice of 
$\sigma\equiv 1$ or $\sigma=\sigma_{B_c}$ with a decisive improvement in the latter case. Indeed, $$\sigma_{B_c}\Psi_{B_1}=\sigma_{B_c}\Psi_b$$ 
As a consequence, the slowly decaying leading order flux terms, localized around $y\sim B_1$, in the estimates of Proposition \ref{lemmapsibtilde} disappear.\\
{\it case $k$ even, $k\geq 4$}: We estimate from \fref{estpsibeventilde}, \fref{contorlocoervcie}:
\bea
\label{choiheohfktwo}
\nonumber & & \left|\int\frac{\sigma W}{r}\partial_t\vul\frac{(\Psi_{B_1})_{\lambda}}{\lambda^2}\right|  \lesssim  \frac b{\lambda^2}\left(\int\frac{W^2}{r^2}\right)^{\frac 12}\left (\int |\Psi_{B_1}|^2\frac{y^{4}}{1+y^{8}}\right)^{\frac 12}\\
\nonumber & \lesssim &  \frac{b}{\lambda^2} |A_\lambda^* W|_{L^2} \left (\int \frac{y^{4}}{1+y^{8}}\left[\frac{b^{k+4}y^k}{1+y^{k+1}}{\bf 1}_{y\leq 2B_1}+b^{k+2}{\bf 1}_{B_1\leq y\leq 2B_1}\right]^2\right)^{\frac 12}\\
& \lesssim &   \frac{b^{k+3} }{\lambda^2} |A_\lambda^* W|_{L^2}\lesssim \frac{b}{\lambda^3}b^{k+2}\sqrt{\mathcal E}
\eea
Next, there holds from \fref{estpsibeventilde}:
\bee
\int \left[A_{\lambda}\left(\frac{(\Psi_{B_1})_{\lambda}}{\lambda^2}\right)\right]^2 &\lesssim & \frac{1}{\lambda^4}\int_{y\leq 2B} \frac{1}{y^2}\left[\frac{b^{k+4}y^k}{1+y^{k+1}}{\bf 1}_{y\leq 2B_1}+b^{k+2}{\bf 1}_{B_1\leq y\leq 2B_1}\right]^2\\
& \lesssim & \frac{b^{2k+4}}{\lambda^4},
\eee
$$
\int \sigma_{B_c}\left[A_{\lambda}\left(\frac{(\Psi_{B_1})_{\lambda}}{\lambda^2}\right)\right]^2\lesssim  \frac{1}{\lambda^4}\int_{y\leq 2B_c} \frac{1}{y^2}\left[\frac{b^{k+4}y^k}{1+y^{k+1}}{\bf 1}_{y\leq 2B_1}\right]^2
\lesssim  \frac{b^{2k+8}}{\lambda^4},
$$
from which:
$$
\left|\int \partial_tWA_{\lambda}\left(\frac{(\Psi_{B_1})_{\lambda}}{\lambda^2}\right)\right|\lesssim \frac{b^{k+2}}{\lambda^2}|\partial_t W|_{L^2}\lesssim \frac{b}{\lambda^3}b^{k+1}\sqrt{\mathcal E} ,$$
$$
\left|\int \sigma \partial_tWA_{\lambda}\left(\frac{(\Psi_{B_1})_{\lambda}}{\lambda^2}\right)\right|\lesssim \frac{b^{k+4}}{\lambda^2}| \sqrt{\sigma}\partial_t W|_{L^2}\lesssim \frac{b}{\lambda^3}b^{k+3}\sqrt{\mathcal E_{\sigma}} .$$ Finally, we derive from \eqref{crudeboundb} the crude bound valid for all $k\geq 1$:
\be
\label{estalambdafdeux}
|A_{\lambda}F_2|\lesssim \frac{|b_s|}{\lambda^2}\left(\frac{1}{1+y}{\bf 1}_{y\leq 2B_1}\right)
\ee
which yields:
\bee
\left|\int A_{\lambda}F_2A_{\lambda}\left(\frac{(\Psi_{B_1})_{\lambda}}{\lambda^2}\right)\right| &\lesssim & \frac{|b_s|}{\lambda^3}\int_{y\leq 2B_1}\frac{1}{y(1+y)}\left[\frac{b^{k+4}y^k}{1+y^{k+1}}{\bf 1}_{y\leq 2B_1}+b^{k+2}{\bf 1}_{B_1\leq y\leq 2B_1}\right]\\
& \lesssim & \frac{b^{k+2}|b_s|}{\lambda^3} \lesssim \frac{b}{\lambda^3}[b^{2k+2}+|b_s|^2],
\eee
\bee
\left|\sigma_{B_c} \int A_{\lambda}F_2A_{\lambda}\left(\frac{(\Psi_{B_1})_{\lambda}}{\lambda^2}\right)\right|&\lesssim & \frac{|b_s|}{\lambda^3}\int_{y\leq 2B_c}\frac{1}{y(1+y)}\left[\frac{b^{k+4}y^k}{1+y^{k+1}}\right]\\
& \lesssim & \frac{b^{k+4}|b_s|}{\lambda^3} \lesssim \frac{b}{\lambda^3}[b^{2k+6}+|b_s|^2].
\eee
{\it case $k$ odd, $k\geq 3$}: We estimate from \fref{estpsiboddtilde}:
\bee
& &\hskip -5pc \left|\int\frac{\sigma W}{r}\partial_t\vul\frac{(\Psi_{B_1})_{\lambda}}{\lambda^2}\right|  
\lesssim  \frac b{\lambda^2}\left(\int\frac{W^2}{r^2}\right)^{\frac 12}\left (\int |\Psi_{B_1}|^2\frac{y^{4}}{1+y^{8}}\right)^{\frac 12}\\
\nonumber & \lesssim &  \frac{b}{\lambda^2} |A_\lambda^* W|_{L^2} \left (\int \frac{y^{4}}{1+y^{8}}\left[\frac{b^{k+3}y^k}{1+y^{k+2}}{\bf 1}_{y\leq 2B_1}+\frac{b^{k+1}}{y}{\bf 1}_{B_1\leq y\leq 2B_1}\right]^2\right)^{\frac 12}\\
& \lesssim &   \frac{b^{k+3} }{\lambda^2} |A_\lambda^* W|_{L^2}\lesssim  \frac{b}{\lambda^3}b^{k+2} \sqrt{\mathcal E}
\eee   
Next,  from \fref{estpsiboddtilde} there holds:
\bee
\int \left[A_{\lambda}\left(\frac{(\Psi_{B_1})_{\lambda}}{\lambda^2}\right)\right]^2 &\lesssim & \frac{1}{\lambda^3}\int_{y\leq 2B_1} \frac{1}{y^2}\left[b^{k+3}\frac{y^{k}}{1+y^{k+2}}{\bf 1}_{y\leq B_1}+\frac{b^{k+1}}
{1+y}{\bf 1}_{B_1\leq y\leq 2B_1}\right]^2\\
& \lesssim & \frac{b^{2k+4}}{\lambda^3},
\eee
$$
\int \sigma_{B_c}\left[A_{\lambda}\left(\frac{(\Psi_{B_1})_{\lambda}}{\lambda^2}\right)\right]^2\lesssim  \frac{1}{\lambda^3}\int_{y\leq 2B_1} \frac{1}{y^2}\left[\frac{b^{k+3}y^k}{1+y^{k+2}}{\bf 1}_{y\leq B}\right]^2
\lesssim  \frac{b^{2k+6}}{\lambda^3},
$$
from which:
$$
\left|\int \partial_tWA_{\lambda}\left(\frac{(\Psi_{B_1})_{\lambda}}{\lambda^2}\right)\right|\lesssim \frac{b^{k+2}}{\lambda^2}|\partial_t W|_{L^2}\lesssim \frac{b}{\lambda^3}b^{k+1}\sqrt{\mathcal E},$$
$$
\left|\int \sigma_{B_c} \partial_tWA_{\lambda}\left(\frac{(\Psi_{B_1})_{\lambda}}{\lambda^2}\right)\right|\lesssim \frac{b^{k+3}}{\lambda^2}|\sqrt{\sigma_{B_c}}\partial_t W|_{L^2}\lesssim \frac{b}{\lambda^3}b^{k+2}\sqrt{\mathcal E_{\sigma}}.$$ Finally, from \eqref{estalambdafdeux}:
\bee
& & \left| \int A_{\lambda}F_2A_{\lambda}\left(\frac{(\Psi_{B_1})_{\lambda}}{\lambda^2}\right)\right|\\
& \lesssim & \frac{|b_s|}{\lambda^3}\int_{y\leq 2B_1}\frac{1}{y(1+y)}\left[\frac{b^{k+3}y^k}{1+y^{k+2}}{\bf 1}_{y\leq 2B_1}+\frac{b^{k+1}}{y}{\bf 1}_{B_1\leq y\leq 2B_1}\right]\\
& \lesssim & \frac{b^{k+2}|b_s|}{\lambda^3}\lesssim \frac{b}{\lambda^3}[b^{2k+2}+|b_s|^2].
\eee
\bee
 \left| \sigma \int A_{\lambda}F_2A_{\lambda}\left(\frac{(\Psi_{B_1})_{\lambda}}{\lambda^2}\right)\right|& \lesssim & \frac{|b_s|}{\lambda^3}\int_{y\leq 2B_1}\frac{1}{y(1+y)}\left[\frac{b^{k+3}y^k}{1+y^{k+2}}\right]\\
& \lesssim & \frac{b^{k+3}|b_s|}{\lambda^3}\lesssim \frac{b}{\lambda^3}[b^{2k+4}+|b_s|^2].
\eee

{\it case $k=2$}: The chain of estimates \fref{choiheohfktwo} is still valid even taking into account the term $c_bb^4\Lambda Q$ in \fref{estpsibeventildektwo} and leads to:
$$\left|\int\frac{\sigma W}{r}\partial_t\vul\frac{(\Psi_{B_1})_{\lambda}}{\lambda^2}\right| \lesssim   \frac{b^{k+3} }{\lambda^2} |A_\lambda^* W|_{L^2}\lesssim \frac{b}{\lambda^3}b^{k+2} \sqrt{\mathcal E}.$$
Next, we use in a {\it crucial way} the cancellation $$A(\Lambda Q)=0$$ to conclude from \fref{estpsibeventildektwo} that for $k=2$:
\bee
\int \left[A_{\lambda}\left(\frac{(\Psi_{B_1})_{\lambda}}{\lambda^2}\right)\right]^2 &\lesssim & \frac{1}{\lambda^3}\int_{y\leq 2B_1} \frac{1}{y^2}\left[C(M) b^{k+4}\frac{y^{k}}{1+y^{k+1}}{\bf 1}_{y\leq 2B_1}+b^{k+2}{\bf 1}_{B_1\leq y\leq 2B_1}\right]^2\\
& \lesssim & \frac{b^{2k+4}}{\lambda^4},
\eee
$$
\int \sigma_{B_c}\left[A_{\lambda}\left(\frac{(\Psi_{B_1})_{\lambda}}{\lambda^2}\right)\right]^2\lesssim  \frac{1}{\lambda^3}\int_{y\leq 2B_1} \frac{1}{y^2}\left[C(M)\frac{b^{k+4}y^k}{1+y^{k+1}}{\bf 1}_{y\leq 2B_1}\right]^2
\lesssim  \frac{b^{2k+7}}{\lambda^4}.
$$
Observe that without the cancellation we would expect to have an additional term $\frac {b^4 y^k}{1+y^{k+2}} {\bf 1}_{y\le 2B_1}$, which would not disappear after application of the cut-off function $\sigma_{B_c}$ 
and therefore destroy the extra gain in the localized estimate. Thus:
$$
\left|\int \partial_tWA_{\lambda}\left(\frac{(\Psi_{B_1})_{\lambda}}{\lambda^2}\right)\right|\lesssim \frac{b^{k+2}}{\lambda^2}|\partial_t W|_{L^2}\lesssim \frac{b}{\lambda^3}b^{k+1}\sqrt{\mathcal E} ,$$
$$
\left|\int \sigma \partial_tWA_{\lambda}\left(\frac{(\Psi_{B_1})_{\lambda}}{\lambda^2}\right)\right|\lesssim \frac{b^{k+7/2}}{\lambda^2}|\sqrt{\sigma}\partial_t W|_{L^2}\lesssim\frac{b}{\lambda^3}b^{k+2}\sqrt{\mathcal E_{\sigma}} .$$ 
Finally, using \eqref{estalambdafdeux}:
\bee
 \left| \int  A_{\lambda}F_2A_{\lambda}\left(\frac{(\Psi_{B_1})_{\lambda}}{\lambda^2}\right)\right|& \lesssim & \frac{|b_s|}{\lambda^3}\int \frac{1}{y(1+y)}\left|C(M)\frac{b^{k+4}y^k}{1+y^{k+1}}{\bf 1}_{y\leq 2B_1}+b^{k+2}{\bf 1}_{B_1\leq y\leq 2B_1}\right|\\
& \lesssim & \frac{|b_s|b^{k+2}}{\lambda^3}\lesssim \frac{b}{\lambda^3}[b^{2k+2}+|b_s|^2],
\eee
\bee
 \left| \sigma \int  A_{\lambda}F_2A_{\lambda}\left(\frac{(\Psi_{B_1})_{\lambda}}{\lambda^2}\right)\right| & \lesssim & \frac{|b_s|}{\lambda^3}\int \frac{1}{y(1+y)}\left|C(M)\frac{b^{k+4}y^k}{1+y^{k+1}}\right|\\
& \lesssim & \frac{|b_s|b^{k+3}}{\lambda^3}\lesssim \frac{b}{\lambda^3}[b^{2k+4}+|b_s|^2].
\eee
{\it case $k=1$}: We estimate from \fref{estpsiboddtildekone}:
\bee
& &\hskip -5pc \left|\int\frac{\sigma W}{r}\partial_t\vul\frac{(\Psi_{B_1})_{\lambda}}{\lambda^2}\right|  
\lesssim  \frac {b^{\frac 12}} {\lambda^{\frac 32}}\left(\int\sigma \frac{W^2}{r^2} |\pa_t V^{(1)}_\lambda| \right)^{\frac 12}\left (\int |\Psi_{B_1}|^2\frac{y^{2}}{1+y^{4}}\right)^{\frac 12}\\
\nonumber & \lesssim &  c\int\sigma \frac{W^2}{r^2}| \pa_t V^{(1)}_\lambda|+ 
\frac{b}{c \lambda^3} \int \frac{y^{2}}{1+y^{4}}
\left[
\frac{b^2}{y}{\bf 1}_{{B_1}\leq y\leq 2B_1}+\frac{b^2}{|\log b|}\frac{y}{1+y^2}\right]^2\\
& \lesssim &   c \int\sigma \frac{W^2}{r^2} |\pa_t V^{(1)}_\lambda|
+  \frac{b}{\lambda^3}\frac{b^{4} }{c|\log b|^2}.
\eee 
for some small {\it universal} constant $c>0$. By the Remark \ref{rem:V} the first term on the RHS above can be absorbed in the
energy identity \eqref{computationeergy}.\\
Next, we use again the fundamental cancellation:
$$|A(\chi_{\frac{B_0}{4}}\Lambda Q)|\lesssim \frac{1}{y^2}{\bf 1}_{\frac{B_0}{8}\leq y\leq \frac{B_0}{2}}$$
which implies from \fref{estpsiboddtildekone} and $c_b\sim \frac 1{|\log b|}$:
\bee
& & \int \left[A_{\lambda}\left(\frac{(\Psi_{B_1})_{\lambda}}{\lambda^2}\right)\right]^2 \\
& \lesssim & \frac{1}{\lambda^4}\int_{y\leq 2B_1} \frac{1}{y^2}\left[\frac{b^2}{y}{\bf 1}_{B_1\leq y\leq 2B_1}
+ C(M) b^4 \frac y{1+y^4} +b^{4}\frac{(1+|\log (by)|)}{|\log b|}y{\bf 1}_{1\leq y\leq \frac{B_0}{2}}\right .\\
& + & \left . \frac{b^2}{|\log b|y }{\bf 1}_{ \frac{B_0}{2}\leq y\leq 2B_1}\right]^2 \lesssim  \frac{b^{6}}{\lambda^4},
\eee
\bee
& & \int \sigma_{B_c} \left[A_{\lambda}\left(\frac{(\Psi_{B_1})_{\lambda}}{\lambda^2}\right)\right]^2 \\
& \lesssim & \frac{1}{\lambda^4}\int_{y\leq 2B} \frac{1}{y^2}\left[ C(M) b^4 \frac y{1+y^4} +b^{4}\frac{(1+|\log (by)|)}{|\log b|}y{\bf 1}_{1\leq y\leq \frac{B_0}{2}}+ \frac{b^2}{|\log b|y }{\bf 1}_{ \frac{B_0}{2}\leq y\leq 3B_c}\right]^2\\
& \lesssim & \frac{b^{6}}{|\log b|^2\lambda^4},
\eee
from which:
$$
\left|\int \partial_tWA_{\lambda}\left(\frac{(\Psi_{B_1})_{\lambda}}{\lambda^2}\right)\right|\lesssim \frac{b^{3}}{\lambda^2}|\partial_t W|_{L^2}\lesssim \frac{b}{\lambda^3}b^2\sqrt{\mathcal E} ,$$
$$
\left|\int \sigma \partial_tWA_{\lambda}\left(\frac{(\Psi_B)_{\lambda}}{\lambda^2}\right)\right|\lesssim \frac{b^{3}}{|\log b|\lambda^2}|\sqrt{\sigma}\partial_t W|_{L^2}\lesssim  \frac{b}{\lambda^3}\frac{b^2\sqrt{\mathcal E_{\sigma}}}{|\log b|}.$$ 
Finally:
\bee
& & \left|\int A_{\lambda}F_2A_{\lambda}\left(\frac{(\Psi_{B_1})_{\lambda}}{\lambda^2}\right)\right|\\
& \lesssim & \frac{|b_s|}{\lambda^3}\int \frac{1}{y(1+y)}\left[\frac{b^2}{y}{\bf 1}_{B_1\leq y\leq 2B_1}
+ C(M) b^4 \frac y{1+y^4} +b^{4}\frac{(1+|\log (by)|)}{|\log b|}y{\bf 1}_{1\leq y\leq \frac{B_0}{2}}\right .\\
& + & \left . \frac{b^2}{|\log b|y }{\bf 1}_{ \frac{B_0}{2}\leq y\leq 2B_1}\right] \lesssim  \frac{|b_s|b^3}{\lambda^3}\lesssim \frac{b}{\lambda^3}[b^4+|b_s|^2],
\eee
\bee
& & \left|\int \sigma_{B_c}A_{\lambda}F_2A_{\lambda}\left(\frac{(\Psi_{B_1})_{\lambda}}{\lambda^2}\right)\right|\\
& \lesssim & \frac{|b_s|}{\lambda^3}\int_{y\leq 2B_c}\frac{1}{y(1+y)}\left[C(M) b^4 \frac y{1+y^4} +b^{4}\frac{(1+|\log (by)|)}{|\log b|}y{\bf 1}_{1\leq y\leq \frac{B_0}{2}}+ \frac{b^2}{|\log b|y }{\bf 1}_{ \frac{B_0}{2}\leq y\leq 2B_1}\right]\\
& \lesssim & \frac{|b_s|b^3}{|\log b|\lambda^3}\lesssim \frac{b}{\lambda^3}[\frac{b^4}{|\log b|^2}+|b_s|^2]
\eee
Note the sharpness of the above estimate. Its most significant contribution is generated by the second term
in the square brackets above.
\vskip .5pc

{\bf step 9} Conclusion.\\

The collection of all previous estimates now yields the claimed bounds \fref{vhoheor}, \fref{vbboebvnkonovrpvrvo} and concludes the proof of Lemma \ref{propinside}.


\subsection{Proof of Proposition \ref{bootstrap}}


We are now in position to complete the proof of Proposition \ref{bootstrap}. The key will be to combine the a priori bound on the blow up acceleration given by Lemma \ref{roughboundpointw} with the information provided in \fref{vhoheor}, \fref{vbboebvnkonovrpvrvo}.
The smallness of the coupling constant $(\log M)^{-1}$ in Lemma \ref{roughboundpointw}, linking the behavior
of the blow acceleration $b_s$ with the pointwise behavior of the local energy ${\mathcal E}_\sigma$, provides 
the mechanism allowing us to combine the two estimates and obtain the desired bounds. Equally crucial to this strategy 
is independence of the constants in \fref{vhoheor}, \fref{vbboebvnkonovrpvrvo} on $M$ noted in the Remark \ref{remarkkey}.\\

{\bf step 1} Control of the scaling parameter.\\

We begin with the proof of \fref{controllambdaboot}. First observe from \fref{casekoddbigone} 
and the bootstrap estimate \fref{poitwisebs} that 
\be
\label{rgouboundbs}
|b_s|\leq K\frac{b^2}{|\log b|}\leq \frac{b^2}{100k}
\ee
This implies::
$$\frac{d}{ds}\left(\frac{b^{k+2}}{\lambda}\right)=\frac{b^{k+1}}{\lambda}\left[b^2+(k+2)b_s\right]\geq 0$$ and hence from \fref{smalllmba}:
$$\frac{b^{k+2}(t)}{\lambda(t)}\geq \frac{b^{k+2}(0)}{\lambda(0)}\geq 1$$ and \fref{controllambdaboot} follows.
We derive similarily
\be
\label{booregime}
\frac{b^{k+1}(0)}{\lambda(0)}\le \frac{b^{k+1}(t)}{\lambda(t)},
\ee
\begin{equation}\label{eq:idon}
 \frac{b^{k+1}(0)}{|\log b(0)|\lambda(0)}\le \frac{b^{k+1}(t)}{|\log b(t)|\lambda(t)},\ \ 
b^{2k+2}(0)\frac{\lambda^2(t)}{\lambda^2(0)}\le {b^{2k+2}(t)}
\end{equation}

{\bf step 2} Bound on the global energy.\\

We now turn to  the proof of  \fref{poitnwiseboundWboot}.\\
In this case we use the bootstrap assumptions \eqref{poitwisebs},  \fref{poitnwiseboundW} to obtain from 
 \fref{vhoheor}
\be
\label{chooehewcveieho}
\frac{\mathcal E(t)}{\lambda^2(t)}\lesssim \frac{\mathcal E(0)}{\lambda^2(0)}+\int_0^t\sqrt{K}\frac{b^{2k+3}}{\lambda^3}+
\frac {b^{2k+2}(t)}{\lambda^2(t)} + \frac {b^{2k+2}(0)}{\lambda^2(0)}.
\ee
Note that we used the inequalities $\eta^{\frac 14 } K\le 1$ and $|\log b|^{-1} K\le 1$. We then derive from \fref{rgouboundbs}:
\bee
\int_0^t \frac{b^{2k+3}}{\lambda^3}& = & -\int_0^t\frac{\lambda_tb^{2k+2}}{\lambda^3}=\frac{b^{2k+2}(t)}{2\lambda^2(t)}-\frac{b^{2k+2}(0)}{2\lambda^2(0)}-(k+1)\int_0^t\frac{b_tb^{2k+1}}{\lambda^2}\\
& \leq & \frac{b^{2k+2}(t)}{\lambda^2(t)}+(k+1)\int_0^t\frac{|b_s|b^{2k+1}}{\lambda^2}\\
& \leq & \frac{b^{2k+2}(t)}{\lambda^2(t)}+\frac{1}{2}\int_0^t \frac{b^{2k+3}}{\lambda^3}
\eee
and hence the bound:
\be
\label{keyboundt}
\int_0^t \frac{b^{2k+3}}{\lambda^3}\leq 2\frac{b^{2k+2}(t)}{\lambda^2(t)}.\ee 
Note that the above inequality holds is derived under the assumptions of the regime under consideration. We now insert 
\fref{keyboundt} into \fref{chooehewcveieho} and use \fref{booregime} to conclude: $\forall t\in [0,T_1)$,
\bea
\label{borneimtermemergy}
\nonumber \mathcal{E}(t) & \lesssim & \frac{\lambda^2(t)}{\lambda^2(0)}\mathcal E(0)+\sqrt{K}b^{2k+2}(t)+b^{2k+2}(0)\frac{\lambda^2(t)}{\lambda^2(0)}\\
& \lesssim &  \frac{\lambda^2(t)}{\lambda^2(0)}\mathcal E(0)+\sqrt{K}b^{2k+2}(t).
\eea
Observe now from the initial bound \eqref{inoshohgoer} and \fref{eq:idon}:
\bee
\frac{\lambda^2(t)}{\lambda^2(0)}\mathcal E(0)\lesssim \frac{\lambda^2(t)}{\lambda^2(0)}\frac{b_0^{2k+2}}{|\log b_0|^2}\lesssim\frac{b^{2k+2}(t)}{|\log b(t)|^2}
\eee
and thus \fref{borneimtermemergy} implies:
 \be
 \label{betterglobalbound}
 \mathcal{E}(t)  \lesssim \sqrt{K} b^{2k+2}(t).
  \ee
 This yields \fref{poitnwiseboundWboot} for $K$ large enough.\\
  
{\bf step 3} Bound on the local energy and $b_s$.\\

First observe from the $b_s$ bound \fref{estbds} and the bootstrap bound \fref{localizegviojdo} that $$|b_s|^2\lesssim \frac{b^{2k+2}}{|\log  b|^2}\left(1+\frac{K}{\log M}\right)$$ which implies \fref{poitwisebsboot}. We now substitute \fref{poitwisebsboot}, \fref{poitnwiseboundWboot} and the improved bound \fref{betterglobalbound} into \fref{vbboebvnkonovrpvrvo} and integrate in time to get:
\bea
\label{chooehewcveiehobis}
\nonumber \frac{\mathcal E_{\sigma}(t)}{\lambda^2(t)} & \lesssim & \frac{\mathcal E(0)}{\lambda^2(0)}+\int_0^t \frac{b^{2k+3}}{|\log b|^2\lambda^3}\left(1+\sqrt K+\frac{K}{\sqrt{\log M}}\right)\\
& + & \frac {b^{2k+2}(t)}{\log^2b(0) \lambda^2(t)}+\frac {b^{2k+2}(0)}{\log^2b(0) \lambda^2(0)}.
\eea
We now estimate from \fref{rgouboundbs}:
\bee
\int_0^t \frac{b^{2k+3}}{|\log b|^2\lambda^3}& = & -\int_0^t\frac{\lambda_tb^{2k+2}}{|\log b|^2\lambda^3}\lesssim \frac{b^{2k+2}(t)}{2|\log b(t)|^2\lambda^2(t)}+(k+1)\int_0^t\frac{|b_s|b^{2k+1}}{|\log b|^2\lambda^2}\\
& \leq & \frac{b^{2k+2}(t)}{|\log b(t)|^2\lambda^2(t)}+\frac{1}{2}\int_0^t \frac{b^{2k+3}}{|\log b|^2\lambda^3}
\eee
and substitute this into \fref{chooehewcveiehobis} together with \fref{ivnovnhdlbvjepjvepvo}, \fref{controllambda} to get:
 \bee
 \mathcal{E}_{\sigma}(t) & \lesssim & \left(1+\sqrt K+\frac{K}{\sqrt{\log M}}\right)\frac{b^{2k+2}(t)}{|\log b(t)|^2}\leq \frac{K}{2}\frac{b^{2k+2}(t)}{|\log b(t)|^2}
 \eee
for $K=K(M)$ large enough, and \fref{localizegviojdoboot} follows.\\

{\bf step 4} Finite time blow up.\\

We now have proved that $T_1=T$. It remains to prove that $T<+\infty$. From \fref{controllambda}, the scaling parameter satisfies the pointwise differential inequality
\be
\label{cvnbohveo}-\lambda_t=b\geq \lambda^{\frac{1}{k+1}}\geq \sqrt{\lambda}
\ee 
from which: $$\forall t\in [0,T), \ \ -2\sqrt{\lambda(t)}+2\sqrt{\lambda}(0)\geq t.$$
Positivity of $\lambda$ implies $T<+\infty$.\\

This concludes the proof of Proposition \ref{bootstrap}.


\section{Sharp description of the singularity formation}
\label{sectionfour}


This section is devoted to the proof of Theorem \ref{mainthm}. We will 
provide a precise description of the dynamics of the parameter $b$ and the scaling parameter $\lambda$, as
required in \eqref{universallawkgeq}-\fref{universallawkgeq2}.
In particular, we will prove that $b\to 0$ as $t\to T$, which  together with \fref{poitnwiseboundWboot}, \fref{localizegviojdoboot} implies dispersion of the excess of energy at the blow up time. 
These estimates are crucial for the proof of the quantization of the blow up energy as stated in  
\fref{convustarb}. The first step of the proof relies on a flux computation 
leading to a sharp differential inequality for the parameter $b$. The leading contribution to the flux is provided
by an explicit 
behavior of the radiative part of the $Q_b$ profile. To identify it as a leading contribution we exploit the logarithmic 
gain in the local energy bound \fref{localizegviojdoboot}. This analysis can be thought of as related  to the $L^2$ flux calculation 
in \cite{MR4} leading to the $\log\hskip -.2pc-\hskip -.2pc\log$ blow up law for the  $L^2$ critical (NLS).


\subsection{The flux computation and the derivation of the $b_s$ law}
\label{sectionbsharp}

In this section we derive the precise behavior of the parameter $b(t)$ modulo  negligible time oscillations. 
This is achieved by refining the analysis of Lemma \ref{roughboundpointw} and projecting the $\e$ equation \fref{eqeqb} onto the instability direction of the linearized operator $H_{B_0}$ associated to $P_{B_0}$. 

Define 
\be
\label{deffb}
G(b)=b|\Lambda P_{B_0}|_{L^2}^2+\int_0^b \tilde b(\frac{\partial P_{B_0}}{\partial b},\Lambda P_{B_0})d\tilde b.
\ee
and
\be
\label{defiun}{\mathcal I}(s)=(\partial_s\eb,\Lambda P_{B_0})+b(\e+2\Lambda\eb,\Lambda P_{B_0})+b_s(\frac{\partial P_{B_0}}{\partial b},\Lambda P_{B_0})-b_s \left(\frac{\partial}{\partial b}(P_{B_1}-P_{B_0}), \Lambda P_{B_0})\right ) .
\ee
We claim:
\begin{proposition}[Sharp derivation of the $b$ law]
\label{lemmaalgebra}
For $b\leq b^*_0$ small enough, there holds:
\be
\label{estkeyfb}
G(b)=\left\{\begin{array}{ll} b|\Lambda Q|_{L^2}^2(1+o(1)) \ \ \mbox{for} \ \ k\geq2,\\
					4b |\log b| +O(b)\ \ \mbox{for} \ \ k=1.
					\end{array} \right.
\ee
and
\be
\label{estimateione}
|{\mathcal{I}}|\lesssim \left \{\begin{array}{ll} b^2 |log b|\ \ \mbox{for} \ \ k\geq 2,\\
								   b \ \ \mbox{for} \ \ k=1.
								   \end{array} \right .
\ee
The functions $G, {\mathcal I}$ satisfy the following differential inequalities:
\be
\label{firstcontrol}
\left|\frac{d}{ds}\{G(b)+{\mathcal{I}}(s)\}+\tilde{c}_k b^{2k}\right|\leq \frac{b^{2k}}{|\log b|}
\ee
with 
\be
\label{valuecp}
\tilde{c}_k=\left\{\begin{array}{lll} \frac{c_p^2}{2} \ \ \mbox{for k odd}, \ \ k\geq 3,\\
\frac{k^2c_p^2}{2} \ \ \mbox{for k even}\\
2 \ \ \mbox{for} \ \ k=1.
\end{array} \right .
\ee
\end{proposition}

\begin{remark} Observe that \fref{estkeyfb}, \fref{estimateione}, \fref{firstcontrol} essentially yield a pointwise differential equation $$b_s\sim -\left\{\begin{array}{ll} b^{2k} \ \ \mbox{for} \ \ k\geq 2\\
												\frac{b^2}{2|\log b|} \ \ \mbox{for} \ \ k=1.
												\end{array} \right .
												$$
								 which will allow us to derive the sharp scaling law via the 
								 relationship $-\lsl=b$. Note also that for $k\ge 2$, with a little bit more work, the logarithmic gain in the RHS of \fref{firstcontrol} may be turned into a polynomial gain in $b$. 
\end{remark}

{\bf Proof of Proposition \ref{lemmaalgebra}}\\

We multiply (\ref{eqeqb}) with $\Lambda P_{B_0}$ -- the instability direction of $H_{B_0}$ -- and compute:
\bea
\nonumber & & (b_s\Lambda \qbtb+b(\partial_s\qbtb+2\Lambda\partial_s\qbtb)+\partial_s^2P_{B_1},\Lambda P_{B_0})=-(\Psi_{B_1},\Lambda P_{B_0})-(H_{B_1}\e,\Lambda P_{B_0})\\
& - &  \left( \partial_s^2\eb+b(\partial_s\eb+2\Lambda\partial_s\eb)+b_s \Lambda \eb,\Lambda P_{B_0}\right)-k^2(\frac{N(\e)}{y^2},\Lambda P_{B_0})\notag
\eea
We further rewrite this as follows:
\bea\label{vhoheoheo}
\nonumber & & (b_s\Lambda P_{B_0}+b(\partial_s P_{B_0}+2\Lambda\partial_s P_{B_0})+\partial_s^2P_{B_0},\Lambda P_{B_0})=-(\Psi_{B_1},\Lambda P_{B_0})-(H_{B_1}\e,\Lambda P_{B_0})\\&-&
(b_s\Lambda (P_{B_1}-P_{B_0})+b(\partial_s (P_{B_1}-P_{B_0})+2\Lambda\partial_s (P_{B_1}-P_{B_0}))+\partial_s^2(P_{B_1}-P_{B_0}),\Lambda P_{B_0})\notag
\\
& - &  \left( \partial_s^2\eb+b(\partial_s\eb+2\Lambda\partial_s\eb)+b_s \Lambda \eb,\Lambda P_{B_0}\right)-k^2(\frac{N(\e)}{y^2},\Lambda P_{B_0})
\eea

We now estimate all terms in the above identity.\\

{\bf step 1} Transformation of the LHS of \fref{vhoheoheo}.\\

We claim that the LHS of \fref{vhoheoheo} may be rewritten as follows:
\bea
\label{rewritnghgllhs}
  \nonumber & & (b_s\Lambda P_{B_0}+b(\partial_s P_{B_0}+2\Lambda\partial_s P_{B_0})+\partial_s^2P_{B_0},\Lambda P_{B_0})\\ 
  & = & \frac{d}{ds}\left[G(b)+b_s(\frac{\partial P_{B_0}}{\partial b},\Lambda P_{B_0})\right]+|b_s|^2|\frac{\partial P_{B_0}}{\partial b}|_{L^2}^2
  \eea
  with $G$ given by \fref{deffb} and the bound:
\be
\label{cboboboboe}
|b_s|^2|\frac{\partial P_{B_0}}{\partial b}|_{L^2}^2\lesssim \frac{b^{2k}}{|\log b|^2}.
\ee

{\it Proof of \fref{rewritnghgllhs}} Let $$\phi(t,y)=(P_{B_0})_{\lambda},$$ then: $$\partial_{tt} \phi=\frac{1}{\lambda^2}\left[\partial_s^2 P_{B_0}+b(\partial_sP_{B_0}+2\Lambda\partial_sP_{B_0})+b^2D\Lambda P_{B_0}+b_s\Lambda P_{B_0}\right]_{\lambda}.$$ Using the cancellation $$(D\Lambda P_{B_0},\Lambda P_{B_0})=0,$$  
this yields:
\bea
\label{cbbbdo}
\nonumber & & (b_s\Lambda \qbtb+b(\partial_s\qbtb+2\Lambda\partial_s\qbtb)+\partial_s^2P_{B_1},\Lambda P_{B_0})=\lambda^2(\partial_{tt}\phi(\lambda y),\Lambda \phi(\lambda y))=(\partial_{tt}\phi,\Lambda \phi)\\
& = & \frac{d}{dt}\left[(\partial_t\phi,\Lambda \phi)\right]-(\partial_t\phi,\Lambda\partial_t\phi)= \frac{d}{dt}\left[(\partial_t\phi,\Lambda \phi)\right]+\int (\partial_t\phi)^2.
\eea
We now compute each term separately:
\bee
 \frac{d}{dt}\left[(\partial_t\phi,\Lambda \phi)\right] & = &  \frac{1}{\lambda}\frac{d}{ds}\left[\lambda(\partial_sP_{B_0}+b\Lambda P_{B_0},\Lambda P_{B_0})\right]\\
 & = & \frac{d}{ds}\left[b|\Lambda P_{B_0}|_{L^2}^2+b_s(\frac{\partial P_{B_0}}{\partial b},\Lambda P_{B_0})\right]-b\left[b|\Lambda P_{B_0}|_{L^2}^2+b_s(\frac{\partial P_{B_0}}{\partial b},\Lambda P_{B_0})\right].
 \eee
On the other hand, $$\int (\partial_t\phi)^2=\int(\partial_sP_{B_0}+b\Lambda P_{B_0})^2=|b_s|^2|\frac{\partial P_{B_0}}{\partial b}|_{L^2}^2+2b_sb(\frac{\partial P_{B_0}}{\partial b},\Lambda P_{B_0})+b^2|\Lambda P_{B_0}|_{L^2}^2.$$
Substituting these two computations into (\ref{cbbbdo}) yields:
\bee
 & & (b_s\Lambda \qbtb+b(\partial_s\qbtb+2\Lambda\partial_s\qbtb)+\partial_s^2P_{B_1},\Lambda P_{B_0})\\
 & = &  \frac{d}{ds}\left[b|\Lambda P_{B_0}|_{L^2}^2+b_s(\frac{\partial P_{B_0}}{\partial b},\Lambda P_{B_0})\right]+b_sb(\frac{\partial P_{B_0}}{\partial b},\Lambda P_{B_0})+|b_s|^2|\frac{\partial P_{B_0}}{\partial b}|_{L^2}^2\\
 & = & \frac{d}{ds}\left[G(b) +b_s(\frac{\partial P_{B_0}}{\partial b},\Lambda P_{B_0})\right]
+|b_s|^2|\frac{\partial P_{B_0}}{\partial b}|_{L^2}^2,
\eee
which gives \fref{rewritnghgllhs}. 
To prove \fref{cboboboboe}, we first estimate from \fref{crudeboundb}:
$$|\partial_bP_{B_0}|_{L^2}^2\lesssim  \int_{y\leq 2B_0}\left(\frac{y^2}{(1+y^2)|\log b|^2}+ \frac 1{b^2y^2}{\bf 1}_{\frac{ B_0}{2}\le y\leq 2B_0}\right)\lesssim \frac{1}{b^2},$$ and hence \fref{cboboboboe} follows from \fref{poitwisebsboot}.\\

{\bf step 2} The flux computation.\\

We now turn to the first key step in the derivation of the sharp $b$ law. It is the following outgoing flux computation:
\be
\label{ourgoingflux}
(\Psi_{B_1},\Lambda P_{B_0})=d_pb^{2k}\left(1+O\left(\frac{1}{{|\log b|}}\right)\right) \ \ \mbox{as} \ \ b\to 0.
\ee
The error in this identity is determined by the (non-sharp) choice  of $B_1$ in \fref{defbnot}. The universal constant $$d_p=\left\{\begin{array}{lll} \frac{k^2c_p^2}{2} \ \ \mbox{for k even}\\
							\frac{c_p^2}{2}\ \ \mbox{for k odd,} \ \ k\geq 3\\
							2 \ \ \mbox{for} \ \ k=1\end{array}
							\right .
	$$

{\it Proof of \fref{ourgoingflux}}: Let us define the expression, which in what follows we will refer to as the 
radiation term, 
\be
\label{decompebebo}
\zeta_b=P_{B_1}-P_{B_0}=(\chi_{B_1}-\chi_{B_0})(Q_b-a)
\ee 
with $a=\pi$ for the (WM) problem and $a=-1$ for the (YM).
It satisfies:
\be
\label{suppzetsb}
Supp(\zeta_b)\subset\{B_0\leq y\leq 2B_1\},
\ee 
and the equation:
$$
-\Delta\zeta_b+b^2D\Lambda \zeta_b+k^2\frac{f(P_{B_0}+\zeta_b)-f(P_{B_0})}{y^2}=\Psi_{B_1}-\Psi_{B_0}
$$
which we rewrite:
\be
\label{eqzetab}
-\Delta\zeta_b+b^2D\Lambda \zeta_b+k^2\frac{\zeta_b}{y^2}=\Psi_{B_1}-\Psi_{B_0}-M(\zeta_b)
\ee
with 
\be
\label{erreurline}
M(\zeta_b)=k^2\frac{f(P_{B_0}+\zeta_b)-f(P_{B_0})-f'(P_{B_0})\zeta_b+(f'(P_{B_0})-1)\zeta_b}{y^2}.
\ee
We now manipulate the identity
$$(\Psi_{B_1},\Lambda P_{B_0})=(\Psi_{B_1},\Lambda P_{B_1})-(\Psi_{B_1},\Lambda \zeta_b)=-(\Psi_{B_1},\Lambda \zeta_b).$$ In the last step we used the Pohozaev identity (\ref{pohozaev}): 
$$(\Psi_{B_1},\Lambda P_{B_1})=(-\Delta P_{B_1}+b^2D\Lambda P_{B_1}+
k^2\frac{f(P_{B_1})}{y^2},\Lambda P_{B_1})=0,
$$
which holds for $\Lambda P_{B_1}$ of compact support and $g(P_{B_1}(y))$ with the boundary value $\lim_{y\to +\infty}g(P_{B_1}(y))=0$.
We now integrate by parts, use the formula (\ref{formulapsitilde}) and the localization property (\ref{suppzetsb}) to conclude:
\bea
 \label{poeutpouet}
 \nonumber & &  -(\Lambda \zeta_b,\Psi_{B_1}) =  -\int_{B_1}^{2B_1}\Lambda \zeta_b\Psi_{B_1}ydy-\int_{B_0}^{B_1}\Lambda \zeta_b\chi_{B_1}\Psi_bydy\\
\nonumber & = & -\int_{B_1}^{2B_1}\Lambda \zeta_b(\Psi_{B_1}-\Psi_{B_0})ydy-\int_{B_0}^{B_1}\Lambda \zeta_b\chi_{B_1}\Psi_bydy\\
\nonumber & = & \int_{B_1}^{2B_1}\Lambda \zeta_b[\Delta\zeta_b-b^2D\Lambda \zeta_b-k^2\frac{\zeta_b}{y^2}]ydy\\
& - &   \int_{B_1}^{2B_1}\Lambda \zeta_b M(\zeta_b)ydy- \int_{B_0}^{B_1}\Lambda \zeta_b\chi_{B_1}\Psi_b\, ydy.
\eea
In the last step we also used (\ref{eqzetab}). The first term on the RHS above produces the leading order flux term from the Pohozaev integration (\ref{pohozaev}) and the boundary conditions $\zeta_b(2B_1)=\zeta_b'(2B_1)=0$:
\bee
 \int_{B_1}^{2B_1}\Lambda \zeta_b[\Delta\zeta_b-b^2D\Lambda \zeta_b-k^2\frac{\zeta_b}{y^2}]ydy=\left[\frac{1}{2}(b^2y^2-1)|\Lambda \zeta_b|^2+\frac{k^2}{2}\zeta_b^2\right](B_1).
 \eee
Now from (\ref{decompebebo}) and the estimates on $Q_b$ from Proposition \ref{propqb} with the choice $B_1=\frac{|\log b|}{b}>>\frac{1}{b}$, there holds: $\forall y\in [\frac{B_1}{2}, B_1],$ 
\be
\label{fhrohor} \ \ \zeta_b(y)=(Q_b-a)(y)=\left\{\begin{array}{lll} \frac{c_p}{y}b^{k-1}(1+O(\frac{1}{|\log b|})) \ \ \mbox{for k odd,} \ \ k\geq 3,\\
								   c_pb^k(1+O(\frac{1}{|\log b|})) \ \ \mbox{for k even},\\
								   \frac{2}{y}(1+O(\frac{1}{|\log b|^2})) \ \ \mbox{for} \ \ k=1 
								   \end{array} \right .
\ee
 from which 
\be
\label{computflux}
 \int_{B_1}^{2B_1}\Lambda \zeta_b[\Delta\zeta_b-b^2D\Lambda \zeta_b-k^2\frac{\zeta_b}{y^2}]ydy
=\left \{ \begin{array}{lll} \frac{c_p^2b^{2k}}{2}(1+O(\frac{1}{|\log b|^2})) \ \ \mbox{for k odd},\\
			           \frac{k^2c^2_pb^{2k}}{2}(1+O(\frac{1}{|\log b|^2}))\ \ \mbox{for k even},\\
			           2b^2(1+O(\frac{1}{|\log b|^2})) \ \ \mbox{for} \ \ k=1.
			           \end{array} \right .
\ee		
It remains to estimate the error terms in (\ref{poeutpouet}). For this, first observe the crude bound:
\be
\label{fhrohorbis} 
\forall y\in[B_0,2B_1], \ \ |\zeta_b(y)|+ |\Lambda \zeta_b(y)|\lesssim \left\{\begin{array}{ll} \frac{b^{k-1}}{y} \ \ \mbox{for k odd}\\
b^k \ \ \mbox{for k even}
\end{array} \right .
\ee
and from \fref{erreurline}:
\be
\label{estmzetab}
\forall y\in[B_0,2B_1], \ \ |M(\zeta_b)|\lesssim \frac{1}{y^2}\left[|\zeta_b|^2+\frac{|\zeta_b|}{y^{k}}\right].
\ee
{\it case $k\geq 3$ odd}: From (\ref{estpsibodd}):
$$\int_{B_0}^{B_1}|\Lambda \zeta_b\chi_{B_1}\Psi_b|ydy\lesssim \int_{B_0}^{B_1}\frac{b^{k-1}}{y}\frac{b^{k+3}}{y^2}ydy\lesssim b^{2k+3}.
$$
Next, (\ref{fhrohorbis}) and (\ref{estmzetab}) imply
$$
 \int_{B_1}^{2B_1}|\Lambda \zeta_b M(\zeta_b)|ydy \lesssim  \int_{B_1}^{2B_1}\frac{b^{k-1}}{y}\frac{1}{y^2}(\frac{b^{2k-2}}{y^{2}}+\frac{b^{k-1}}{y^{k+1}})ydy\lesssim b^{3k}.
 $$
{\it case $k\geq 4$ even}:  From (\ref{estpsibeven}), \fref{fhrohorbis} :
$$\int_{B_0}^{B_1}|\Lambda \zeta_b\chi_{B_1}\Psi_b|ydy\lesssim \int_{B_0}^{B_1}b^{k}\frac{b^{k+4}}{y}ydy\lesssim b^{2k+3}.
$$
From (\ref{fhrohorbis}) and (\ref{estmzetab}):
$$
 \int_{B_1}^{2B_1}|\Lambda \zeta_b M(\zeta_b)|ydy \lesssim   \int_{B_1}^{2B_1}\frac{b^{k}}{y^2}(b^{2k}+\frac{b^{k}}{y^{k}})ydy\lesssim b^{3k}$$
{\it case $k=2$}: 
From (\ref{vheohoevbobvob}), \fref{fhrohorbis} :
$$\int_{B_0}^{B_1}|\Lambda \zeta_b\chi_{B_1}\Psi_b|ydy\lesssim \int_{B_0}^{B_1}b^{k}\left[C(M)\frac{b^{k+4}}{y^2}+\frac{b^4}{y^2}\right]ydy\lesssim b^{2k+1}.
$$
From (\ref{fhrohorbis}) and (\ref{estmzetab}):
$$
 \int_{B_1}^{2B_1}|\Lambda \zeta_b M(\zeta_b)|ydy \lesssim   \int_{B_1}^{2B_1}\frac{b^{k}}{y^2}(b^{2k}+\frac{b^{k}}{y^{k}})ydy\lesssim b^{3k}.$$
{\it case $k=1$}: We recall that according to \eqref{eq:imp-Psi}, $|\Psi_b|\lesssim \frac {b^4}{1+y}$ for $y\ge B_0$. Therefore,
$$\int_{B_0}^{B_1}|\Lambda \zeta_b\chi_{B_1}\Psi_b|ydy\lesssim \int_{B_0}^{B_1}\frac{1}{y}\frac{b^{4}}{y}ydy\leq b^4|\log b|\leq b^3.
$$
Next, (\ref{fhrohorbis}) and (\ref{estmzetab}) imply
$$
 \int_{B_1}^{2B_1}|\Lambda \zeta_b M(\zeta_b)|ydy \lesssim  \int_{B_1}^{2B_1}\frac{1}{y}\frac{1}{y^2}(\frac{1}{y^{2}}+\frac{1}{y^{2}})ydy\lesssim b^{3}.
 $$
This concludes the proof of \fref{ourgoingflux}.\\

{\bf step 3} Second line of \eqref{vhoheoheo}.\\

We first observe  
\begin{align*}
b_s\Lambda (\qbtb-P_{B_0})&+b(\partial_s(\qbtb-P_{B_0})+2\Lambda\partial_s(\qbtb-P_{B_0}))+\partial_s^2
(P_{B_1}-P_{B_0})\\ &=
b_s\Lambda \zeta_b + b(\partial_s \zeta_b + 2\Lambda \partial_s \zeta_b) +\partial_s^2 \zeta_b
\end{align*}
We further rewrite
\begin{align*}
&(b_s\Lambda \zeta_b + b(\partial_s \zeta_b + 2\Lambda \partial_s \zeta_b) +\partial_s^2 \zeta_b,\Lambda P_{B_0})=
\frac d{ds} (\partial_s \zeta_b, \Lambda P_{B_0}) + (b_s\Lambda \zeta_b + b(\partial_s \zeta_b + 2\Lambda \partial_s \zeta_b), \Lambda P_{B_0})\\&\,\,\qquad-(\partial_s \zeta_b,\Lambda \pa_s P_{B_0})= \frac d{ds} \left [b_s (\partial_b \zeta_b, \Lambda P_{B_0})\right ] + b_s(\Lambda \zeta_b + b (\partial_b \zeta_b + 2\Lambda \partial_b \zeta_b), \Lambda P_{B_0})-b^2_s (\partial_b \zeta_b,\Lambda \pa_b P_{B_0})
\end{align*}
We use crude bounds similar to \eqref{fhrohorbis}, $\forall y\in[B_0,2B_0]$, 
\begin{align*}
&|\zeta_b(y)|+ |\Lambda \zeta_b(y)|+b |\pa_b \zeta_b(y)|+b |\Lambda\pa_b \zeta_b(y)|\lesssim 
b^k,\\
&|\Lambda P_{B_0}|+ |\Lambda \pa_b P_{B_0}|\lesssim b^k
\end{align*}
As a consequence,
$$
|b_s| |(\partial_b \zeta_b, \Lambda P_{B_0})|\lesssim \frac {b^{k+1}}{|\log b|} b^{2k-3}\le \frac{b^{3k-2}}{|\log b|}
$$
and
\begin{align}
&|b_s|\,|(\Lambda \zeta_b + b (\partial_b \zeta_b + 2\Lambda \partial_b \zeta_b), \Lambda P_{B_0})|\lesssim
 \frac {b^{k+1}}{|\log b|} b^{2k-2}\le  \frac {b^{3k-1}}{|\log b|},\label{eq:zeta1}\\
&b^2_s |(\partial_b \zeta_b,\Lambda \pa_b P_{B_0})|\lesssim  \frac {b^{2k+2}}{|\log b|^2} b^{2k-4}\le  \frac {b^{4k-2}}{|\log b|^2}\label{eq:zeta2}
\end{align}

{\bf step 4} The main linear term.\\

$\Lambda P_{B_0}$ is only approximate element of the kernel of $H^*_{B_1}$. The corresponding linear term $(\e,H_{B_1}^*(\Lambda P_{B_0}))$ on the RHS of \fref{vhoheoheo} is therefore potentially a highly problematic term. 
The control of this term requires the improved local estimate \fref{localizegviojdoboot}. We claim:
\be
\label{keylinearterm}
\left|(H_{B_1}\e,\Lambda P_{B_0})\right|\lesssim \frac{b^{2k}}{{|\log b|}}.
\ee
{\it Proof of \fref{keylinearterm}}: Let us first compute $H_{B_1}^*(\Lambda P_{B_0})$. Observe first from space localization that 
\def\S{{\mathcal S}}
$$H_{B_1}^*(\Lambda P_{B_0})=H_{B_0}^*(\Lambda P_{B_0})+{\mathcal S},\qquad 
\S:=k^2\frac{f'(P_{B_1})-f'(P_{B_0})}{y^2} \Lambda P_{B_0}$$ 
with $\S$ supported only on the set $y\in [{B_0},2B_0]$.

Rescaling (\ref{defpsitilda}), we find that $(P_{B_0})_{\lambda}$ satisfies: $$\Delta(P_{B_0})_{\lambda}-\frac{b^2}{\lambda^2}D\Lambda (P_{B_0})_{\lambda}-\frac{f((P_{B_0})_{\lambda})}{y^2}=-\frac{(\Psi_{B_0})_{\lambda}}{\lambda^2}.$$ Differentiating this relation with respect to $\lambda$ and evaluating the result at $\lambda=1$ yields: $$H_{B_0}\Lambda P_{B_0}+2b^2D\Lambda P_{B_0}=2\Psi_{B_0}+\Lambda \Psi_{B_0}$$ or equivalently from (\ref{deflstar}): $$H_{B_0}^*\Lambda P_{B_0}=2\Psi_{B_0}+\Lambda \Psi_{B_0}.$$ We thus rewrite the main linear term in (\ref{vhoheoheo}):
$$(H_{B_1}\e,\Lambda P_{B_0})=(\e,H_{B_1}^*\Lambda P_{B_0})=(\e,2\Psi_{B_0}+\Lambda \Psi_{B_0}+\S).$$ Let us now  define
\be
\label{defeb}
e_b=\frac{\left(2\Psi_{B_0}+\Lambda \Psi_{B_0}+\S,\Lambda Q\right)}{(\Lambda Q,\chi_M\Lambda Q)},
\ee 
we claim that we can find $\Sigma_b$ solution to:
\be
\label{defsigma}
H\Sigma_b=2\Psi_{B_0}+\Lambda \Psi_{B_0}+\S-e_b\chi_M\Lambda Q
\ee
with the property that 
$$
\Sigma_b=\Sigma_b^1+\Sigma_b^2,
$$
where $Supp(\Sigma^1_b,A\Sigma_b^2)\subset\{y\leq 2B_0\}$
and
\begin{align}
& |\Sigma^1_b(y)|_{L^{\infty}}\lesssim b^{k},\label{estsigmab}\\
 &|A\Sigma^2_b(y)|\lesssim 
\frac {b^{k+1}}{|\log b|}{\bf 1}_{y\le 2B_0} +\frac {b^{k+1}}{\log M}{\bf 1}_{y\le 2M}+{b^{k+1}}{\bf 1}_{B_0\le y\le 2B_0}.
\label{estsigmab'}
\end{align}
Assume \fref{estsigmab},\fref{estsigmab'}. We then use the orthogonality condition \fref{orthe} 
and \eqref{harybis}, \eqref{harybis-non} to estimate:

{\it case} $k\ge 2$:
\bea
\nonumber \left|(\e,2\Psi_{B_0}+y\cdot\nabla\Psi_{B_0})\right|& = & \left|(\e,2\Psi_{B_0}+y\cdot\nabla\Psi_{B_0}-e_b\chi_M\Lambda Q)\right|=(A^*A\e,\Sigma_b)\\
\nonumber  &&\hskip -8pc\lesssim  b^{k}\left(\int_{y\leq 2B_0}(A^*A\e)^2\right)^{\frac{1}{2}}\left(\int _{y\leq 2B_0}1\right)^{\frac{1}{2}}\\ &&\hskip -8pc +b^{k+1}\left(\int_{y\leq 2B_0}\frac {(A\e)^2}{y^2}
\left (\frac {1}{|\log^2 b|}{\bf 1}_{y\le 2B_0} +\frac {1}{\log^2 M}{\bf 1}_{y\le 2M}+ {\bf 1}_{B_0\le y\le 2B_0}\right)\right)^{\frac{1}{2}}\left(\int _{y\leq 2B_0} y^2 \right)^{\frac{1}{2}}\notag\\
&&\hskip -8pc  \lesssim  b^{k-1}\left(\int_{y\leq 2B_0}|A^*A\e|^2\right)^{\frac{1}{2}} + b^{k-1} 
\left(\int_{y\leq 2B_0} \frac{|A\e|^2}{y^2}\right)^{\frac{1}{2}}.\hskip 3pc\label{estnonejkljlsoe}
\eea
From \fref{estpotentialvone}:
\begin{align}
\int_{y\leq 2B_0}|A^*A\e|^2 & =  \int_{y\leq 2B_0}|\partial_y(A\e)+\frac{1+V^{(1)}}{y}A\e|^2\notag\\
& \lesssim  \int_{y\leq 2B_0}\left[|\partial_y(A\e)|^2+\frac{k^2+1+2V^{(1)}+V^{(2)}}{y^2}(A\e)^2\right]\notag\\
& \lesssim  \lambda^2\mathcal E_{\sigma}\label{eq:A-coer}
\end{align}
and thus from \fref{localizegviojdoboot} for $k\ge 2$:
\be
\label{pjvpovjpvje}
\int_{y\leq 2B_0}\frac{|A\e|^2}{y^2}+\int_{y\leq 2B_0}|A^*A\e|^2\lesssim \lambda^2\mathcal E_{\sigma}\lesssim \frac{b^{2k+2}}{|\log b|^2}.
\ee
Inserting this into \fref{estnonejkljlsoe} yields:
$$\left|(\e,2\Psi_{B_0}+y\cdot\nabla\Psi_{B_0})\right|\lesssim  b^{k-1}
\left(\frac{b^{2k+2}}{|\log b|^2}\right)^{\frac{1}{2}}\lesssim \frac{b^{2k}}{{|\log b|}},
$$
which gives \fref{keylinearterm}.\\
{\it case} $k=1$:
We first obtain the bound 
\be
\label{vnkoeoio}
\left|(\e,2\Psi_{B_0}+y\cdot\nabla\Psi_{B_0})\right|\lesssim  b^{k-1}
\left(\frac{b^{2k+2}}{|\log b|^2}\right)^{\frac{1}{2}}\lesssim \frac{b^{2k}}{{\sqrt{|\log b|}}},
\ee
Using \fref{estsigmab},\fref{estsigmab'}, the orthogonality condition \fref{orthe} 
and \eqref{harybis}, \eqref{harybis-non} we obtain:
\bea
\nonumber \left|(\e,2\Psi_{B_0}+y\cdot\nabla\Psi_{B_0})\right|& = & \left|(\e,2\Psi_{B_0}+y\cdot\nabla\Psi_{B_0}-e_b\chi_M\Lambda Q)\right|=(A^*A\e,\Sigma_b)\\
\nonumber  &&\hskip -8pc\lesssim  b^{k}\left(\int_{y\leq 2B_0}(A^*A\e)^2\right)^{\frac{1}{2}}\left(\int _{y\leq 2B_0}1\right)^{\frac{1}{2}}\\ &&\hskip -8pc +b^{k+1}\left(\int_{y\leq 2B_0}(A\e)^2
\left (\frac {1}{|\log^2 b|}{\bf 1}_{y\le 2B_0} +\frac {1}{\log^2 M}{\bf 1}_{y\le 2M}+ {\bf 1}_{B_0\le y\le 2B_0}\right)\right)^{\frac{1}{2}}\left(\int _{y\leq 2B_0}1 \right)^{\frac{1}{2}}\notag\\
&&\hskip -8pc  \lesssim  b^{k-1}\left(\int_{y\leq 2B_0}|A^*A\e|^2\right)^{\frac{1}{2}} + b^{k-1} 
\sqrt{|\log b|}\left(\int_{y\leq 1}|A\e|^2+\int_{y\leq 2B_0}|\nabla A\e|^2\right)^{\frac{1}{2}}.\hskip 3pc\label{estnonejkljlsoebis}
\eea
Since by \eqref{eq:A-coer} and \fref{localizegviojdoboot}:
\be
\label{pjvpovjpvjebis}
\int_{y\leq 1}|A\e|^2+\int_{y\leq 2B_0}|\nabla A\e|^2+\int_{y\leq 2B_0}|A^*A\e|^2\lesssim \lambda^2\mathcal E_{\sigma}\lesssim \frac{b^{2k+2}}{|\log b|^2}.
\ee
we obtain
$$\left|(\e,2\Psi_{B_0}+y\cdot\nabla\Psi_{B_0})\right|\lesssim  b^{k-1}\sqrt{|\log b|}\left(\frac{b^{2k+2}}{|\log b|^2}\right)^{\frac{1}{2}}\lesssim \frac{b^{2k}}{\sqrt{|\log b|}}.
$$
To obtain the stronger estimate 
\be
\label{iehioeh}
\left|(\e,2\Psi_{B_0}+y\cdot\nabla\Psi_{B_0})\right|\lesssim \frac{b^{2k}}{{{|\log b|}}},
\ee
we claim that we can redefine the decomposition $\Sigma_b=\tilde \Sigma_b^1+\tilde\Sigma_b^2$ so that 
\eqref{estsigmab}, \eqref{estsigmab'} are replaced by the estimates
\begin{align}\label{eq:str-S}
|\tilde\Sigma_b^1|&\lesssim b,\\
|A\tilde \Sigma^2_b(y)|&\lesssim 
\frac {b^{2}}{|\log b|}{\bf 1}_{y\le 2B_0} +\frac {b^{2}}{\log M}{\bf 1}_{y\le 2M}.\label{eq:str-S2}
\end{align}
The absence of the term $b^{k+1}{\bf 1}_{B_0\le y\le 2B_0}$ in \eqref{eq:str-S2} eliminates the additional
logarithmic divergence in \eqref{estnonejkljlsoebis} and leads to the desired bound. We omit the 
straightforward details.

\begin{remark} The gain in \fref{iehioeh} with respect to the simpler bound \fref{vnkoeoio} will allow us to obtain the $O(\frac{b^2}{|\log b|})$ estimate on the remaining terms in the RHS of \fref{firstcontrol}. This in turn will lead to the $O(1)$ term in the derivation of the blow up speed
\fref{universallawkgeq2} after reintegration of the modulation equations, see in particular \fref{lawfors}.
\end{remark}

{\it Proof of \fref{estsigmab},\eqref{estsigmab'}}: Let $$g_b=2\Psi_{B_0}+\Lambda \Psi_{B_0}+\S-e_b\chi_M\Lambda Q,$$ so that 
\be
\label{cnoeoe}
(g_b,\Lambda Q)=0
\ee
from \fref{defeb}. Then, as in \eqref{defulinearsolver}, a solution to \fref{defsigma} is given by
$$\Sigma_b(y)=\Gamma(y)\int_{0}^y \Lambda Q g_budu-\Lambda Q(y)\int_1^{y}g_b\Gamma udu=\Sigma_b^1+\Sigma_b^2.$$ 
The compact support of $\Psi_{B_0}$ and hence of $g_b$ in $y\leq 2B_0$ and \fref{cnoeoe} ensure $Supp(\Sigma^1_b)\subset\{y\leq 2B_0\}.$ On the other hand, using that $A(\Lambda Q)=0$,
\be\label{eq:AL}
A\Sigma_b^2=\Lambda Q g_b \Gamma y
\ee
and the property $Supp(A\Sigma^2_b)\subset\{y\leq 2B_0\}.$ follows.
We now turn to the proof of the $L^{\infty}$ estimates \fref{estsigmab}, \fref{estsigmab'}.\\
In what follows we will use the bound 
\be\label{eq:bound-S}
|\S|\lesssim b^{2k+2} {\bf 1}_{B_0\le y\le 2B_0},
\ee
which easily follows from 
$$
|\frac{f'(P_{B_1})-f'(P_{B_0})}{y^2} \Lambda P_{B_0}|\le \frac 1{y^2} |P_{B_1}-P_{B_0}| 
|\Lambda P_{B_0}|
$$
{\it case $k\geq 3$}: We use the bound from \fref{estpsibeventilde}, \fref{estpsiboddtilde}:
$$|\Psi_{B_0}|+|\Lambda \Psi_{B_0}|\lesssim b^{k+3}{\bf 1}_{y\leq 2B_0}+ 
b^{k+2}{\bf 1}_{B_0\le y\leq 2B_0},$$ 
which yields:
$$|e_b|\lesssim \int\frac{b^{k+2}y^k}{1+y^{2k}}ydy\lesssim b^{k+2},$$ 
$$|\Sigma^1_b(y)|\lesssim \frac{1+y^{2k}}{y^{k}}\int_y^{2B_0}\frac{b^{k+2}u^k}{1+u^{2k}}udu
\lesssim  b^k.
$$
On the other hand, taking into account that $|\Lambda Q \Gamma|\lesssim 1$,
$$
|A\Sigma_b^2(y)|=|\Lambda Q \Gamma \, g_b y|\le b^{k+2} {\bf 1}_{y\le 2B_0}+ b^{k+1} {\bf 1}_{B_0\le y\le 2B_0}
$$
{\it case $k=2$}: We estimate from \fref{estpsibeventildektwo}:
$$|\Psi_{B_0}|+|\Lambda \Psi_{B_0}|\lesssim  \frac{b^4y^2}{1+y^{4}}{\bf 1}_{y\leq B_0}+b^4{\bf 1}_{B_0\leq y\leq 2B_0},
$$
and hence:
$$|e_b|\lesssim \int \frac{y^k}{1+y^{2k}}\left[\frac{b^4y^2}{1+y^{4}}+b^4{\bf 1}_{B_0\leq y\leq 2B_0},\right]ydy\lesssim b^{4},$$
\bee
 |\Sigma^1_b(y)|& \lesssim &\frac{1+y^{4}}{y^{2}}\int_y^{2B_0}\frac{u^2}{1+u^{4}}\left[\frac{b^4u^2}{1+u^{4}}+b^4{\bf 1}_{B_0\leq u\leq 2B_0}\right]udu\lesssim  b^2.
 \eee
 $$
 |A\Sigma_b^2(y)|\lesssim |g_b y|\lesssim  b^{4} {\bf 1}_{y\le 2B_0}+ b^{3} {\bf 1}_{B_0\le y\le 2B_0}
 $$
{\it case $k=1$}: We estimate from \fref{estpsiboddtildekone}:
\bee|\Psi_{B_0}|+|\Lambda \Psi_{B_0}|& \lesssim & \frac{b^2}{|\log b|}\frac{y}{1+y^2}{\bf 1}_{y\leq 2B_0}+\frac{b^2}{y}{\bf 1}_{B_0\leq y\leq 2B_0},
\eee
and hence:
$$\log M |e_b|\lesssim \int_{y\leq 2B_0}\frac{y}{1+y^{2}}\left[\frac{b^2y}{|\log b|(1+y^{2})}+\frac{b^2}{y}{\bf 1}_{B_0\leq y\leq 2B_0},\right]ydy\lesssim b^2,$$ 
\bee
 |\Sigma^1_b(y)|& \lesssim &\frac{1+y^{2}}{y}\int_y^{2B_0}\frac{u}{1+u^{2}}\left[\frac{b^2u}{|\log b|(1+u^{2})}+\frac{b^2}{u}{\bf 1}_{B_0\leq u\leq 2B_0}+\frac{b^2}{\log M (1+u)}{\bf 1}_{u\leq 2M}\right]udu
\lesssim  b.
 \eee
 $$
 |A\Sigma_b^2(y)|\lesssim |g_b y|\lesssim \frac {b^2}{|\log b|} {\bf 1}_{y\le 2B_0} + 
 \frac{b^2}{\log M} {\bf 1}_{y\le 2M}+ b^2{\bf 1}_{B_0\le y\le 2B_0}
 $$
This concludes the proof of \fref{estsigmab}, \fref{estsigmab'}.\\
\vskip 1pc
{\it Proof of \fref{eq:str-S},\eqref{eq:str-S2}}: As before let 
$$g_b=2\Psi_{B_0}+\Lambda \Psi_{B_0}+\S-e_b\chi_M\Lambda Q,$$ so that 
$$\Sigma_b(y)=-\Gamma(y)\int_{y}^\infty \Lambda Q g_budu-\Lambda Q(y)\int_1^{y}g_b\Gamma udu
$$
We now recall that according to \eqref{formulapsitilde}
\begin{align*}
\Psi_{B_0} &=  \chi_{B_0}\Psi_b+\frac{k^2}{y^2}\left\{f(P_{B_0})-\chi_{B_0}f(Q_b)\right\}-(Q_b-\pi)\Delta \chi_{B_0}-2\chi_{B_0}'Q'_b\\
& +  b^2\left\{(Q_b-\pi)D\Lambda \chi_{B_0}+2y^2\chi_{B_0}'Q'_b\right\}
\end{align*}
Set
\begin{align*}
\Psi_{B_0} ^1&=\frac 2y\Delta \chi_{B_0}-\frac{4}{y^2}\chi_{B_0}' 
 -  \frac {2b^2}{y}\left\{D\Lambda \chi_{B_0}-2y\chi_{B_0}'\right\},\\
 \Psi_{B_0} ^2&=\chi_{B_0}\Psi_b+\frac{1}{y^2}\left\{f(P_{B_0})-\chi_{B_0}f(Q_b)\right\}-(Q_b-\pi+\frac 2y)\Delta \chi_{B_0}\\&\quad-2\chi_{B_0}'(Q_b+\frac 2y)'
 +  b^2\left\{(Q_b-\pi+\frac 2y)D\Lambda \chi_{B_0}+2y^2\chi_{B_0}'(Q_b-\frac 2y)'\right\}
\end{align*}
and define
\begin{align*}
&\Sigma^1_b(y)=-\Gamma(y)\int_{y}^\infty \Lambda Q g_budu-\frac 14\Lambda Q(y)\int_0^{y} \pa_u (u^2\Psi_{B_0}^1)
u du,\\
&\Sigma^2_b(y)= -\Lambda Q(y)\int_0^{1} \pa_u (u^2\Psi_{B_0}^1)
\Gamma du-\Lambda Q(y)\int_1^{y}\left (g_b-2\Psi^1_{B_0}-\Lambda\Psi_{B_0}^1\right)\Gamma udu\\
&\qquad\quad-\Lambda Q(y)\int_0^{y} \pa_u (u^2\Psi_{B_0}^1)
(\Gamma -\frac u4) du
\end{align*}
Therefore,
\begin{align*}
A\Sigma^2_b(y)&=-(g_b-2\Psi^1_{B_0}-\Lambda\Psi_{B_0}^1)\Lambda Q \Gamma y=
(2\Psi^2_{B_0}+\Lambda \Psi^2_{B_0}+\S-e_b\chi_M\Lambda Q) \Lambda Q \Gamma y\\ &
-\pa_y (y^2\Psi_{B_0}^1)
(\Gamma -\frac y4)\Lambda Q
\end{align*}
and thus we need to show that 
$$
\frac 1y |\pa_y (y^2\Psi_{B_0}^1)
(\Gamma -\frac y4)|+y\,|2\Psi^2_{B_0}+\Lambda \Psi^2_{B_0}+\S-e_b\chi_M\Lambda Q|\lesssim \frac {b^2}{|\log b|} {\bf 1}_{y\le 2B_0} + 
 \frac{b^2}{\log M} {\bf 1}_{y\le 2M}
$$
From \eqref{estpsiboddkequalone} we have that on the support of $\chi_{B_0}$
$$
|\Psi_b|+|\Lambda \Psi_b|\lesssim \frac {b^2}{|\log b|} \frac {y}{1+y^2}
$$
Furthermore, \eqref{eq:bound-S} gives
$$
|\S|\lesssim b^{2k+2} {\bf 1}_{B_0\le y\le 2B_0},
$$
and 
$$
|e_b \chi_M \Lambda Q|\lesssim \frac {b^2}{|\log M|} \frac {y}{1+y^2}{\bf 1}_{y\le 2M}.
$$
Using that $f(\pi)=0$, $f'(\pi)=1$, we also obtain
\begin{align*}
|\frac 2{y^2}&(f(P_{B_0})-\chi_{B_0}f(Q_b))+ \Lambda\left[\frac 1{y^2}(f(P_{B_0})-\chi_{B_0}f(Q_b))\right]=
\frac 1y \pa_y \left [f(P_{B_0})-\chi_{B_0}f(Q_b)\right]\\&= 
\frac 1y \pa_y \left [P_{B_0}-\pi-\chi_{B_0}((Q_b-\pi)\right ]\\ &+
\frac 1y \pa_y \left [\int_0^1 \tau \int_0^1 \left (f''(\tau\tau' P_{B_0}) (P_{B_0}-\pi)^2-
\chi_{B_0} f''(\tau\tau' Q_b)  (Q_b-\pi)^2\right)\right ]\\&=
\frac 1y \pa_y \left [\int_0^1\tau \int_0^1 \left (f''(\tau\tau' P_{B_0}) (P_{B_0}-\pi)^2-
\chi_{B_0} f''(\tau\tau' Q_b)  (Q_b-\pi)^2\right)\right]
\lesssim\frac 1{y^4}.
\end{align*}
Since $f(P_{B_0})-\chi_{B_0}f(Q_b)$ vanishes outside the interval $B_0\le y\le 2B_0$, the above bound 
can be replaced by $b^4{\bf 1}_{B_0\le y\le 2B_0}$. 
The estimate for the remaining part of $\Psi_{B_0}^2$ follows from the bounds
\begin{align*}
&|\frac {d^m}{dy^m} (Q_b-\pi+\frac 2y)|\lesssim |\frac {d^m}{dy^m}(Q-\pi+\frac 2y)|+b^2|\frac{d^m}{dy^m} T_1|\lesssim
\frac 1{y^{3+m}} + \frac 1{|\log b| y^{1+m}},\\
&|\frac {d^m}{dy^m} \Psi_{B_0}^1|\lesssim \frac {b^2}{y^{1+m}},\qquad |\Gamma - \frac y4|\lesssim 1
\end{align*}
which hold for $B_0\le y\le 2B_0$ (in particular on the support of $\chi'_{B_0}$) and follow from 
\eqref{asympttone}, \eqref{expansionqr} and \eqref{ezpansionLQH}.

These estimates imply the desired bound \eqref{eq:str-S2}.

\vskip 1pc
To prove \eqref{eq:str-S} it suffices to show that $\Lambda Q(y)\int_0^{y} \pa_u (u^2\Psi_{B_0}^1)
\Gamma du$ is supported in
$y\le 2B_0$ and establish the bound
$$
|\Lambda Q(y)\int_0^{y} \pa_u (u^2\Psi_{B_0}^1)
\Gamma du|\lesssim b
$$
We argue that a careful choice of $B_0$ ensures that 
\be\label{eq:zero}
\int_0^{\infty} \pa_u (u^2\Psi_{B_0}^1)u du=0.
\ee
Assuming this we immediately conclude the statement about the support, since $\Psi_{B_0}^1$ is supported
in $B_0\le y\le 2B_0$.  Furthermore, from \eqref{eq:qpsi} and \eqref{eq:qpsi2} for $y\ge 2B_0$
$$
|\Lambda Q(y)\int_0^{y} \pa_u (u^2\Psi_{B_0}^1)
\Gamma du|\lesssim \frac y{1+y^2}\int_0^y \frac {1+u^2}{u} {b^2}{\bf 1}_{B_0\le u\le 2B_0} du\lesssim b.
$$ 
To show \eqref{eq:zero} we rewrite
$$
y^2\Psi_{B_0}^1= 2y(1-b^2y^2)\chi^{''}_{B_0} - 2 \chi_{B_0}'
$$
\begin{align*}
\int_0^{\infty} \pa_u (u^2\Psi_{B_0}^1)u du&=-\int_0^{\infty} u^2\Psi_{B_0}^1 du =
-\int_0^{\infty}\left (2y(1-b^2y^2)\chi^{''}_{B_0} - 2 \chi_{B_0}'\right) dy\\&=
-2+2\int_0^{\infty}(1-3b^2y^2)\chi'_{B_0} dy\\&=
-4+12b^2\int_0^{\infty}y\chi_{B_0}dy=-4+12b^2B_0^2\int_0^{\infty}y\chi dy
\end{align*}
Therefore, the choice 
$$
B_0^2=\frac 1{3b^2\int_0^\infty y\chi dy}
$$
gives the desired property.\\

{\bf step 4} Lower order linear terms in $\e$.\\

We are left with estimating the third line on the RHS of \fref{vhoheoheo}. We first claim:
\bea
\label{estbcohpow}
 \nonumber & & \left|\left( \partial_s^2\eb+b(\partial_s\eb+2\Lambda\partial_s\eb)+b_s \Lambda \eb,\Lambda P_{B_0}\right)-\frac{d}{ds}\left[(\partial_s\eb,\Lambda P_{B_0})+b(\e+2\Lambda\eb,\Lambda P_{B_0})\right]\right|\\
& \lesssim & \frac{b^{2k}}{|\log b|}.
\eea
Indeed, we integrate by parts to obtain:
\bea
\label{jhoehceoce}
\nonumber & & \left( \partial_s^2\eb+b(\partial_s\eb+2\Lambda\partial_s\eb)+b_s \Lambda \eb,\Lambda P_{B_0}\right)=\frac{d}{ds}\left[(\partial_s\eb,\Lambda P_{B_0})+b(\e+2\Lambda\eb,\Lambda P_{B_0})\right]\\
\nonumber & - & b_s\left[(\partial_s\eb+b\Lambda \e,\Lambda \frac{\partial P_{B_0}}{\partial b})+(\e,\Lambda P_{B_0}+b\Lambda \frac{\partial P_{B_0}}{\partial b})+b(\Lambda\e,\Lambda\frac{\partial P_{B_0}}{\partial b})+ (\Lambda\e,\Lambda P_{B_0})\right]\\
& = & \frac{d}{ds}\left[(\partial_s\eb,\Lambda P_{B_0})+b(\e+2\Lambda\eb,\Lambda P_{B_0})\right]- b_s\left[(\partial_s\eb+b\Lambda \e,\Lambda \frac{\partial P_{B_0}}{\partial b})+(\e,\Phi_b)\right]
\eea
with
\be
\label{defhofhiroh}
\Phi_b=-\Lambda P_{B_0}-\Lambda^2P_{B_0}-b\Lambda\frac{\partial P_{B_0}}{\partial b}-b\Lambda^2\frac{\partial P_{B_0}}{\partial b}.
\ee
We now estimate the RHS of \fref{jhoehceoce}. To wit, let 
\be
\label{defrrb}
r_b=\frac{(\Phi_b,\Lambda Q)}{(\Lambda Q,\chi_M\Lambda Q)},
\ee
we claim that we can find $\tilde{\Phi}=\tilde{\Phi}_1+\tilde{\Phi}_2$ such that 
$$
H\tilde{\Phi}=\Phi_b-r_b\chi_M\Lambda Q, \ \ Supp(\tilde{\Phi}_1)\cup Supp(A\tilde{\Phi}_2)\subset[0,2B_0],
$$
and 
\be
\label{vhkoehoe}
|\tilde{\Phi}_1|_{L^{\infty}}\lesssim \left\{ 
\begin{array}{ll}
1 \ \ \mbox{for} \ \ k\geq 2\\
					\frac{|\log b|}{b}\ \ \mbox{for} \ \ k=1,
					\end{array} \right .
\ee
\be
\label{vhkoehoebidib}
|A\tilde{\Phi}_2(y)|\lesssim \left\{ 
\begin{array}{ll}
\frac{y^{k+1}}{1+y^{2k}}{\bf 1}_{y\leq 2B_0} \ \ \mbox{for} \ \ k\geq 2\\
					\frac{y^2}{1+y^{2}}\left[{\bf 1}_{y\leq 2B_0}+|\log b|{\bf 1}_{y\leq 2M}\right]\ \ \mbox{for} \ \ k=1,
					\end{array} \right .
\ee
Let us assume \fref{vhkoehoe}, \fref{vhkoehoebidib} and conclude the proof of \fref{estbcohpow}.\\
{\it case $k\geq 2$}: First recall from \fref{poitwisebsboot} the  bound: $$|b_s|\lesssim b^{k+1}.$$ Moreover, \fref{estbodd}, \fref{estp[rgehdiobisbis} imply: 
\be
\label{estinetclcjoierm}
|\Lambda ^l\partial_bP_{B_0}|\lesssim C(M) b{\bf 1}_{y\leq 2B_0}, \ \ 0\leq l\leq 2.
\ee 
We conclude from \fref{controldt}, \fref{poitwisebsboot}, \fref{poitnwiseboundWboot}:
\bea
& & |b_s|\left|(\partial_s\eb+b\Lambda \e,\Lambda \frac{\partial P_{B_0}}{\partial b})\right|\lesssim  C(M) b^{k+1}\lambda |\partial_tw|_{L^{\infty}}\int_{y\leq 2B_0}b\notag\\
& \lesssim &C(M) \lambda b^{k}\left(|A^*_{\lambda}W|_{L^2}^2+|\partial_tW|_{L^2}^2\right)^{\frac{1}{2}}\lesssim C(M) b^{2k+1}.
\eea
Next, from \fref{vhkoehoe}, \fref{vhkoehoebidib} and the choice of the orthogonality condition \eqref{orthe}:
\bee
|b_s||(\e,\Phi_b)|& = &   |b_s||(A^*A\e,\tilde{\Phi}_b)|\lesssim b^{k+1}\frac{1}{b}|A^*A\e|_{L^2}+b^{k+1}|\frac{A\e}{y}|_{L^2}\left(\int_{y\leq 2B_0}\frac{y^{2k+2}y^2}{1+y^{4k}}\right)^{\frac{1}{2}}\\
& \lesssim & b^{2k+2}\frac{1}{b}\lesssim b^{2k+1},
\eee
where we used \fref{eq:coerc}, \fref{poitnwiseboundWboot}.\\
{\it case $k=1$}: By \eqref{vhovhohvoh}
$$
|\Lambda ^l\partial_bP_{B_0}|\lesssim {\bf 1}_{y\leq 2B_0}, \ \ 0\leq l\leq 2.
$$
Thus, using \fref{controldt}, \fref{poitwisebsboot}, \fref{poitnwiseboundWboot}, \fref{crudeboundb}:
\bea
\label{eq:hte}
& & |b_s|\left|(\partial_s\eb+b\Lambda \e,\Lambda \frac{\partial P_{B_0}}{\partial b})\right|\lesssim \frac{b^2}{|\log b|}\lambda |\partial_tw|_{L^{\infty}}\int_{y\leq 2B_0}ydy\notag\\
\nonumber & \lesssim &\frac {\lambda}{|\log b|}\left(|A^*_{\lambda}W|_{L^2}^2+|\partial_tW|_{L^2}^2\right)^{\frac{1}{2}}\lesssim \frac{b^2}{|\log b|}.
\eea
Next from \fref{vhkoehoe} and the choice of the orthogonality condition \fref{orthe}:
\bee
& & |b_s||(\e,\Phi_b)|  =   |b_s||(A^*A\e,\tilde{\Phi}_b)|\\
& \lesssim & \frac{b^{2}}{|\log b|}\left[\frac{|\log b|}{b^2}|A^*A\e|_{L^2(y\leq 2B_0)}+|\frac{A\e}{y}|_{L^2(y\leq 2B_0)}\left(\int \frac{y^6}{1+y^4}[{\bf 1}_{y\leq 2B_0}+\log^2 b {\bf 1}_{y\leq 2M}]\right)^{\frac{1}{2}}\right].
\eee
We then observe from \fref{pjvpovjpvjebis} and \fref{harybis-non}:
\be\label{eq:boundA}
|\frac{A\e}{y}|_{L^2(y\leq 2B_0)}\lesssim |\log b|\frac{b^2}{|\log b|}\lesssim b^2
\ee 
and hence from the refined bound \fref{pjvpovjpvjebis}:
\bee
|b_s||(\e,\Phi_b)|  & = &  \frac{b^2}{|\log b|}\left[\frac{|\log b|}{b^2}\frac{b^2}{|\log b|}+b^2\frac{1}{b^2}\right]\lesssim \frac{b^2}{|\log b|}.
\eee
This concludes the proof of \fref{estbcohpow}.\\
{\it Proof of \fref{vhkoehoe}, \fref{vhkoehoebidib}}:  We let 
\bea
\label{def[hinbovnfo}
\nonumber \tilde{\Phi}_b& = & \Gamma (y)\int_0^{y}\Lambda Q(\Phi_b-r_b\chi_M\Lambda Q) udu-\Lambda Q(y)\int_0^{y}\Gamma (\Phi_b-r_b\chi_M\Lambda Q)udu\\
& = & \tilde{\Phi}_1+\tilde{\Phi}_2.
\eea
The support of $\Phi_b$ belongs to the set $y\le 2B_0$. Therefore  $Supp(\tilde{\Phi}_1)\subset[0,2B_0]$ by 
the choice of $r_b$ in \fref{defrrb} and $Supp(A\tilde{\Phi}_2)\subset[0,2B_0]$ which follows from the identity
$$A\tilde{\Phi}_2= \Lambda Q\Gamma(\Phi_b-r_b\chi_M\Lambda Q)y.$$\\
{\it case $k\geq 2$}: We derive from \fref{defhofhiroh}, \fref{estp[rgehdiobisbis}, \fref{estbodd}  the bound:
 $$|\Phi_b|\lesssim \frac{y^{k}}{1+y^{2k}}{\bf 1}_{y\leq 2B_0},$$ and hence $r_b$, given by \fref{defrrb}, satisfies:
$$|r_b|\lesssim 1.$$ We then estimate:
$$
|\tilde{\Phi}_1(y)| \lesssim   \frac{1+y^{2k}}{y^k}\int_y^{2B_0}\frac{u^k}{1+u^{2k}}\frac{u^k}{1+u^{2k}}udu
\lesssim  \frac{1}{1+y^{k-2}}\lesssim 1
$$
and \fref{vhkoehoe} follows. Similarily, $$|A\tilde{\Phi}_2(y)|\lesssim y|\Phi_b-r_b\chi_M\Lambda Q|\lesssim \frac{y^{k+1}}{1+y^{2k}}{\bf 1}_{y\leq 2B_0}$$ and \fref{vhkoehoebidib} follows.\\
{\it case $k=1$}: We estimate from \fref{defhofhiroh}, \eqref{vhovhohvoh}:
$$|\Phi_b|\lesssim \frac{y}{1+y^2}{\bf 1}_{y\leq 2B_0}$$ from which $r_b$, given by \fref{defrrb}, satisfies:
$$|r_b|\lesssim |\log b|$$ and 
$$|\tilde{\Phi}_1(y)|\lesssim   \frac{1+y^{2}}{y}\int_y^{2B_0}\frac{u}{1+u^{2}}\frac{u}{1+u^{2}}\left[1+|\log b|{\bf 1}_{y\leq M}\right]udu \lesssim  \frac{|\log b|}{b}
$$
and \fref{vhkoehoe} follows. Next,
$$|A\tilde{\Phi}_2(y)|\lesssim y|\Phi_b-r_b\chi_M\Lambda Q|\lesssim \frac{y^2}{1+y^{2}}{\bf 1}_{y\leq 2B_0}+|\log b|\frac{y^2}{1+y^2}{\bf 1}_{y\leq 2M}$$ and \fref{vhkoehoebidib} follows.\\
 This concludes the proof of \fref{vhkoehoe}\, \fref{vhkoehoebidib}.\\
 
{\bf step 5} Control of the nonlinear term.\\ 

{\it case $k\geq 2$}: There holds from \fref{estun}, \fref{contorlocoervcie}, \fref{poitnwiseboundWboot}:
\be
\label{nolineaireone}
\left|(\frac{N(\e)}{y^2},\Lambda P_{B_0})\right|\lesssim \int |\e|^2\frac{y^k}{y^2(1+y^{2k})}\lesssim \int\frac{|\e|^2}{y^4}\lesssim \lambda^2|A^*_{\lambda}W|_{L^2}^2\lesssim b^{2k+2}.
\ee

{\it case $k=1$}: From \fref{estnonlineargloballocalized}
\bea
\label{nolineaireonebis}
\nonumber \left|(\frac{N(\e)}{y^2},\Lambda P_{B_0})\right|& \lesssim & \left(\int_{y\leq 2B_0}\frac{|\e|^4}{y^4}\right)^{\frac{1}{2}}|\Lambda P_{B_0}|_{L^2}\lesssim |\log b|b^{\frac{3}{4}}\lambda|A^*_{\lambda}W|_{L^2}\\
& \lesssim b^{2+\frac{1}{2}}.
\eea

{\bf step 5} Control of $G(b)$ and ${\mathcal{I}}$.\\

Using estimates  \fref{rewritnghgllhs}, \fref{ourgoingflux}, \eqref{eq:zeta1}, \eqref{eq:zeta2}, \fref{keylinearterm}, \fref{estbcohpow}, \fref{nolineaireone}, \fref{nolineaireonebis} in conjunction with the the algebraic formula \fref{vhoheoheo} concludes the proof of \fref{firstcontrol}. It remains to prove \fref{estkeyfb}, (\ref{estimateione}).\\
{\it Proof of \fref{estkeyfb}}: Recall the formula \fref{deffb} for $G(b)$. We compute 
$$
\Lambda P_{B_0} = \chi_{B_0} \Lambda Q_b + \Lambda \chi_{B_0} (Q_b-a) =\chi_{B_0} \Lambda Q + \chi_{B_0}\Lambda (Q_b-Q)+ \Lambda \chi_{B_0} (Q_b-a)
$$
It then follows from Proposition \ref{propqb} that for any $k\ge 1$
$$
|\Lambda P_{B_0}-\chi_{B_0} \Lambda Q|\lesssim C(M) b^2 \frac{y^k}{1+y^{2k-2}} {\bf 1}_{y\le 2B_0}
$$
As a consequence,
\be
\label{hoeoehvdophone}
|\Lambda P_{B_0}|_{L^2}^2=\left\{\begin{array}{ll}|\Lambda Q|_{L^2}^2 +O(b^2)=|\Lambda Q|_{L^2}^2(1+o(1)) \ \ \mbox{for} \ \ k\geq2,\\
										|\chi_{B_0}\Lambda Q|_{L^2}^2+O(1)= 4 |\log b|+O(1))\ \ \mbox{for} \ \ k=1.
\end{array} \right .
\ee
Similarly, using \fref{crudeboundb}:
$$
|(\frac {\pa P_{B_0}}{\pa b}, \Lambda P_{B_0})|\lesssim \int_{y\leq 2B_0}\frac{y^k}{1+y^{2k}}\lesssim \left\{\begin{array}{ll}
|\log b|\ \ \mbox{for} \ \ k\geq 2,\\
\frac{1}{b} \ \ \mbox{for} \ \ k=1,
\end{array} \right .
$$ from which:
$$\left|\int_0^b b'(\frac{\partial P_{B_0}}{\partial b},\Lambda P_{B_0})db'\right|\lesssim \left\{\begin{array}{ll} {b^2}{|\log b|}\ \ \mbox{for} \ \ k\geq 2,\\
b \ \ \mbox{for} \ \ k=1,
\end{array} \right .
$$ 
which together with \fref{deffb}, \fref{hoeoehvdophone} concludes the proof of \fref{estkeyfb}.\\
{\it Proof of \fref{estimateione}}: We integrate by parts in space in \fref{defiun} to rewrite:
\be
\label{cneohebibi}
{\mathcal I}(s)=(\partial_s\e+b\Lambda \e,\Lambda P_{B_0})+b_s(\frac{\partial P_{B_0}}{\partial_b},\Lambda P_{B_0})-b(\e,\Lambda P_{B_0}+\Lambda^2P_{B_0})-b_s \left(\frac{\partial}{\partial b}(P_{B_1}-P_{B_0}), \Lambda P_{B_0})\right ).
\ee
The last term above has been estimated in step 3.
We let 
\be
\label{defuguos}
\tilde{r}_b=\frac{(\Lambda P_{B_0}+\Lambda^2P_{B_0},\Lambda Q)}{(\chi_M\Lambda Q,\Lambda Q)}
\ee
and claim that we can solve:
$$L\Theta_b=\Lambda P_{B_0}+\Lambda^2P_{B_0}-\tilde{r}_b\chi_M\Lambda Q$$ with $\Theta_b=\Theta_1+\Theta_2$, $Supp(\Theta_1)\cup Supp(A\Theta_2)\subset[0,2B_0]$ and
\be
\label{vhkoehoebisbis}
|\Theta_1|_{L^{\infty}}\lesssim \left\{ 
\begin{array}{ll}
1 \ \ \mbox{for} \ \ k\geq 2\\
					\frac{|\log b|}{b}\ \ \mbox{for} \ \ k=1,
					\end{array} \right .
\ee
\be
\label{vhkoehoebidibbisbis}
|A\Theta_2(y)|\lesssim \left\{ 
\begin{array}{ll}
\frac{y^{k+1}}{1+y^{2k}}{\bf 1}_{y\leq 2B_0} \ \ \mbox{for} \ \ k\geq 2\\
					\frac{y^2}{1+y^{2}}\left[{\bf 1}_{y\leq 2B_0}+|\log b|{\bf 1}_{y\leq 2M}\right]\ \ \mbox{for} \ \ k=1,
					\end{array} \right .
\ee
The proof of \fref{vhkoehoebisbis}, \fref{vhkoehoebidibbisbis} is completely similar to the one of \fref{vhkoehoe}, \fref{vhkoehoebidib} and left to the reader.\\
{\it case $k\geq 2$}: From \fref{controldt}, \fref{poitnwiseboundWboot}:
$$\left|(\partial_s\eb+b\Lambda \e,\Lambda  P_{B_0})\right|\lesssim \lambda |\partial_tw|_{L^{\infty}}|\Lambda P_{B_0}|_{L^1}\lesssim  |\log b|b^{k+1}\lesssim b^{2}.
$$
Next, from \fref{poitwisebsboot}:
$$\left|b_s(\frac{\partial P_{B_0}}{\partial b},\Lambda P_{B_0})\right|\lesssim b^{k+1}|\Lambda P_{B_0}|_{L^1}\lesssim b^2.$$ Finally, from \fref{poitnwiseboundWboot}, \fref{vhkoehoebisbis} and the choice of the orthogonality condition \fref{orthe}:
\bee
& & b|(\e,\Lambda P_{B_0}+\Lambda^2P_{B_0})  = b|(A^*A\e,\Theta_b)|\\
& \lesssim & b|A^*A\e|_{L^2}\frac{1}{b}+b|\frac{A\e}{y}|_{L^2}\left(\int_{y\leq 2B_0}\frac{y^2y^{2k+2}}{1+y^{4k}}\right)^{\frac{1}{2}}\\
& \lesssim & b|A^*A\e|_{L^2}\frac{1}{b}+b|\frac{A\e}{y}|_{L^2}\frac{1}{b}\lesssim b^{k+1}\lesssim b^2.
\eee

{\it case $k=1$}: From \fref{controldt}, \fref{poitnwiseboundWboot}:
$$\left|(\partial_s\eb+b\Lambda \e,\Lambda  P_{B_0})\right|\lesssim \lambda |\partial_tw|_{L^{\infty}}|\Lambda P_{B_0}
|_{L^1}\lesssim  \frac{b^{2}}{b}\lesssim b.
$$
Next, from \fref{poitwisebsboot}:
$$\left|b_s(\frac{\partial P_{B_0}}{\partial b},\Lambda P_{B_0})\right|\lesssim \frac{b^{2}}{|\log b|}|\Lambda P_{B_0}|_{L^1}\lesssim \frac{b}{|\log b|}.$$ Finally, from \fref{vhkoehoebisbis} and the choice of orthogonality condition \fref{orthe}:
\bee
& & b|(\e,\Lambda P_{B_0}+\Lambda^2P_{B_0}) =  b|(A^*A\e,\Theta_b)|\\
& \lesssim & b|A^*A\e|_{L^2(y\leq 2B_0)}\frac{|\log b|}{b^2}+b\left(\int_{y\leq 2B_0}\frac{(A\e)^2}{y^2}\right)^{\frac{1}{2}}\left(\int_{y\leq 2B_0}\frac{y^2y^4}{1+y^4}\left(1+|\log b|^2{\bf 1}_{y\leq 2M}\right)\right)^{\frac{1}{2}}\\
& \lesssim & |A^*A\e|_{L^2(y\leq 2B_0)}\frac{|\log b|}{b}+b\frac{b^2}{|\log b|}\frac{|\log b|}{b^2}\lesssim b
\eee
 where we used \fref{harybis-non}, the improved localized bound \fref{pjvpovjpvjebis} and \eqref{eq:boundA}.\\
This concludes the proof of \fref{estimateione}.\\

This concludes the proof of Proposition \ref{lemmaalgebra}.


\subsection{Proof of Theorem \ref{mainthm}}


We are now in position to conclude the proof of Theorem \ref{mainthm}.\\

First recall that finite time blow up is a consequence of Proposition \ref{bootstrap}. This coupled with the standard scaling lower bound:
$$\lambda(t)\leq T-t$$ 
 implies that the rescaled time $s$ is global: $$\frac{ds}{dt}=\frac{1}{\lambda}\geq \frac{1}{T-t} \ \ \mbox{and hence} \ \  s(t)\to+\infty \ \ \mbox{as} \ \ t\to T.$$ 
 
{\bf step 1} Derivation of the scaling law.\\

We begin with with the proof of \fref{universallawkgeq}, \fref{universallawkgeq2}, which are consequences of \fref{firstcontrol}.\\
{\it Proof of \fref{universallawkgeq}}: {\it For $k\geq 2$} let $G,\mathcal I, \tilde{c}_k$ be given  by \fref{deffb}, \fref {defiun}, \fref{valuecp} and $$\mathcal J=G+\mathcal I.$$ From \fref{estkeyfb}, \fref{estimateione}, \fref{firstcontrol} we have that:
\be
\label{estimationun}
\mathcal J(b)=b|\Lambda Q|_{L^2}^2+o(b) \ \ \mbox{and} \ \ J_s+\tilde{c}_kb^{2k}=o(b^{2k}).
\ee
In particular, this yields: $$\mathcal J_s+\tilde{c}_k\left(\frac{\mathcal J}{|\Lambda Q|_{L^2}^2}\right)^{2k}=o(\mathcal J^{2k}).$$ Dividing by $\mathcal J^{2k}$, which is strictly positive by \fref{estimationun}, \fref{controllambdaboot}, and integrating in $s$ yields:$$\frac{1}{(2k-1)\mathcal J^{2k-1}(s)}=\frac{1}{(2k-1)\mathcal J^{2k-1}(s_0)}+\frac{\tilde{c}_k}{|\Lambda Q|_{L^2}^{4k}}s+o(s).$$ Together with \fref{estimationun}, this provides the asymptotics:
\be
\label{lawforb}
b(s)=\left(\frac{|\Lambda Q|^2_{L^2}}{(2k-1)\tilde{c}_ks}\right)^{\frac{1}{2k-1}}(1+o(1)) \ \ \mbox{as} \ \ s\to +\infty.
\ee 
We now integrate the law for the scaling parameter $-\lsl=b$ to obtain:
$$-\log \lambda(s)=\frac{2k-1}{2k-2}\left(\frac{|\Lambda Q|^2_{L^2}}{(2k-1)\tilde{c}_k}\right)^{\frac{1}{2k-1}}s^{\frac{2k-2}{2k-1}}(1+o(1)) \ \ \mbox{as} \ \ s\to +\infty.$$ In particular, taking into account \fref{lawforb}:
$$b=\frac{d_k}{|\log \lambda|^{\frac{1}{2k-2}}}(1+o(1))\ \ \mbox{with} \ \ d_k=\left(\frac{|\Lambda Q|_{L^2}^2}{(2k-2)\tilde{c}_k}\right)^{\frac{1}{2k-2}}.$$ As a result $\lambda$ satisfies the following differential equation: 
\be
\label{estlawbinterm}
-\lambda_t=b=\frac{d_k}{|\log \lambda|^{\frac{1}{2k-2}}}(1+o(1))  \ \ \mbox{with} \ \ \lambda(t)\to 0 \ \ \mbox{as} \ \ t\to T.
\ee 
Integrating this in time yields: $$\lambda(t)=\frac{d_k(T-t)}{|\log (T-t)|^{\frac{1}{2k-2}}}(1+o(1)).$$ This gives \fref{universallawkgeq}.\\
{\it Proof of \fref{universallawkgeq2}}: Let $k=1$, then \fref{estkeyfb}, \fref{estimateione}, \fref{firstcontrol} imply: 
\be
\label{estimationunbis}
\mathcal J(b)=4b|\log b|+O(b) \ \ \mbox{and} \ \ \mathcal J_s+\frac{\mathcal J ^{2}}{8|\log (\mathcal J/|\log {\mathcal J}|)|^2}=O(\frac{\mathcal J^{2}}{|\log \mathcal J|^3}).
\ee
Let 
$$
4\beta=\frac{\mathcal J}{|\log \mathcal J|}-\frac{\mathcal J}{|\log \mathcal J|^2}\log|\log\mathcal J|,\qquad
\log\beta=\log{\mathcal J}-\log|\log \mathcal J|+O(1)
$$
so that 
\begin{align}\notag
4\beta &=\frac{4b|\log b|+O(b)}{|\log b+\log|\log b|+O(1)|}-\frac{4b|\log b|+O(b)}{(\log b+\log|\log b|+O(1))^2}(\log|\log b|+
O(\frac{\log|\log b|}{|\log b|}))\\ &=4b+O(\frac b{|\log b|})\label{eq:beta-b}
\end{align}
We compute
\begin{align*}
&4\beta_s=\frac{\mathcal J_s}{|\log \mathcal J|}(1-\frac {\log|\log\mathcal J|}{|\log\mathcal J|}) + \mathcal J_s O\left(\frac 1{|\log\mathcal J|^2}\right),\\&
\frac {16\beta^2}{|\log\beta|}=\frac {\mathcal J^2}{|\log\mathcal J|^2|\log(\mathcal J/|\log \mathcal J|)|}-
\frac {2\mathcal J^2\log|\log\mathcal J|}{|\log\mathcal J|^3|\log(\mathcal J/|\log \mathcal J|)|}+O(\frac{\mathcal J^{2}}{|\log \mathcal J|^4})
 \end{align*}
and therefore
\begin{align*}
4\beta_s+\frac {2\beta^2}{|\log\beta|^2}&=
-\frac{\mathcal J ^{2}}{8|\log \mathcal J|\,|\log (\mathcal J/|\log {\mathcal J}|)|^2}(1-\frac {\log|\log\mathcal J|}{|\log\mathcal J|})+\frac {\mathcal J^2}{8|\log\mathcal J|^2|\log(\mathcal J/|\log \mathcal J|)|}\\ &-
\frac {\mathcal J^2\log|\log\mathcal J|}{4|\log\mathcal J|^3|\log(\mathcal J/|\log \mathcal J|)|}+
O(\frac {\beta^2}{\log\beta^2})\\&= -\frac{\mathcal J ^{2}}{8|\log \mathcal J|^3}(1-3\frac {\log|\log\mathcal J|}{|\log\mathcal J|})\\&+\frac{\mathcal J ^{2}}{8|\log \mathcal J|^3}(1-3\frac {\log|\log\mathcal J|}{|\log\mathcal J|})+O(\frac {\beta^2}{\log\beta^2})=O(\frac {\beta^2}{\log\beta^2})
\end{align*}
\renewcommand\b{\beta}
To solve the problem 
$$\b_s=-\frac{\b^2}{2|\log \b|}+O(\frac{\b^2}{|\log \b|^2})$$ we multiply by 
$\frac{|\log \b|}{\b^2}$ so that $$\frac{\b_s\log \b}{\b^2}=\frac{1}{2}+O(\frac{1}{|\log \b|}).$$ Now $$(\frac{\log u}{u}+\frac{1}{u})'=-\frac{\log u}{u^2}$$ and thus $$-\frac{\log \b+1}{\b}=\frac{s}{2}+O\left(\int_0^s\frac{d\tau}{|\log \b|}\right).$$ To leading order, this leads to: 
$$\b=\frac{2\log s}{s}(1+o(1)),\qquad\log\b=\log\log s-\log s+O(1)$$ 
from which 
\be
\label{firstlaw}
\frac{-\log \b}{\b}=\frac{s}{2}\left(1+O\left(\frac{1}{\log s}\right)\right), \qquad \b=\frac{-2\log \b}{s}\left(1+O\left(\frac{1}{\log s}\right)\right).
\ee 
Therefore,
$$
\b
=  \frac{2\log s}{s}-2\frac{\log\log s}{s}+O\left(\frac{1}{s}\right).
$$
Using \eqref{eq:beta-b} we also conclude that 
\be\label{ceieiy}
b
=  \frac{2\log s}{s}-2\frac{\log\log s}{s}+O\left(\frac{1}{s}\right).
\ee
We now integrate the law for $\lambda$: $$-\lsl=b=\frac{2\log s}{s}-2\frac{\log\log s}{s}+O\left(\frac{1}{s}\right)$$ resulting
in 
$$-\log(\lambda)=(\log s)^2-2(\log s) \log\log s+O(\log s)=(\log s)^2\left(1-2\frac{\log \log s}{\log s}+O\left(\frac{1}{\log s}\right)\right)$$ which implies: 
\be
\label{lawfors}
\sqrt{-\log \lambda}=\log s\left(1-\frac{\log \log s}{\log s}+O\left(\frac{1}{\log s}\right)\right)=\log s-\log\log s+O(1).
\ee 
and thus
\be\label{eq:lambdas}
e^{\sqrt{-\log \lambda}+O(1)}=\frac s{\log s},\qquad s= \sqrt{-\log \lambda} e^{\sqrt{-\log \lambda}+O(1)}
\ee
We now observe from (\ref{ceieiy}):   
\be
\label{choegoege}
\sqrt{-\log \lambda}=\frac{b s}{2}+O(1)=-\frac{\lambda_t}{2}s+O(1)
\ee
 and thus $$-\frac{\lambda_t}{\sqrt{-\log \lambda}}s=2+o(1) .$$  
 Taking into account \eqref{eq:lambdas}
 gives the differential equation for $\lambda$: 
 \be
 \label{eq:del}
 -\lambda_te^{\sqrt{|\log\lambda|}+O(1)}=2+o(1) \ \ \mbox{and equivalently} \ \  -\lambda_te^{\sqrt{|\log\lambda|}}=e^{O(1)}.
 \ee Integrating this in time gives: 
 \be\label{eq:lam}
 \lambda(t)=(T-t)e^{-\sqrt{|\log (T-t)|}+O(1)}.
 \ee

It remains to prove the strong convergence of the excess of energy \fref{convustarb} which easily implies the quantization of the focused energy \fref{qunitoief}.\\

{\bf step 2} Sharp derivation of the $b$ law.\\
 
 Let us start with the following slightly  different control on $b$:
\be
\label{refinedlaw}
b(t)=\frac{\lambda(t)}{T-t}(1+o(1)) \ \ \mbox{as} \ \ t\to T.
\ee
For $k\geq 2$,  this follows directly from  \fref{universallawkgeq}, \fref{estlawbinterm}. We need to be more careful for $k=1$. Indeed, \eqref{eq:del} and \fref{eq:lam} imply:
\be
\label{estbkequalun}
b(t)=O(1)e^{-\sqrt{|\log(T-t)}},
\ee
but this together with \fref{eq:lam} is not sufficient to yield \fref{refinedlaw}. However, we compute:
\bee
\int_t^Tb^2 &= & \int_t^T-b\lambda_t=b(t)\lambda(t)+\int_t^T\lambda b_t\\
& =& b(t)\lambda(t)+\int_t^Tb_s= b(t)\lambda(t)+o\left(\int_t^Tb^2\right)
\eee
where we used \fref{poitwisebsboot} in the last step. Hence: 
\be
\label{estzero}
\frac{1}{b(t)\lambda(t)}\int_t^Tb^2=1+o(1) \ \ \mbox{as} \ \ t\to T.
\ee
On the other hand, 
\bea
\label{etscniocnonc}
\nonumber \left|\frac{1}{(T-t)b^2(t)}\int_t^Tb^2-1\right|& = & \frac{2}{(T-t)b^2(t)}\left|\int_t^Tbb_t(T-\tau)\right|\\
& \lesssim & 
\frac{1}{(T-t)b^2(t)}\int_t^T\frac{b^2}{|\log b|}\frac{b(T-\tau)}{\lambda(\tau)}d\tau.
\eea
We now observe from \fref{estimationunbis} that $$\forall \tau \in [t,T), \ \ \frac{b^2(\tau)}{|\log b(\tau)|}\leq 2 \frac{b^2(t)}{|\log b(t)|}$$ and hence \fref{etscniocnonc} yields the bound:
\be
\label{cbeooevniohvro}
\left|\frac{1}{(T-t)b^2(t)}\int_t^Tb^2-1\right|\lesssim \frac{1}{(T-t)|\log b(t)|}\int_t^T\frac{b(T-\tau)}{\lambda(\tau)}d\tau.
\ee
We now claim 
\be
\label{voecheohoehe}
\frac{1}{(T-t)|\log b(t)|}\int_t^T\frac{b(T-\tau)}{\lambda(\tau)}d\tau=o(1) \ \ \mbox{as} \ \ t\to T.
\ee
Assume \fref{voecheohoehe}, then \fref{estzero} and \fref{cbeooevniohvro} yield
$$\int_t^Tb^2=b\lambda(1+o(1))=(T-t)b^2(1+o(1))$$ which implies \fref{refinedlaw}.\\
{\it Proof of \fref{voecheohoehe}}: We compute:
\be
\label{xbzc}
\int_t^T\frac{b(T-\tau)}{\lambda(\tau)}d\tau =  -\int_t^T\frac{\lambda_t(T-\tau)}{\lambda(\tau)}d\tau
=(T-t)\log\lambda(t)-\int_t^T\log \lambda d\tau.
\ee
We now substitute \fref{universallawkgeq2} which implies $$\log \lambda(t)=\log (T-t)-\sqrt{|\log(T-t)|}+O(1)$$ and derive from \fref{xbzc} after some explicit integration by parts:
$$
\int_t^T\frac{b(T-\tau)}{\lambda(\tau)}d\tau=O((T-t)) \ \ \mbox{as} \ \ t\to T.
$$
We hence conclude from \fref{estbkequalun} that:
$$\frac{1}{(T-t)|\log b(t)|}\int_t^T\frac{b(T-\tau)}{\lambda(\tau)}d\tau=o\left(\frac{1}{|\log b(t)|}\right)=o(1) \ \ \mbox{as} \ \ t\to T,
$$
and \fref{voecheohoehe} is proved.\\

{\bf step 3} Strong convergence of $(w,\partial_tw)$ in $H$.\\

We are now in position to conclude the proof of \fref{convustarb} which is a consequence of the sharp asymptotics \fref{universallawkgeq}, \fref{universallawkgeq2} and \fref{refinedlaw} and the control of the excess of energy \fref{poitnwiseboundWboot}.\\
Statement \fref{convustarb} is equivalent to the existence of the strong limit for $(w(t),\partial_tw(t))$ in $\mathcal H$ as $t\to T$.\\
Let  $\zeta$ be a cut-off function with $\zeta(r)=0$ for $r\leq 1$ and $\zeta(r)=1$ for $r\geq 2$ and 
let $\zeta_R(r)=\zeta(Rr)$. The non-concentration of energy of the full solution $u$ outside the origin is well known and follows by a simple domain
of dependence argument combined with the results in \cite{ST}.  Therefore, using the decomposition \fref{defet} 
we obtain existence of $u^*,g^*$ such that 
\be
\label{estcoljeorij}
\forall R>0, \ \ \|\zeta_R(w(t)-u^*), \zeta_R(\partial_tw-g^*)\|_{\mathcal H}\to 0 \ \ \mbox{as} \ \ t\to T.
\ee 
The proof of the strong convergence \fref{convustarb} is now equivalent to the non-concentration of the energy for $w$ or equivalently: 
\be
\label{limitieequa}
E(u^*,g^*)=\lim_{t\to T}E(w(t),\partial_t w(t)).
\ee
{\it Proof of \fref{limitieequa}}: We adapt the argument from \cite{MR5}. For $t\in[0.T)$ define $$R(t)=B_1(t)\lambda(t).$$ and $$E_R(u,v)=\int\zeta_R \left[v^2+(\partial_ru)^2+k^2\frac{g^2(u)}{r^2}\right].$$ 
Integrating by parts using the equation \fref{equation}, we compute:
$$ \left|\frac{d}{d\tau}E_{R(t)}(u(\tau),\partial_tu(\tau))\right|\lesssim  \frac{1}{R(t)} \int_{R(r)\leq r\leq 2R(t)}\left[(\partial_t u)^2+(\partial_ru)^2+k^2\frac{g^2(u)}{r^2}\right]\lesssim \frac{1}{R(t)},$$
where in the last step we used conservation of energy. Integrating this from $t$ to $T$ using \fref{estcoljeorij} yields: 
\be
\label{coschinlhkdl}
\left|E_{R(t)}(u^*,g^*)-E_{R(t)}(u(t),\partial_tu(t))\right|  \lesssim  \frac{T-t}{R(t)}=\frac{T-t}{\lambda(t)B_1(t)}.
\ee
We now observe from \fref{defbnot}, \fref{refinedlaw} that:
$$ \frac{T-t}{\lambda(t)B_1(t)}=\frac{b(t)(T-t)}{\lambda(t)}\frac{1}{b(t)B_1(t)}\to 0 \ \ \mbox{as} \ \ t\to T.$$
Letting $t\to T$ in \fref{coschinlhkdl}, we conclude:
$$
E_{R(t)}(u(t),\partial_tu(t))\to E(u^*,g^*) \ \ \mbox{as} \ \ t\to T.
$$
\fref{limitieequa} now follows from:
\be
\label{choeooeheoheoc}
E_{R(t)}(u(t),\partial_tu(t))-E(w(t),\partial_tw(t))\to 0  \ \ \mbox{as}\ \ t\to T
\ee
Indeed, observe that:
$$\left|E_{R(t)}(u(t),\partial_tu(t))-E(w(t),\partial_tw(t))\right|  \lesssim  \int_{R(t)\leq r\leq 2R(t)}\left[(\partial_tw)^2+(\partial_rw)^2+k^2\frac{g^2(w)}{r^2}\right].
$$
For the first term, we have from \fref{controldt}, \fref{poitnwiseboundW}:
\bea
\label{vhihoev}
 \int_{r\leq 2R(t)}(\partial_tw)^2\lesssim R^2(t)\int\frac{(\partial_tw)^2}{r^2}\lesssim B_1^2(t)\mathcal E(t)\lesssim \frac{|\log b|^4}{b^2}b^4\to 0 \ \ \mbox{as} \ \ t\to T.
\eea
Similarily, from \fref{estdeux}:
\bea
\label{cnekonoene}
\nonumber \int_{2 r\leq 2R(t)}\left[(\partial_rw)^2+\frac{g^2(w)}{r^2}\right]& \lesssim & R^2|\log b|^2\int_{r\leq 2R(t)}(\nabla W)^2\\
& \lesssim & \frac{|\log b|^4}{b^2}\mathcal E(t)\to 0 \ \ \mbox{as} \ \ t\to T.
\eea
This concludes the proof of \fref{choeooeheoheoc} and \fref{limitieequa}.\\

{\bf step 2} Proof of the quantization of the blow up energy \fref{qunitoief}.\\

From the conservation of the Hamiltonian: $$E_0=E\left((P_{B-1})_{\lambda}+w, \partial_t\left[(P_{B_1})_{\lambda}+w\right]\right).$$ We develop this identity. The construction of $P_B$ implies from direct check $$E\left((P_{B_1})_{\lambda}, \partial_t\left[(P_{B_1})_{\lambda}\right]\right)\to E(Q) \ \ \mbox{as} \ \ t\to T$$ and the crossed term is easily proved to converge to zero using \fref{vhihoev}, \fref{cnekonoene} and the space localization of $P_{B_1}$.

 \fref{limitieequa} now yields \fref{qunitoief}.\\

This concludes the proof of Theorem \ref{mainthm}.


\begin{appendix}

\section{Inversion of $H$}


We formulate the following lemma about solutions of the inhomogeneous problem $Hv=h$ with 
the linear operator
$$H=-\Delta +k^2\frac{f'(Q)}{y^2}$$ 
associated to $Q$. Hamiltonian $H$ is a standard Schr\"odinger operator with the kernel generated by the $\dot{H}^1$ scaling invariance: $$\mbox{Ker}(H)=\mbox{span}(\Lambda Q),$$ see \cite{RS} for a further introduction to the spectral structure of $H$. The following Lemma is elementary but crucial for the construction of $Q_b$:

\begin{lemma}[Inversion of $H$]
\label{inversionl}
For $k\ge 4$ let $1\leq j\leq \frac{k}{2}-1$ and let $h_j(y)$ be a smooth function with 
\be
\label{cancellation}
(h_j,\Lambda Q)=0.
\ee
and the following asymptotics:
\be
\label{asymptotf}
h_j(y)=\left\{ \begin{array}{ll}
	 y^k(e_j+O(y^2)) \ \ \mbox{as} \ \ y\to 0\\
	 d_j\frac{y^{2j}}{y^{k}}(1+\frac{f_j}{y^2}+O(\frac 1{y^{3}}))\ \ \mbox{as} \ \ y\to +\infty,\\
	 \end{array}
	 \right .
\ee
Then there exists a smooth solution $Hv_{j+1}=h_j$ with 
\be
\label{scalruj}
(v_{j+1},\chi_{M}\Lambda Q)=0
\ee
and the following asymptotics:\\
(i) for $j+1<\frac{k}{2}$, for $0\leq m\leq 2$,
\be
\label{asymptotuj}
\frac{d^mv_{j+1}}{dy^m}(y)=\left\{ \begin{array}{ll}
	 y^{k-m}(\alpha_{j+1,m}+O(y^{2})) \ \ \mbox{as} \ \ y\to 0,\\
	 \beta_{j+1}\frac{d^m y^{2(j+1)-k}}{dy^m}\left[1+\frac{\gamma_{j+1}}{y^2}+O(\frac{1}{y^{3}})\right]\ \ \mbox{as} \ \ y\to +\infty,\\
	 \end{array}
	 \right .
\ee
(ii) for $j+1=\frac{k}{2}$ with $k$ even: 
\be
\label{asymptotujkeven}
v_{j+1}(y)=\left\{ \begin{array}{ll}
	 y^k(\alpha_{j+1}+O(y^{2}))\ \ \mbox{as} \ \ y\to 0,\\
	\beta_{j+1} \left[1+\frac{\gamma_{j+1}}{y^2}+O(\frac{1}{y^{3}})\right]\ \ \mbox{as} \ \ y\to +\infty,\\
	 \end{array}
	 \right .
\ee
For $1\le m\le 2$
\be
\label{asymptotujkevenderivative}
\frac {d^mv_{j+1}(y)}{dy^m}=\left\{ \begin{array}{ll}
	 y^{k-m}(\alpha_{j+1,m}+O(y^2))\ \ \mbox{as} \ \ y\to 0,\\
\beta_{j+1}\gamma_{j+1}\frac{d^my^{-2}}{dy^m}+O(\frac{1}{y^{3+m}})\ \ \mbox{as} \ \ y\to +\infty,\\
	 \end{array}
	 \right .
\ee
Moreover, if
 \be
\label{asymptotf'}
h'_j(y)=\left\{ \begin{array}{ll}
	 ky^{k-1}(e_j+O(y^2)) \ \ \mbox{as} \ \ y\to 0\\
	 d_j(2j-k)\frac{y^{2j-1}}{y^{k}}(1+\frac{f_{j}}{y^2}+O(\frac 1{y^{3}}))\ \ \mbox{as} \ \ y\to +\infty,\\
	 \end{array}
	 \right .
\ee
then \eqref{asymptotuj}, \eqref{asymptotujkevenderivative} hold for $m=3$. The constants $\alpha_{j+1},
\alpha_{j+1,m},\gamma_{j+1}$ implicitly depend on $d_j,e_j$ and $\beta_{j+1}$ can be found from the relation:
\be
\label{recurrencedj}
\beta_{j+1}=-\frac{d_j}{4(j+1)(k-(j+1))}.
\ee
\end{lemma}

\vskip 1pc

\begin{proof} The proof relies on the accessibility  of the explicit expression for 
the Green's function of $H$. \\

{\bf step 1} Solving the linear equation.\\

From \eqref{eq:WQ} in the Wave Map case $Q$ has the following asymptotics
\be
\label{expansionqr}
Q(y)=\left\{\begin{array}{ll} 
		2y^k(1+O(y^{k})) \ \ \mbox{as}\ \ y\to 0,\\
		\pi-\frac{2}{y^k}(1+O(\frac{1}{y^{k}}))\ \ \mbox{as}\ \ y\to \infty.
		\end{array}
		\right .
\ee
and:
\be
\label{ezpansionLQ}
J=\Lambda Q=\left\{\begin{array}{ll} 
		2ky^k(1+O(y^{k}))\ \ \mbox{as}\ \ y\to 0,\\
		\frac{2k}{y^{k}}(1+O(\frac{1}{y^{k}}))\ \ \mbox{as}\ \ y\to \infty,	
			\end{array}
		\right .
\ee
Similarly, in the (YM) case ($k=2$, not covered by the Lemma) we find 
\be
\label{expansionqr'}
Q(y)=\left\{\begin{array}{ll} 
		(1+O(y^{k})) \ \ \mbox{as}\ \ y\to 0,\\
		(-1+O(\frac{1}{y^{k}}))\ \ \mbox{as}\ \ y\to \infty.
		\end{array}
		\right .
\ee
and:
\be
\label{ezpansionLQ'}
J=\Lambda Q=\left\{\begin{array}{ll} 
		-2ky^k(1+O(y^{k}))\ \ \mbox{as}\ \ y\to 0,\\
		-\frac{2k}{y^{k}}(1+O(\frac{1}{y^{k}}))\ \ \mbox{as}\ \ y\to \infty,	
			\end{array}
		\right .
\ee
Let now $$\G(y)=J(y)\int_{1}^y\frac{dx}{xJ^2(x)}$$ 

be the other (singular) element of the kernel of $H$, which can be found from the Wronskian relation: 
\be
\label{wronskian}
\G'J-\G J'=\frac{1}{y}.
\ee
From this we can easily find the asymptotics of $\G$: 
\be
\label{ezpansionLQH}
\G(y)=\left\{\begin{array}{ll} 
		-\frac{1}{4k^2y^k}(1+O(y^{k}))\ \ \mbox{as}\ \ y\to 0,\\
		\frac{y^k}{4k^2}(1+O(\frac{1}{y^{k}}))\ \ \mbox{as}\ \ y\to \infty,
			\end{array}
		\right .
\ee
in the (WM) case. In the (YM) case
\be
\label{ezpansionLQH'}
\G(y)=\left\{\begin{array}{ll} 
		\frac{1}{4k^2y^k}(1+O(y^{k}))\ \ \mbox{as}\ \ y\to 0,\\
		-\frac{y^k}{4k^2}(1+O(\frac{1}{y^{k}}))\ \ \mbox{as}\ \ y\to \infty,
			\end{array}
		\right .
\ee
 Using the method of variation of parameters and 
 (\ref{wronskian}), we find that a solution to $Hw_{j+1}=h_j$ is given by: 
 \be
 \label{defulinearsolver}
 w_{j+1}(y)=\G(y)\int_0^yh_j(x)J(x)xdx-J(y)\int_{1}^yh_j(x)\G(x)xdx.
 \ee
 
 {\bf Step 2} Asymptotics of $w_{j+1}$.\\
 
 We compute the asymptotics of $w_{j+1}$ near $+\infty$. In what follows we restrict our analysis to the
 (WM) case. For the first term in (\ref{defulinearsolver}), 
 we use \fref{cancellation}, (\ref{asymptotf}) to derive:
 \bee
 \G(y)\int_0^yh_j(x)J(x)xdx & = & -\G(y)\int_y^{+\infty}h_j(x)J(x)xdx\\
 & = & -\frac{y^k}{2k^2}\left(1+O(\frac{1}{y^{k}})\right)\int_y^{+\infty}x\frac{k}{x^k}\frac{d_jx^{2j}}{x^k}\left(1+\frac{f_{j}}{x^2}+O(\frac{1}{x^{3}})\right)dx\\
 & =  & -\frac{d_jy^k}{2k}\left(1+O(\frac{1}{y^{k}})\right)\int_y^{+\infty}\frac{x^{2j+1}}{x^{2k}}\left(1+\frac{f_{j}}{x^2}+O(\frac{1}{x^{3}})\right)dx\\
 & = & -\frac{d_j}{4k(k-(j+1))}\frac{y^{2(j+1)}}{y^k}\left(1+\frac{f^{(1)}_{j+1}}{y^2}+O(\frac{1}{y^{3}})\right)
 \eee
 In the above $f^{(1)}_{j+1}$ is a constant dependent only on $f_j, k$ and $j$.
 
For the second term, we estimate
 \bee
 -J(y)\int_{1}^yh_j(x)\G(x)xdx & = & -\frac{k}{y^k}\left(1+O(\frac{1}{y^{k}})\right)\int_1^{y}\frac{xx^k}{2k^2}\frac{d_jx^{2j}}{x^k}\left(1+\frac{f_{j}}{x^2}+O(\frac{1}{x^{3}})\right)dx\\
 & = & -\frac{d_j}{2ky^k}\left(1+O(\frac{1}{y^{k}})\right)\int_1^{y}x^{2j+1}\left(1+\frac{f_{j}}{x^2}+O(\frac{1}{x^{3}})\right)dx.
 \eee
 and (\ref{asymptotuj}), (\ref{asymptotujkeven}) and \fref{recurrencedj} follow for $y\to +\infty$.\\
 We compute the asymptotics of $v_{j+1}$ near the origin. First, 
 \bee
 \G(y)\int_0^yh_j(x)J(x)xdx & = & -\frac{1}{2k^2y^k}(1+O(y^{k}))\int_0^y x e_jx^kkx^k(1+O(x^{2}))dx\\
 & = & y^k (O(y^{2}))
 \eee
 For the second term in (\ref{defulinearsolver}), 
 \bee
 -J(y)\int_{1}^yh_j(x)\G(x)xdx & = & ky^k(1+O(y^{k}))\int_1^ye_jx^kx\frac{1}{2k^2x^k}(1+O(x^{2}))dx\\
 & = & \frac{e_j}{2k}y^k \left[-\int_0^1(x+O(x^2))dx+O(y^{2})\right]
 \eee
 and (\ref{asymptotuj}) and (\ref{asymptotujkeven}) follow for $v_{j+1}$ as $y\to 0$.\\
 
 {\bf step 3} Estimates for the derivatives.\\
 
For $2j<k-2$, the estimates for the derivatives (\ref{asymptotuj}) are derived similarily and left to the reader. For $k$ even and $j=\frac{k}{2}-1$, there holds an extra cancellation as $y\to +\infty$ leading to (\ref{asymptotujkevenderivative}) which we now exploit. Indeed, 
$$
w'_{j+1}(y)=-\G'(y)\int_{y}^{+\infty}h_j(x)J(x) xdx-J'(y)\int_1^yh_j(x)\G(x) xdx.$$ For the first term, 
\bee
-\G'(y)\int_{y}^{+\infty}h_j(x)J(x) xdx & = & -\frac{ky^{k-1}}{2k^2}\left(1+O(\frac{1}{y^{k}})\right)\int_y^{+\infty}\frac{kd_j x^{2j+1}}{x^{2k}}\left(1+\frac{f_{j}}{x^2}+O(\frac{1}{x^{3}})\right)dx\\
& = & -\frac{d_j}{2ky}\left(1+\frac{f^{(2)}_{j+1}}{y^2}+O(\frac{1}{y^{3}})\right).
\eee
Similarly, 
\bee
 -J'(y)\int_{1}^yh_j(x)H(x)xdx & = & \frac{k^2}{y^{k+1}}\left(1+O(\frac{1}{y^{k}})\right)\int_1^{y}\frac{d_jx^{2j+1}}{2k^2}\left(1+O(\frac{1}{x^{3}})\right)dx\\
 & = & \frac{d_j}{2ky}\left(1+\frac{f^{(3)}_{j+1}}{y^2}+O(\frac{1}{y^{3}})\right),
 \eee
resulting in the cancellation leading to (\ref{asymptotujkevenderivative}). 
The constants $f^{(2)}_{j+1}$,  $f^{(3)}_{j+1}$ depend only on $f_j, k$ and $j$.

The second derivative $w_{j+1}''$ is estimated using the equation and the asymptotics for $(w_{j+1},w_{j+1}')$, this is left to the reader.\\

 {\bf step 4} Satisfying the orthogonality condition.\\
 
 We now let $$v_{j+1}=w_{j+1}-\frac{(w_{j+1},\chi_{M}\Lambda Q)}{(\Lambda Q,\chi_{M}\Lambda Q)}\Lambda Q$$  so that (\ref{scalruj}) is satisfied. Moreover, $L(\Lambda Q)=0$ implies $Lv_{j+1}=Lw_{j+1}=f_j$. It now remains to observe  from (\ref{ezpansionLQ}) that the behavior of $v_{j+1}$ near the origin and $+\infty$ is the same as
 of $w_{j+1}$.

This concludes the proof of Lemma \ref{inversionl}.
\end{proof}



\section{Some linear estimates}


\begin{lemma}[Logarithmic Hardy inequalities]
\label{lemmaloghrdy}
$\forall R>2$, $\forall v\in\dot{H}^1_{rad}(\RR^2)$, there holds the following controls:
\be
\label{harfylog}
 \int_{y\le R} \frac{|v|^2}{y^2(1+|\log y|)^2}ydy\lesssim  \int_{y\le R} |\nabla v|^2,
\ee
\be
\label{estmedium}
|v|^2_{L^{\infty}(1\leq y\leq R)}\lesssim \int_{1\leq y\leq 2}|v|^2+R^2\int \frac{|\nabla v|^2}{y^2}ydy,
\ee
\be
\label{hardyone}
\int_{y\leq R}  |v|^2ydy\lesssim R^2\left(\int_{y\leq 2}|v|^2ydy+\log R\int_{y\leq R}|\nabla v|^2ydy\right),
\ee
\be
\label{harybis}
\int_{R\leq y\leq2 R}  \frac{|v|^2}{y^2}ydy\lesssim \int_{y\leq 2}|v|^2ydy+\log R\int_{y\leq 2R}|\nabla v|^2ydy.
\ee
\be
\label{harybis-non}
\int_{y\leq2 R}  \frac{|v|^2}{y^2}ydy\lesssim \log R \int_{y\leq 2}|v|^2ydy+\log^2 R\int_{y\leq 2R}|\nabla v|^2ydy.
\ee
\end{lemma}

\begin{proof} Let $v$ smooth. To prove \fref{harfylog}, let $f(y)=-\frac{{\bf e}_y}{y(1+|\log (y)|)}$ so that $\nabla \cdot f=\frac{1}{y^2(1+|\log y|)^2}$, and integrate by parts to get: 
\bea
\label{stepwzofp}
\nonumber & & \int_{\e\le y\le R} \frac{|v|^2}{y^2(1+|\log y|)^2}ydy  =  \int_{\e\leq y\le R} |v|^2\nabla \cdot f ydy\\
\nonumber & = &- \left[\frac{|v|^2}{1+|\log (y)|}\right]_{\e}^R +2\int_{y\le R} v\partial_y v \frac{1}{y(1+|\log y|)}ydy\\
& \lesssim &  \frac{|v(\e)|^2}{1+|\log \e|}+\left(\int_{y\le R} \frac{|v|^2}{y^2(1+|\log y|)^2}ydy\right)^{\frac{1}{2}}\left(\int_{y\le R} |\nabla v|^2ydy\right)^{\frac{1}{2}}.
\eea
On the other hand, $$|v(\e)|^2\lesssim |v(1)|^2+\left(\int_{\e\leq y \leq 1}|v'(y)|dy\right)^2\lesssim |v(1)|+|\log \e|\int_{y\leq R}|\nabla v|^2ydy$$ and hence: $$\limsup_{\e\to 0} \frac{|v(\e)|^2}{1+|\log \e|}\lesssim \int_{y\leq R}|\nabla v|^2ydy.$$ Injecting this into \fref{stepwzofp} and letting $\e\to 0$ yields \fref{harfylog}. To prove \fref{estmedium}, let $y_0\in[1,2]$ such that $$|v(y_0)|^2\lesssim \int_{1\leq y\leq 2}|v|^2ydy.$$ Then: $\forall y\in [1,R]$, $$|v(y)|=|v(y_0)+\int_{y_0}^yv'(r)dr|\lesssim |v(y_0)|+R\left(\int \frac{|\nabla v|^2}{y^2}ydy\right)^{\frac{1}{2}},$$ and \fref{estmedium} follows. Similarily, $$|v(y)|=|v(y_0)+\int_{y_0}^yv'(r)dr|\lesssim |v(y_0)|+\left(\int_{y\leq R} |\nabla v|^2ydy\right)^{\frac{1}{2}}\sqrt{\log R},$$ and \fref{hardyone}, \fref{harybis} follow by squaring this estimate and integrating in $R$. Finally,
\eqref{harybis-non} follows from \fref{harybis} by summing over dyadic $R$-intervals.
\end{proof}

\begin{lemma}[Hardy type estimates with $A$]
\label{lemmahardy1}
Let $M\geq 1$ fixed. Then there exists $c(M)>0$ such that the following holds true. Let $u\in H^1$ with $$(u,\chi_M\Lambda Q)=0,$$ then:\\ 
(i) \be\label{eq:orb}
\int \left (|\nabla u|^2+\frac {|u|^2}{y^2}\right) \leq C(M) \int |Au|^2
\ee
(ii) if 
\be
\label{eoheogh}
\int \frac{|u|^2}{y^4}+\int\frac{|\nabla u|^2}{y^2}<+\infty;
\ee 
then:
\be
\label{estun}
\ \ \int \frac{|\nabla u|^2}{y^2}+\int \frac{|u|^2}{y^4}\leq c(M)\int\frac{|Au|^2}{y^2};
\ee
(iii) if 
\be
\label{ionoeghe}
\int \frac{|u|^2}{y^4(1+|\log y|)^2}+\int |\nabla (Au)|^2<+\infty,
\ee
then:
\bea
\label{estdeux}
& & \nonumber \int \frac{|\nabla u|^2}{y^2(1+|\log y|)^2}+ \int \frac{|u|^2}{y^4(1+|\log y|)^2} \\
\nonumber & \leq & c(M)\left[\int\frac{|Au|^2}{y^2(1+y^2)}+\int|\nabla (Au)|^2ydy\right]\\
& \lesssim & c(M)|A^*Au|_{L^2}^2.
\eea
\end{lemma}

\begin{remark}
\label{rkc} The norm \fref{eoheogh} is finite for $u=w$ for $k\geq 2$. For $k=1$, the finitness of the $\H^2$ norm implies that $$\frac{w}{y}, \nabla (Aw)\in \dot{H}^1$$ and hence the norm \fref{ionoeghe} is finite using \fref{harfylog}.
\end{remark}

\begin{proof} \fref{eq:orb} is equivalent to \fref{coercivityneergy} ie the coercitivity of the linearized energy. The proof of the global Hardy type inequality \fref{estun}, \fref{estdeux} with $c(M)$ follows as in Rodnianski-Sterbenz' \cite{RS} Appendix for $k\geq 3$. The cases $k=1,2$ require some more attention. We treat $k=1$ which is the most delicate case and leave $k=2$ to the reader.\\
We claim the key subcoercivity property:
\bea
\label{keyestnatappendix}
\nonumber& &  \int\frac{|Au|^2}{y^2(1+y^2)}+\int|\nabla (Au)|^2\\
 & \geq&  C\left[\int\frac{|\pa_y u|^2}{y^2(1+|\log y|)^2}+ \int \frac{|u|^2}{y^4(1+|\log y|)^2}-\int \frac{|u|^2}{1+y^5}\right].
\eea
Assume \fref{keyestnatappendix}, then \fref{estdeux} follows by contradication. Let $M>0$ fixed and consider a sequence $u_n$ such that 
\be
\label{normlaization}
\int\frac{|\pa_y u_n|^2}{y^2(1+|\log y|)^2}+ \int \frac{|u_n|^2}{y^4(1+|\log y|)^2}=1,  \ \ (u_n,\chi_M\Lambda Q)=0,
\ee and 
\be
\label{seeumporgpo}
  \int\frac{|Au_n|^2}{y^2(1+y^2)}+\int|\nabla (Au_n)|^2\leq \frac{1}{n},
  \ee
   then by semicontinuity of the norm, $u_n$ weakly converges on a subsequence to $u_{\infty}\in H^1_{loc}$ solution to $Au_{\infty}=0.$ $u_{\infty}$ is smooth away from the origin and hence the explicit integration of the ODE and the regularity assumption at the origin $u_{\infty}\in H^1_{loc}$ implies $$u_{\infty}=\alpha \Lambda Q.$$ On the other hand, from the uniform bound \fref{normlaization} together with the local compactness of Sobolev embeddings, we have up to a subsequence: $$\int \frac{|u_n|^2}{1+y^5}\to \int \frac{|u_{\infty}|^2}{1+y^5} \ \ \mbox{and} \ \ (u_n,\chi_M\Lambda Q)\to (u_{\infty},\chi_M\Lambda Q).$$ We thus conclude that  $$\alpha(\Lambda Q,\chi_M\Lambda Q)=(u_{\infty},\chi_M\Lambda Q)=0 \ \ \mbox{and thus} \ \ \alpha=0.$$ On the other hand, 
 from the subcoercitivity property \fref{keyestnatappendix} and \fref{normlaization}, \fref{seeumporgpo}
 $$\alpha^2\int \frac{|\Lambda Q|^2}{1+y^5}=\int \frac{|u_{\infty}|^2}{1+y^5}\geq C>0 \ \ \mbox{and thus} \ \ \alpha\neq 0.$$ A contradiction follows. Finally, the last step in \fref{estdeux} is a direct consequence of \fref{eq:coerc} ie the structure of the conjuguate Hamiltonian $\tilde{H}$.\\
 {\it Proof of \fref{keyestnatappendix}}: Let a smooth cut off function $\chi(y)=1$ for $y\leq 1$, $\chi(y)=0$ for $y\geq 2$, and consider the decomposition: $$u=u_1+u_2=\chi u+(1-\chi u).$$ Then from \fref{harfylog}:
\be
\label{fihoehyeog}
  \int\frac{|Au|^2}{y^2(1+y^2)}+\int|\nabla (Au)|^2\geq C\left[ \int \frac{|Au|^2}{y^2(1+y^2)}+\frac{|Au|^2}{y^2(1+|\log y|)^2}\right].
  \ee
   For the first term, we rewrite:
\bee
\int \frac{|Au|^2}{y^2(1+y^2)}& \geq& \int \frac{|Au_1|^2}{y^2}+2\int \frac{(Au_1)(Au_2)}{y^2(1+y^2)}\\
& \geq & C\left[\int \frac{|y\pa_y\left(\frac{u_1}{y}\right)|^2}{y^2}-\int\frac{|V^{(1)}-1|^2}{y^2}|u_1|^2-\int_{1\leq y \leq 2}|u|^2\right]
\eee
where in the last step we integrated by parts the quantity: $$(Au_1)(Au_2)=(\chi Au-\chi'u)((1-\chi)Au+\chi'u)\geq \chi (Au) \chi'u-\chi'u(1-\chi)(Au)-(\chi')^2u^2.$$ We hence conclude from $|V^{(1)}(y)-1|\lesssim y$ for $y\leq 1$ and the Hardy inequality \fref{harfylog} applied to $\frac{u_1}{ y}$ that: 
\be
\label{firstestimate}
\int \frac{|Au|^2}{y^2(1+y^2)}\geq C\left[\int\frac{|u_1|^2}{y^4(1+|\log y|)^2}-\int_{y\leq 2}|u|^2\right].
\ee
Similarily we estimate:
\bea
\label{estprlimemiddle}
\nonumber \int \frac{|Au|^2}{y^2(1+|\log y|)^2}& \geq& \int \frac{|Au_2|^2}{y^2(1+|\log y|)^2}+2\int \frac{(Au_1)(Au_2)}{y^2(1+|\log y|)^2}\\
\nonumber & \geq & C\left[\int \frac{1}{y^2(1+|\log y|)^2}|\pa_yu_2+\frac{u_2}{y}|^2-\int\frac{|V^{(1)}+1|^2}{y^2(1+|\log y|)^2}|u_2|^2-\int_{1\leq y \leq 2}|u|^2\right]\\
& \geq & C\left[\int \frac{|\pa_yu_2|^2}{y^2(1+|\log y|)^2}+\int\frac{|u_2|^2}{y^4(1+|\log y|)^2}-\int \frac{|u_2|^2}{y^6(1+|\log y|)^2}
\right]
\eea
where we integrated by parts for the last step and used the bound $|V^{(1)}(y)+1|\lesssim \frac{1}{y^2}$ for $y\geq 1$. \fref{fihoehyeog}, \fref{firstestimate} and \fref{estprlimemiddle} imply:
\be
\label{estprlimeojeoi}
\int\frac{|Au|^2}{y^2(1+y^2)}+\int|\nabla (Au)|^2\geq C\left[\int \frac{|u|^2}{y^4(1+|\log y|)^2}-\int \frac{|u|^2}{1+y^5}\right].
\ee
This implies using again \fref{harfylog}:
\bee
\int\frac{|\pa_y u|^2}{y^2(1+|\log y|)^2} & \lesssim & \int\frac{|A u|^2}{y^2(1+|\log y|)^2}+\int\frac{|u|^2}{y^4(1+|\log y|)^2}\\
& \lesssim & \int |\nabla (Au)|^2+\int\frac{|Au|^2}{y^2(1+y^2)}+\int \frac{|u|^2}{1+y^5}
\eee
which together with \fref{estprlimeojeoi} concludes the proof of \fref{keyestnatappendix}.\\
This concludes the proof of Lemma \ref{lemmahardy1}.
\end{proof}

\begin{lemma}[Control of the $\partial_t$ derivative]
\label{lemmahardy}
There holds:
\be
\label{controldt}
\int|\nabla \partial_t w|^2+\int\frac{|\partial_t w|^2}{r^2}\leq C(M)\left[ \int (\partial_t W)^2+\int |A_\lambda^*W|^2\right].
\ee
\end{lemma}

\begin{proof} We compute from \fref{defaloambinot}:
$$\partial_tW=A(\partial_tw)+\frac{\partial_t\vul w}{r}$$ and hence:
\be
\label{vhieoheo}
\int (A\partial_t w)^2\lesssim \int (\partial_tW)^2+\int (\frac{\partial_t\vul w}{r})^2.
\ee 
We now recall the following coercitivity property of the linearized Hamiltonian:
$$\int (A\partial_t w)^2\geq c(M)\left(\int|\nabla \partial_t w|^2+\int\frac{|\partial_t w|^2}{r^2}\right)-\frac{1}{c(M)\lambda ^4}(\partial_tw,(\chi_M\Lambda Q)_{\lambda})^2.$$ From the choice of orthogonality condition \fref{orthe}:
\bee
\left|(\partial_tw,(\chi_M\Lambda Q)_{\lambda})\right|& = & \left|(w,\partial_t((\chi_M\Lambda Q)_{\lambda}))\right|=\frac{b}{\lambda} \left|(w,(\Lambda(\chi_M\Lambda Q))_{\lambda})\right|\\
& \leq & c(M)b\lambda \left(\int_{y\leq 2M} |\e|^2\right)^{\frac{1}{2}}.
\eee
Combining this with \fref{vhieoheo} and the pointwise bound \fref{eq:V3} yields: 
\be
\label{hgigrghohg}
\int|\nabla \partial_t w|^2+\int\frac{|\partial_t w|^2}{r^2}\lesssim \int (\partial_tW)^2+\frac{b^2}{\lambda ^2}\int |\e|^2\left[{\bf 1}_{y\leq M} +\frac{y^{4}}{y^2(1+y^{8})}\right].
\ee
We then estimate from \fref{estdeux}:
\bee
\int |\e|^2\left[{\bf 1}_{y\leq M}+\frac{y^{4}}{y^2(1+y^{8})}\right] &\lesssim &\int \frac{|\e|^2}{y^4(1+|\log y|^2)}\\
\leq C(M)\int |A^*A\e|^2= \lambda ^2\int |A^*W|^2,
\eee
which together with \fref{hgigrghohg} concludes the proof of \fref{controldt}.
\end{proof}

\end{appendix}


\end{document}